\documentclass[number]{elsarticle}
\usepackage{enumitem}
\usepackage{subcaption}

\usepackage{amsmath}
\usepackage{amsthm}
\usepackage{mathtools}

\theoremstyle{definition}
\newtheorem{remark}{Remark}[section]
\newtheorem{example}[remark]{Example}
\newtheorem{definition}[remark]{Definition}
\theoremstyle{plain}
\newtheorem{theorem}[remark]{Theorem}
\newtheorem{proposition}[remark]{Proposition}
\newtheorem{lemma}[remark]{Lemma}
\newtheorem{corollary}[remark]{Corollary}

\usepackage{hyperref}
\hypersetup{colorlinks=true}

\let\digamma\relax

\usepackage{tikz}
\usepackage{varwidth}
\usepackage{pgfplots}
\pgfplotsset{compat=1.18}
\pgfkeys{/pgf/number format/1000 sep={}}
\usepackage{tikz-3dplot}
\usetikzlibrary{
	backgrounds,
	positioning,
	decorations.markings,
	decorations.pathmorphing,
	arrows.meta,bending,
	fit,
	knots,
	hobby,
	decorations.pathreplacing,
	shapes.geometric,
	calc,
	patterns
}
\usepackage{quiver}
\tikzcdset{scale cd/.style={every label/.append style={scale=#1}, cells={nodes={scale=#1}}}} 

\tikzset{
  curarrow/.style={
  rounded corners=8pt,
  execute at begin to={every node/.style={fill=red}},
    to path={-- ([xshift=-50pt]\tikztostart.center)
    |- (#1) node[fill=white] {$\scriptstyle \partial_*$}
    -| ([xshift=50pt]\tikztotarget.center)
    -- (\tikztotarget)}
    }
}

\tikzset{
 gedge/.style={
    thick,
    postaction={decorate,
    decoration={markings,
      mark=at position 0.59 with {\arrow{>}}}
      }
    }
}

\newcommand{\lineone}{
\tikz[xscale=.3, baseline=-.5ex]{
    \node (A) at (0,0) {\(\bullet\)};
    \node (B) at (2,0) {\(\bullet\)};
    \draw[gedge] (A.center) -- (B.center);
    }
}

\usepackage{manfnt}        
\usepackage{stmaryrd}      
\usepackage{ifthen}
\usepackage{xargs}
\usepackage{xcolor}

\newcommand{\FM}{Freyd–Mitchell}
\newcommand{\EZ}{Eilenberg–Zilber}
\newcommand{\MV}{Mayer–Vietoris}

\newcommand\look[1]{{\bfseries #1}}

\DeclareMathOperator{\Hom}{Hom}

\DeclareMathOperator{\Tor}{Tor}
\DeclareMathOperator{\End}{End}

\DeclareMathOperator{\id}{id}
\DeclareMathOperator{\Id}{Id}

\DeclareMathOperator{\Ob}{Ob}
\DeclareMathOperator*{\colimb}{colim}
\DeclareMathOperator{\colim}{colim}
\DeclareMathOperator{\Pro}{Pro}
\DeclareMathOperator{\Ind}{Ind}
\DeclareMathOperator{\swap}{swap}

\newcommand{\Repr}{\mathbb{R}\mathrm{epr}}

\newcommand{\D}{\mathbb{D}}

\newcommand{\cat}{\mathrm{Cat}}

\newcommand{\Cub}{\mathrm{Cub}}

\newcommand{\set}{\mathrm{Set}}
\newcommand{\ring}{\mathrm{Ring}}
\newcommand{\alg}{\mathrm{Alg}}
\newcommand{\mmod}{\mathrm{Mod}}
\newcommand{\DDD}{\mathcal{D}}
\newcommand{\Ch}{\mathrm{Ch}}
\newcommand{\coCh}{\mathrm{coCh}}

\newcommand{\Bimod}{{{}_R\mathbb{M}\mathrm{od}_R}}
\newcommand{\BimodZ}{{{}_{\Z}\mathbb{M}\mathrm{od}_{\Z}}}
\newcommand{\BimodP}{{{}_{\End(P)}\mathbb{M}\mathrm{od}_{\End(P)}}}
\newcommand{\BimodUI}{{{}_{\End(U_I)}\mathbb{M}\mathrm{od}_{\End(U_I)}}}

\newcommand{\N}{\mathbb{N}}
\newcommand{\Z}{\mathbb{Z}}
\newcommand{\R}{\mathbb{R}}
\newcommand{\CC}{\mathbb{C}}
\newcommand{\AAA}{\mathbb{A}}
\newcommand{\BBB}{\mathbb{B}}
\newcommand{\A}{\mathcal{A}}
\newcommand{\B}{\mathcal{B}}
\newcommand{\C}{\mathcal{C}}
\newcommand{\DD}{\mathcal{D}}

\newcommand{\Dist}{\mathbb{D}\mathrm{ist}}
\newcommand{\rmod}{R\text{-}\mathrm{Mod}}
\newcommand{\ct}{c}
\newcommand{\cct}{cc}
\newcommand{\cart}{cart}
\newcommand{\cocart}{cocart}
\newcommand{\act}{\cdot}

\newcommand{\tobar}{\mathrel{\mkern3mu\vcenter{\hbox{$\scriptscriptstyle+$}}%
                    \mkern-12mu{\to}}}
\newcommandx{\yaHelper}[2][1=\empty]{%
    \ifthenelse{\equal{#1}{\empty}}%
        {\ensuremath{\scriptstyle{#2}}}%
        {\raisebox{#1}[0pt][0pt]{\ensuremath{\scriptstyle{#2}}}}%
}
\newcommand{\cc}[2]{\mathrel{\smash{\xRightarrow[\protect{\yaHelper[1pt]{#1}}]{#2}}}}
\newcommand{\ccs}[2]{\cc{#1}{#2}\vphantom{\xRightarrow[\cdot]{\cdot}}}

\newcommand{\lus}[1]{\;\!\!{}_{#1}\:\!\!}
\newcommand{\lex}[1]{{}^{#1}\!}

\DeclarePairedDelimiter\fact{\lvert}{\rvert}
\DeclarePairedDelimiter\rea{\lvert}{\rvert}

\newcommand{\HM}{H}

\newtheorem*{def31}{Definition \ref{def:afb}}
\newtheorem*{def41}{Definition \ref{def:chaincomplex}}
\newtheorem*{def46}{Definition \ref{def:homology}}
\newtheorem{thm54}{Theorem \ref{th:exact_seq_rel_hom}}
\newtheorem*{thm57}{Theorem \ref{th:mayer_vietoris}}
\newtheorem*{thm59}{Theorem \ref{th:kunneth}}
\newtheorem*{def626}{Definition \ref{def:modulelike}}
\newtheorem*{thm645}{Theorem \ref{thm:embedding}}

\newcommand\K{\overrightarrow{K}}

\begin{document}

\begin{frontmatter}

\title{Homological Algebra in Abelian Framed Bicategories :  Exact Sequences and Embedding Theorems}

\author[1]{Augustin Albert\corref{cor1}}
\ead{augustin.albert@polytechnique.edu}

\author[1]{J\'er\'emy Dubut\fnref{fn2}}
\ead{jeremy.dubut@polytechnique.edu}

\author[1]{Eric Goubault}
\ead{eric.goubault@polytechnique.edu}

\cortext[cor1]{Corresponding author}

\fntext[fn2]{Partially funded by the Academic and Research Chair “Architecture des Systèmes Complexes” Dassault Aviation, Naval Group, Dassault Systèmes, KNDS France, Agence de l’Innovation de Défense, Institut Polytechnique de Paris}

\affiliation[1]{organization={
LIX, CNRS, École polytechnique, Institut Polytechnique de Paris}, city={Palaiseau}, 
country={France}
}

\begin{abstract}
We introduce abelian framed bicategories, which are particular framed bicategories that are locally abelian, and show that they are suitable for developing homology and cohomology theories for directed structures. This means in particular that similar exact sequences as the relative homology and \MV{} long exact sequences can be shown to hold. Also, for closed monoidal abelian framed bicategories, Künneth theorem holds as well. Finally, we prove embedding theorems similar to the Gabriel and \FM{} theorems, for particular abelian framed bicategories, allowing to see those as bicategories of bimodules over algebras. This naturally links to the original motivation of this work, which was to generalize directed homology developed in the abelian framed bicategory of bimodules over (path) algebras. 
\end{abstract}

\begin{keyword}
\MSC{18G90} \sep Homological algebra \sep
Framed bicategories \sep Directed topology \sep Embedding theorems \sep Acyclic model theorem
\end{keyword}

\end{frontmatter}

\section{Introduction}
Designing suitable homology and homotopy theories for objects which are directed by nature, such as directed spaces, is a notorious difficult problem~\cite{goubault1995geometrie,grandis2004,fahrenberg2004directed,patchkoria2006,dubut2017}. The non-reversibility of time in those models makes the needed algebra much more complex than in the classical case: to faithfully abstract this non-reversibility in time, the algebraic structures employed cannot be reversible (like abelian groups in classical algebraic topology, or ordered homology groups in~\cite{grandis2004}) or even cancellative (like Patchkoria's theory in cancellative monoids~\cite{patchkoria2006}), because non-homotopic directed paths can become homotopic when extended, see Figure~\ref{fig:matchbox}.
\begin{figure}[ht]
    \centering
     \begin{tikzpicture}[auto,scale = .8, baseline=($0.2*(C) + (A)$)]
 		\node at (-2,-0.5) {};
		\coordinate (A) at (0,0);
		\coordinate (B) at (2,0.9);
		\coordinate (C) at (-1.3,1.8);
		\coordinate (D) at (0.7,2.6);
		\coordinate (A') at (0,1);
		\coordinate (B') at (2,1.9);
		\coordinate (C') at (-1.3,2.8);
		\coordinate (D') at (0.7,3.6);
         \draw[fill=gray, opacity=0.2] 
    ($0.2*(B) + (A)$) -- 
    ($0.2*(B) + (B)$) -- 
    ($0.2*(B) + (B')$) -- 
    ($0.2*(B) + (A')$) -- 
    ($0.2*(B) + (A)$);
     \draw[fill=gray, opacity=0.2] 
    ($1.4*(C) + 0.2*(B) + (A)$) -- 
    ($1.4*(C) + 0.2*(B) + (B)$) -- 
    ($1.4*(C) + 0.2*(B) + (B')$) -- 
    ($1.4*(C) + 0.2*(B) + (A')$) -- 
    ($1.4*(C) + 0.2*(B) + (A)$);
        \draw[fill=gray, opacity=0.2] 
    ($1.4*(B) + 0.2*(C) + (A)$) -- 
    ($1.4*(B) + 0.2*(C) + (C)$) -- 
    ($1.4*(B) + 0.2*(C) + (C')$) -- 
    ($1.4*(B) + 0.2*(C) + (A')$) -- 
    ($1.4*(B) + 0.2*(C) + (A)$);
    \draw[fill=gray, opacity=0.2] 
    ($0.2*(C) + (A)$) -- 
    ($0.2*(C) + (C)$) -- 
    ($0.2*(C) + (C')$) -- 
    ($0.2*(C) + (A')$) -- 
    ($0.2*(C) + (A)$);
     \draw[fill=gray, opacity=0.2] 
    ($0.2*(D) + (A')$) -- 
    ($0.2*(D) + (B')$) -- 
    ($0.2*(D) + (D')$) -- 
    ($0.2*(D) + (C')$) -- 
    ($0.2*(D) + (A')$);
		\draw[->, green, ] ($0.2*(B) + (A)$) -- ($0.2*(B) + (B)$);
		\draw[->, blue, ] ($0.2*(C) + (A)$) -- ($0.2*(C) + (C)$);
		\draw[->, green, ] ($1.4*(B) + 0.2*(C) + (A)$) -- ($1.4*(B) + 0.2*(C) + (C)$);
		\draw[->, blue, ] ($1.4*(C) + 0.2*(B) + (A)$) -- ($1.4*(C) + 0.2*(B) + (B)$);
		\draw[->, red,  ] ($1.4*(C) + 0.2*(B) + (B)$) -- 
    ($1.4*(C) + 0.2*(B) + (B')$);
    \draw[->, red,  ] ($1.4*(B) + 0.2*(C) + (C)$) -- 
    ($1.4*(B) + 0.2*(C) + (C')$);
        \end{tikzpicture}
        ~
     \begin{tikzpicture}[auto,scale = .8, baseline=(A)]
 		\node at (-2,-0.5) {};
		\coordinate (A) at (0,0);
		\coordinate (B) at (2,0.9);
		\coordinate (C) at (-1.3,1.8);
		\coordinate (D) at (0.7,2.6);
		\coordinate (A') at (0,1);
		\coordinate (B') at (2,1.9);
		\coordinate (C') at (-1.3,2.8);
		\coordinate (D') at (0.7,3.6);
        \draw [fill = gray,opacity = 0.22] (A) -- (B) -- (B') -- (A') -- (A);
		\draw [fill = gray,opacity = 0.22] (A) -- (C) -- (C') -- (A') -- (A);
		\draw [fill = gray,opacity = 0.2] (A') -- (B') -- (D') -- (C') -- (A');
        \draw[->, green, ] (A) -- (B);
		\draw[->, blue, ] (A) -- (C);
		\draw[->, green, ] (B) -- (D);
		\draw[->, blue, ] (C) -- (D);
		\draw[->, red,  ] (D) -- (D');
\end{tikzpicture}
    \caption{Non-cancellative behavior in Fahrenberg's matchbox~\cite{fahrenberg2004directed}. The blue and green directed 
    paths are not dihomotopic because any homotopy would have to go through the 
    upper face and so through undirected paths. However, they become 
    dihomotopic after extending them with the red directed path.}
    \label{fig:matchbox}
\end{figure}
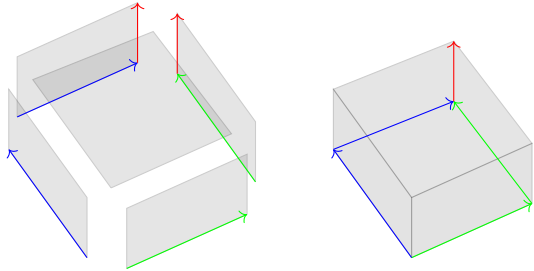
This means we have to move away from classical abelian categories as a suitable algebraic framework. Still, some form of local abelianity is needed: directed spaces are $(\infty,1)$-categories by nature (see~\cite{dubut2016directed}), because if paths are not invertible, (higher) homotopies are. Consequently, over the non-invertible low-dimensional data lies a more classical, fully invertible world where abelian structures can be used. That is the direction some took to develop non-cancellative homology theories for directed spaces, such as homology as functor from a category of traces to the singular homology of trace spaces for directed spaces \cite{dubut2015natural} or homology bimodules over a path algebra for precubical sets \cite{goubault2025directed}.
Those theories have a similar flavor: they look at the classical homology modules of trace spaces between two points and describe how those modules evolve by pre- and post-composing with additional traces. Written differently, they make traces act on the left and on the right of an object in an abelian category. Following this idea, we introduce abelian framed bicategories as a new suitable framework for developing homology theories for directed structures.
Framed bicategories, introduced by Shulman in~\cite{shulman2007framed}, are a way to categorify the way two structures, the coefficients, act (one on the left, one on the right) on another structure and how those actions evolve by changing the coefficients by extension and restriction. They are given by two bifibrations $B \to C$ interacting with each other (understand that $C$ is the category of coefficients while $B$ is a category of structures on which coefficients can act), one for each action, and for which the effect of restriction (resp. the extension) of coefficients is given by cartesian (resp. cocartesian) liftings. Then, abelian framed bicategories are framed bicategories for which the fibers above pairs of coefficients (that is the subcategories of structures on which fixed coefficients act) are abelian and for which the bicategorical structure preserves enough of the abelian structure. The archetypal examples of such are categories of bimodules over algebras, such as the ones used in~\cite{goubault2025directed} to develop the homology of precubical sets.
In order to validate our postulate, we then develop homological algebra in abelian framed categories, proving variants of standard theorems related to exact sequences, such as long exact sequences for relative homology, \MV{} theorem, and the Künneth formula.
Finally, we prove that under some conditions, an abelian framed bicategories is essentially just a category of bimodules over algebras, reminiscent of Gabriel's and \FM{}'s embedding theorems~\cite{gabriel1972unzerlegbare,mitchell1964full}.

\paragraph{Contents of the paper and main results}
In Section~\ref{section:fb}, 
we give all the necessary background material about framed bicategories, as introduced in e.g. \cite{shulman2007framed}, 
especially basic definitions of double categories, framed bicategories, and 
their restriction and (co)extensions of coefficients. We also 
introduce several examples used throughout the paper: bimodules over various 
structures (algebras, absorption monoids) and distributors.
We then define the main object of study here, 
abelian framed bicategories, in Section~\ref{section:afb}:
\begin{def31}
An \look{abelian framed bicategory} is a framed bicategory $\D$ that is locally abelian, i.e. all local categories $\DD(A, B)$ are abelian, and such that horizontal composition yields additive functors: $$\DD(A,B) \times \DD(B,C) \to \DD(A,C), (M,N) \mapsto M \odot N$$ for all objects $A,B,C \in \D_0$. 
\end{def31}
In such abelian framed bicategories, we can define chain and cochain complexes, and then 
homology and cohomology, see Section~\ref{section:homology}: 
\begin{def41}
A \look{chain complex} in $\AAA$ is a chain complex in a local abelian category of $\AAA$, i.e., a chain complex in $\A(A,B)$ for some objects $A$ and $B$. 
A \look{morphism} $(\alpha_i)_{i\geq 0}$ from a chain complex $(M_i)_{i \geq 0}$ in $\A(A,B)$ to a chain complex $(N_i)_{i \geq 0}$ in $\A(C,D)$ is a pair of vertical arrows $f: A \to C$ and $g: B \to D$ along with $2$-cells $\alpha_i:  M_i \cc{f}{g} N_i$ for all $i \geq 0$, such that the following diagram commutes for all $i \geq 0$:
\begin{equation*}
\begin{tikzcd}
	A && B \\
	\\
	C && D
	\arrow[""{name=0, anchor=center, inner sep=0}, "{M_{i+1}}"', "\shortmid"{marking}, curve={height=12pt}, from=1-1, to=1-3]
	\arrow[""{name=1, anchor=center, inner sep=0}, "{M_i}", "\shortmid"{marking}, curve={height=-12pt}, from=1-1, to=1-3]
	\arrow["f"', from=1-1, to=3-1]
	\arrow["g", from=1-3, to=3-3]
	\arrow[""{name=2, anchor=center, inner sep=0}, "{N_{i+1}}"', "\shortmid"{marking}, curve={height=12pt}, from=3-1, to=3-3]
	\arrow[""{name=3, anchor=center, inner sep=0}, "{N_i}", "\shortmid"{marking, text={rgb,255:red,128;green,128;blue,128}}, color={rgb,255:red,128;green,128;blue,128}, curve={height=-12pt}, from=3-1, to=3-3]
	\arrow["{ \partial_{i+1}}", shorten <=5pt, shorten >=5pt, Rightarrow, from=0, to=1]
	\arrow["{\alpha_i}"{description, pos=0.7}, shift left=5, color={rgb,255:red,128;green,128;blue,128}, shorten <=4pt, shorten >=4pt, Rightarrow, from=1, to=3]
	\arrow["{\alpha_{i+1}}"{description, pos=0.3}, shift right=5, shorten <=4pt, shorten >=4pt, Rightarrow, from=0, to=2]
	\arrow["{\partial_{i+1}}"', color={rgb,255:red,128;green,128;blue,128}, shorten <=5pt, shorten >=5pt, Rightarrow, from=2, to=3]
	\arrow[curve={height=12pt}, from=1-1, to=1-3]
\end{tikzcd}
    \end{equation*}
\end{def41}
\noindent 

\begin{def46}
The \look{homology} of an object $X$ in $\C$ is the homology: $$\HM_*(X): L(X) \tobar R(X)$$ of the chain complex $M_*(X)$ in the local abelian category $\DD(L(X), R(X))$. This construction is functorial.
\end{def46}

Similarly to homology theories in abelian categories, we derive interesting 
exact sequences in Section~\ref{section:exact_sequences}.
For particular relative pairs that we define, see Section~\ref{section:les_rh}, we have a relative homology exact sequence, very similar to the one in singular or simplicial homology: 

\begin{thm54}
For all relative pairs $(X,Y)$ that we define, we show that there is a long exact sequence akin to the classical relative homology exact sequence in simplicial or singular homology theories: 
\begin{center}
    \begin{tikzcd}[arrow style=math font,cells={nodes={text height=2ex,text depth=0.75ex}}]
    & & 0 \arrow[draw=none]{d}[name=X,shape=coordinate]{} \\
       \HM_0(X,Y) \arrow[curarrow=X]{urr}{} 
       & \HM_{0}(X) \arrow[l] \arrow[draw=none]{d}[name=Y, shape=coordinate]{} & \arrow[l] \cdots \\
       \HM_{i}(X,Y) \arrow[curarrow=Y]{urr}{} & \HM_{i}(X) \arrow[l] \arrow[draw=none]{d}[name=Z,shape=coordinate]{} & \lex X \HM_i(Y) \arrow[l] \\
       \HM_{i+1}(X,Y) \arrow[curarrow=Z]{urr}{} & \HM_{i+1}(X) \arrow[l] & \cdots \arrow[l]
   \end{tikzcd}
\end{center}
\end{thm54}

Similarly, we will prove a \MV{} sequence, for `good covers', see Section~\ref{sec:MayerVietoris}: 

\begin{thm57}
Consider a good cover of an object $X$ in $\C$, as in Definition~\ref{def:good_cover}. Then, we have the following long exact sequence in homology:
\begin{center}
    \begin{tikzcd}[arrow style=math font,cells={nodes={text height=2ex,text depth=0.75ex}}, column sep=small]
    & & 0 \arrow[draw=none]{d}[name=X,shape=coordinate]{} \\
       \HM_1(X) \arrow[curarrow=X]{urr}{} 
       & \lex X \HM_{1}(X_1) \oplus \lex X \HM_1(X_2) \arrow[l] \arrow[draw=none]{d}[name=Y, shape=coordinate]{} & \arrow[l] \cdots \\
       \HM_{i}(X) \arrow[curarrow=Y]{urr}{} & \lex X \HM_{i}(X_1) \oplus \lex X HMi(X_2) \arrow[l] \arrow[draw=none]{d}[name=Z,shape=coordinate]{} & \lex X \HM_i(X_1 \cap X_2) \arrow[l] \\
       \HM_{i+1}(X) \arrow[curarrow=Z]{urr}{} & \lex X \HM_{i+1}(X_1) \oplus \lex X \HM_{i+1}(X_2) \arrow[l] & \cdots \arrow[l]
   \end{tikzcd}
\end{center}
\end{thm57}
\noindent 
as well as a K\"unneth theorem, in particular monoidal abelian framed bicategories, see Section~\ref{sec:Kunneth}: 

\begin{thm59}
For all objects $X, Y$ in $\C$, if the tensor product functor $\otimes: \DD(L(X), R(X)) \times \DD(L(Y), R(Y)) \to \DD(L(X) \otimes L(Y), R(X) \otimes R(Y))$ satisfies the conditions of the algebraic Künneth theorem~\cite[Theorem 3.19]{fluch2004kunneth}, see Appendix~\ref{annex:algebraic}, then we have the following short exact sequence, where $\Tor_1$ denotes the first left derived functor of the tensor product functor:
$$ 0 \to l^*(\HM(X) \otimes \HM(Y)){r}^* \to \HM(X \otimes Y) \to l^*\Tor_1(\HM(X),\HM(Y)){r}^* \to 0.$$
\end{thm59}
Finally, we define a specific class of abelian framed bicategories, see Section~\ref{section:embedding_theorems}:
\begin{def626}
A \look{module-like} abelian framed bicategory is a locally cocomplete, closed, abelian framed bicategory $\AAA$ with an initial coefficient $I$ such that all restriction functors are faithful and $U_I$ is a compact projective generator of $\A(I,I)$.
\end{def626}
\noindent for which we get a representation theorem, similar to Gabriel theorem for abelian categories: 
\begin{thm645}
Let $\AAA$ be a module-like abelian framed bicategory. Then, there is a framed lax functor $F : \AAA \to \BimodUI$ such that $F_1$ is fully faithful and locally an equivalence of categories. 
\end{thm645}
\paragraph{Related work}
This research builds upon and extends the concrete directed homology theory of precubical sets developed in~\cite{goubault2025directed}. We abstract and generalize the combinatorial arguments from~\cite{goubault2025directed} using our new framework of abelian framed bicategories, based on the work of Shulman~\cite{shulman2007framed}.

Abelian framed bicategories are not homological categories (they are not even pointed in general). Rather, they consist of a structured patchwork of abelian categories. However, one could just as well consider homological categories instead of abelian ones, see Remark~\ref{remark:related}.
An abelian framed bicategory can be regarded as the Grothendieck construction of an indexed category $F: \C^{\mathrm{op}} \to \cat$ whose constituent categories are abelian. Just as Shulman’s work~\cite{shulman2007framed} studies the interactions between the Grothendieck construction and monoidal structures, our work studies its interactions with abelian structures.

In Section~\ref{section:embedding_theorems}, we construct embedding theorems mapping certain abelian framed bicategories into concrete abelian framed bicategories of bimodules over rings. These embeddings share some properties with the Yoneda embedding constructed by Paré in his Yoneda theory for double categories~\cite{pare2011yoneda}, ours, however, is strictly functorial, see Remark~\ref{remark:pare}.
\section{Framed bicategories}\label{section:fb}
Our work is based on, and extends, the established notion of framed bicategory introduced by Shulman in~\cite{shulman2007framed}. 
We assume no prior familiarity with double categories and framed bicategories, and we present in this section the definitions, notations and the examples that will be used throughout this document.

\subsection{Double categories}

Double categories were first introduced by Ehresmann in~\cite{ehresmann1963categories}. Intuitively, double categories are internal categories in the $2$-category of small categories. Alternatively, they are a variant of  $2$-category where $1$-cells are interpreted as objects parameterized by pairs of $0$-cells, just like bimodules over rings. $2$-cells are then regarded as parameterized morphisms between parameterized objects. In addition, a different kind of $1$-cell represents morphisms between $0$-cells. For more details on this viewpoint, we refer to Shulman's introduction in~\cite{shulman2007framed}.

\begin{definition}[\cite{ehresmann1963categories}]\label{def:double_category}
    A \look{double category} $\D$ consists of two small categories $\D_0$ and $\D_1$, along with structural functors $L$, $R$ and $U$:
\begin{align*}
    U &: \D_0 \to \D_1 \\
    L, R &: \D_1 \to \D_0 \\
    \odot &: \D_1 \times_{\D_0} \D_1 \to \D_1
\end{align*}
where $\D_1 \times_{D_0} \D_1$ is the pullback over $\D_1 \xrightarrow{R} \D_0 \xleftarrow{L} \D_1$, 
satisfying the usual category laws, in particular the following equations for all objects $A \in \D_0$ and $M,N \in \D_1$:
   \begin{align*}
    LU(A) &= A = RU(A) & (M \odot N ) \odot P &= M \odot (N \odot P )\\
    L(M \odot N ) &= LM & UA \odot M &= M \\
    R(M \odot N ) &= RN & M \odot UB &= M
\end{align*}

Objects of the \look{vertical category} $\D_0$ are called \look{objects}, or \look{$0$-cells}, while arrows of $\D_0$ are called \look{vertical arrows} and written $f : A \to B$. Objects of the \look{horizontal category} $\D_1$ are called \look{$1$-cells} and written $M : A \tobar B$ whenever $L(M) = A$ and $R(M) = B$. An arrow $\alpha : M \to N$ in $\D_1$ with $L(\alpha) = f: A \to C$ and $R(\alpha) = g : B \to D$ is called a \look{$2$-cell}, written $\alpha : M \cc{{f}}{g} N$, or just $\alpha: M \to N$ when $g$ and $f$ are clear from the context, and drawn as:
$$
\begin{tikzcd}[sep=scriptsize]
	A & B \\
	C & D
	\arrow[""{name=0, anchor=center, inner sep=0}, "M"{inner sep=.8ex}, "\shortmid"{marking}, from=1-1, to=1-2]
	\arrow["f"', from=1-1, to=2-1]
	\arrow["g", from=1-2, to=2-2]
	\arrow[""{name=1, anchor=center, inner sep=0}, "N"'{inner sep=.8ex}, "\shortmid"{marking}, from=2-1, to=2-2]
	\arrow["\alpha"{description}, between={0.2}{0.8}, Rightarrow, from=0, to=1]
\end{tikzcd}
.$$

If $A$ and $B$ are objects, we write $\DD(A, B)$ for the category of $1$-cells and \look{globular} $2$-cells from $A$ to $B$, i.e., whose vertical arrows are the identities $\id_A$ and $\id_B$. For all $1$-cells $M$ and $N$, we denote the homset $\Hom_{\D_1}(M,N)$ by $\D(M,N)$, and the subset of $2$-cells of the form $M \cc{f}{g} N$ by $\D_{f,g}(M,N)$. Moreover, if both $M$ and $N$ are objects of $\DD(A,B)$ for some $0$-cells $A$ and $B$, then we write $\DD(M,N)$ for the homset $\Hom_{\DD(A,B)}(M,N)$. We say a property holds \look{locally} when it holds for all categories $\DD(A,B)$. The image of the \look{unit} functor $U$ on a object $A$ is also denoted $U_A$. Finally, the functor $\odot$ is called \look{horizontal composition}, while \look{vertical composition} denotes the usual composition in the categories $\D_0$ and $\D_1$. 
\end{definition}
Following Shulman~\cite{shulman2007framed}, we work in the weakened setting of pseudo double categories. These are internal categories that are only weakly associative and unital, i.e., there are natural isomorphisms satisfying the coherence axioms for bicategories:
\begin{align*}
\mathfrak{a}_{M,N,P} &: M  \odot (N  \odot P) \xrightarrow{\cong}(M  \odot N)  \odot P  \\
\mathfrak{l}_M &: U_A  \odot M \xrightarrow{\cong} M \\
\mathfrak{r}_M &: M  \odot U_B \xrightarrow{\cong} M
\end{align*}
In addition, the vertical and horizontal categories are allowed to be locally small rather than small. 
We also call them double categories, and use the word `strict' when necessary. By a result of Grandis and Paré~\cite[Theorem 7.5]{grandis1999limits}, every pseudo double category is equivalent to a strict one. Still, the main examples of double categories are more easily formulated as pseudo double categories. On a side note, if one wishes to weaken the horizontal composition even more, for example to define horizontal composition using universal properties, then one can use the notion of virtual double category~\cite{burroni1971t, cruttwell2009unified}.
In a double category $\D$, local categories of the form $\DD(A,A)$, for all vertical objects $A$ are monoidal. Therefore, the following result, due to Kelly~\cite{kelly1964maclane}, holds.
\begin{lemma}\label{lemma:reql}
For all objects $A$ in a double category $\D$, the two isomorphisms $\mathfrak l_{U_A}, \mathfrak r_{U_A}: U_A \odot U_A \to U_A$ coincide.
\end{lemma}
\begin{example}\label{ex:bimod} Let $R$ be an arbitrary commutative ring. Consider two $R$-algebras $A$ and $B$. A \look{left $A$-module} (resp. \look{right $B$-module}) $M$ is a left $A$-module (resp. right $B$-module) $(M, 0, +, \act_A)$ where $A$ (resp. $B$) is regarded as a ring.  
Similarly, we define \look{$A$-$B$-bimodule} as abelian groups $M$ that are both left $A$-modules and right $B$-modules in a compatible way, i.e. such that for all $a\in A, b \in B$ and $m \in M$,  $$ a \act_A (m \act_B b) = (a \act_A m) \act_B b,$$
and such that $M$ is $R$-symmetric in the sense that the two actions of $R$ induced by $A$ and $B$ on $M$ coincide.
Let $M$ and $N$ be two $A$-$B$ and $C$-$D$-bimodules, respectively. An \look{$(f, g)$-bilinear map} $\alpha: M \to N$ is a an abelian group homomorphism $\alpha: M \to N$ along with $R$-module homomorphisms $f: A \to C$ and $g: B \to D$ satisfying for all $a \in A, b \in B$ and $m \in M$:
$$ \alpha(a \act_A m \act_B b) = f(a) \act_C \alpha(m) \act_D g(b).$$
The \look{double category of bimodules over $R$-algebras}, denoted $\Bimod$, has:
\begin{itemize}[noitemsep]
    \item as objects $R$-algebras,
    \item as vertical arrows homomorphisms of $R$-algebras,
    \item as $1$-cells $M: A \tobar B$, $A$-$B$-bimodules $M$.
    \item as $2$-cells $\alpha: M \cc{f}{g} N$, $(f, g)$-bilinear maps $\alpha: M \to N$.
\end{itemize}
$R$-algebras can be regarded as bimodules over themselves, and horizontal composition is induced by the tensor products $\otimes_B$ of $A$-$B$-bimodules and $B$-$C$-bimodules over an $R$-algebra $B$. When $R = \Z$, we recover the double category of bimodules over rings. This is the prototypical example of double category. 
\end{example}
\begin{example}[{\cite[Ex. 2.6]{shulman2007framed}},\cite{ponto2014linearity}]\label{ex:dist} For all bicomplete closed symmetric monoidal category $\mathcal V$ with unit $1$ and product $\otimes$, like $\set$ or ${}_R\mmod$ for $R$ a commutative ring, the double category of $\mathcal V$-distributors, or profunctors, denoted $\mathcal V$-$\Dist$, has:
\begin{itemize}[noitemsep]
    \item as objects small categories,
    \item as vertical arrows functors,
    \item as $1$-cells $F: A \tobar B$, $A$-$B$-distributors, i.e., functors $A \times B^{\mathrm{op}} \to \mathcal V$,
    \item as $2$-cells $\alpha: F \cc{f}{g} G$, natural transformations: $$\alpha_{(x,y)}: F(x,y) \to G(f(x),g(y)).$$
\end{itemize}
For all small categories $A$, the $1$-cell $U_A: A \tobar A$ is defined by the copowers of the unit $1$ by the homsets of $A$: $U_A(x,y) = A(y,x) \cdot 1$, and the horizontal composition of two $1$-cells $F$ and $G$ is given on objects by the coend ${F \odot G}(a,c) = \int^{b \in B} F(a,b)\otimes G(b,c)$.
When $\mathcal V$ is $\set$ (resp. ${}_R\mmod$), $U_A(x,y)$ is the homset $A(y,x)$  (resp. the free $R$-module generated by $A(y,x)$).
Distributors are related to persistence: let $\Repr$ denote the sub-double category of $\rmod$-$\Dist$ whose $0$-cells are posets, regarded as small, thin categories. Distributors $M : P \tobar \{*\}$ are known as \look{representations} of $P$.  
\end{example}
\begin{example}
\label{ex:absorptionmonoid}
A notion of directed homotopy (bi)module was introduced in \cite{goubault2025homotopy}. Algebraically, they are \look{bimodules over absorption monoids} valued in \look{absorption semigroups}. Absorption monoids (resp. semigroups) are monoids (resp. semigroups) with an absorbing element $0$. A left $T$-{module} $(G,T)$ over absorption monoids is an absorption semigroup $M$ together with an action of an absorption monoid $T$, i.e. a map
$\act: \ T\times M \rightarrow M$ satisfying the classical axioms of monoids, and the fact that the action of $0$ on elements of $M$ is $0$. Similarly, we define right $T$-modules and $T$-$S$-bimodules, as in Example \ref{ex:bimod}.
Morphisms of bimodules over monoids are triples $(h,f,g): \ (M,T,S) \rightarrow (M',T',S')$ where $h: \ M \rightarrow M'$, $f: \ T \rightarrow T'$ and $h: \ S \rightarrow S'$ are absorption monoid morphisms such that 
$f(t\bullet m)=f(t)\bullet h(m)\bullet g(s)$. As in the case of bimodules over algebras, Example \ref{ex:bimod}, $h$ is called an $(f,g)$-bilinear map from $M$ to $M'$. 
    The double category of bimodules over absorption monoids
    has:
    \begin{itemize}[noitemsep]
        \item as objects absorption monoids,
        \item as vertical arrows morphisms of absorption monoids,
        \item as $1$-cells $M: A \tobar B$, absorption semigroups $M$ equipped with compatible actions of $A$ and $B$ on $M$, i.e., bimodules over absorption monoids in the sense above,
        \item as $2$-cells $\alpha: M \cc{f}{g} N$, $(f,g)$-bilinear morphisms.
    \end{itemize}
An absorption monoids can be regarded as bimodule over itself, and horizontal composition is given by the tensor product of bimodules.

In \cite{goubault2025homotopy}, it is proven that traces in directed spaces, i.e. directed paths modulo reparametrization, equipped with concatenation, an absorbing and a neutral 1 element, is an absorption monoid. The singular simplices of dimension $n$ of the trace space, the topological space of directed paths modulo reparametrization, equipped with the respective quotient topology of the compact-open topology is shown to form a bimodule over the absorption monoid of traces. The construction of singular simplices in the trace space actually forms a Kan simplicial object in the category of bimodules over the trace monoid. From this stems a natural definition of directed homotopy bimodules and some of their properties. 
\end{example}
\subsection{A framed bicategory}\label{subsection:fb}
Framed bicategories, or fibrant double category, is a double category which admits restriction and extension of $1$-cells along vertical arrows, just as bimodules admit restriction and extension of scalars. They were introduced by Shulman in~\cite{shulman2007framed}, building upon the equivalent notion of proarrow equipment on $2$-categories~\cite{wood1982abstract,wood1985proarrows}.
\begin{definition}[\cite{shulman2007framed}]\label{def:fb}
A \look{framed bicategory} $\D$ is a double category such that the functor  
$(L, R) : \D_1 \to \D_0 \times \D_0$ is a bifibration.
Explicitly, given a framed bicategory $\D$, for all vertical arrows $f: A \to B$ and $g: C \to D$ and $1$-cell $M: B \tobar D$, there are \look{cartesian} $2$-cells, unique up to precomposition by a unique globular isomorphic $2$-cell, that are written $f^*Mg^* \xrightarrow{\cart} M$ or $f^*Mg^* \xrightarrow{\ct} M$ and drawn as:
$$
\begin{tikzcd}[sep=scriptsize]
	A\  & \ C \\
	B & D
	\arrow[""{name=0, anchor=center, inner sep=0}, "{f^*Mg^*}", "\shortmid"{marking}, from=1-1, to=1-2]
	\arrow["f"', from=1-1, to=2-1]
	\arrow["g", from=1-2, to=2-2]
	\arrow[""{name=1, anchor=center, inner sep=0}, "M"', "\shortmid"{marking}, from=2-1, to=2-2]
	\arrow["\cart"{description}, draw=none, from=0, to=1]
\end{tikzcd}
,$$ cartesian in the sense that any $2$-cell: 
$$
\begin{tikzcd}[sep=scriptsize]
	{A'} & {C'} \\
	B & D
	\arrow[""{name=0, anchor=center, inner sep=0}, "N", "\shortmid"{marking}, from=1-1, to=1-2]
	\arrow["fh"', from=1-1, to=2-1]
	\arrow["gk", from=1-2, to=2-2]
	\arrow[""{name=1, anchor=center, inner sep=0}, "M"', "\shortmid"{marking}, from=2-1, to=2-2]
	\arrow[shorten <=9pt, shorten >=9pt, Rightarrow, from=0, to=1]
\end{tikzcd}
$$
factors uniquely as:
$$
\begin{tikzcd}[sep=scriptsize]
	{A'} & {C'} \\
	A\  & \ C \\
	B & D
	\arrow[""{name=0, anchor=center, inner sep=0}, "N", from=1-1, to=1-2]
	\arrow["h"', from=1-1, to=2-1]
	\arrow["k", from=1-2, to=2-2]
	\arrow[""{name=1, anchor=center, inner sep=0}, "{f^*Mg^*}"{description}, from=2-1, to=2-2]
	\arrow["f"', from=2-1, to=3-1]
	\arrow["g", from=2-2, to=3-2]
	\arrow[""{name=2, anchor=center, inner sep=0}, "M"'{inner sep=.8ex}, "\shortmid"{marking}, from=3-1, to=3-2]
	\arrow[shorten <=9pt, shorten >=9pt, Rightarrow, from=0, to=1]
	\arrow["\cart"{description}, draw=none, from=1, to=2]
\end{tikzcd}
.$$

Dually, there are \look{cocartesian} $2$-cells, unique up to postcomposition by a unique globular isomorphic $2$-cell, that are denoted $M \xrightarrow{\cocart} f_!Mg_!$ or $M \xrightarrow{\cct} f_!Mg_!$ for simplicity. 
\end{definition}
We record here some facts about cofibrations (the dual statements hold for fibrations):
\begin{enumerate}[noitemsep]
    \item For all cocartesian $2$-cells $\alpha$, $\alpha$ is epic.\label{fact1}
    \item For all $2$-cells $\alpha$ and $\beta$, if $\beta\circ \alpha$ and $\alpha$ are cocartesian, then $\beta$ is too.\label{fact2}
    \item For all cocartesian $2$-cells $\alpha: X \to X_1$ and $\beta: X \to X_2$, if $\alpha$ and $\beta$ lie over the same vertical arrow $h$, then there is a unique isomorphism $X_1 \cong X_2$ commuting with $\alpha$ and $\beta$.\label{fact3} 
\end{enumerate}
Assuming the axiom of choice, any fibration admits a (normal) cleavage, i.e., a choice of (co)cartesian $2$-cells for all compatible $1$-cells and pairs of vertical arrows. A cleavage allows to define change of coefficients functors for all pairs of vertical arrows. Due to the uniqueness property of (co)cartesian $2$-cells, any other choice of cleavage gives naturally isomorphic change of coefficients functors.  
In Sections~\ref{section:afb}, \ref{section:homology} and~\ref{section:exact_sequences}, any normal cleavage suffices, as $1$-cells are studied up to isomorphism. In Section~\ref{section:embedding_theorems}, however, we construct a specific cleavage to satisfy certain strictness conditions. Indeed, in Section~\ref{section:framed_embedding} in particular, $1$-cells are studied up to equality. 
\begin{definition}
For all vertical arrows $f: A \to B$ and $g: C \to D$, the \look{restriction of scalars} functor $f^*(\mathunderscore)g^* : \DD(B,D) \to \DD(A,C)$ is defined on $1$-cells $M : B \tobar D$ by $f^*Mg^*$, and on globular $2$-cells $\alpha: M \to N$ as the factorization of the $2$-cell $f^*Mg^* \xrightarrow{\cart} M \xrightarrow{\alpha} N$ through the cartesian $2$-cell $f^*Ng^* \xrightarrow{\cart} N$. Dually, we define \look{extension of scalars} functors for all pairs of vertical arrows, as well as change of coefficient functors on the right and on the left for all vertical arrows.  
\end{definition}
By~\cite[Proposition 3.9]{shulman2007framed}, the restriction functor along $(f,g)$ is right adjoint to the extension functor along $(f,g)$. More concretely, the unit $\eta$ of the adjunction is defined pointwise on $1$-cells $M$ by the factorization of the cocartesian $2$-cell $X \xrightarrow{\cocart} f_!Mg_!$ through the cartesian $2$-cell $f^*f_!Mg_!g^* \xrightarrow{\cart} f_!Mg_!$.
When all restriction functors admit a right adjoint, these adjoints are called \look{coextension}. This condition is equivalent to $\D$ being a closed framed bicategory. 
\begin{definition}[\cite{shulman2007framed}]\label{def:closed}
    A framed bicategory $\D$ is \look{closed} when its underlying horizontal bicategory $\DD$ is closed as a bicategory, i.e., horizontal composition has right adjoints on both variables $\D_1(M \odot N, P) \cong \D_1(M, N \rhd P) \cong \D_1(N, P \lhd M)$. In this case, for all vertical arrows $f: A \to B$, the restriction functors $f^*$ have right adjoints called \look{coextension}, denoted $f_*$ and naturally isomorphic to $(\mathunderscore) \lhd {}_f B$.
\end{definition}
In the rest of this document, we will also use the notion of `external' monoidal structure: 
\begin{definition}[\cite{shulman2007framed}]\label{def:monoidal}
A \look{monoidal} framed bicategory is a pseudo-monoid in the $2$-category of framed bicategories~\cite[\S 6]{shulman2007framed}. More explicitly, this consists of a framed bicategory $\D$ such that the horizontal and vertical categories $\D_0$ and $\D_1$ are monoidal, with both monoidal products denoted $\otimes$, and satisfying the usual axioms. In particular, if $I$ denotes the monoidal unit in $\D_0$, then the monoidal unit in $\D_1$ is $U(I)$, and for all $1$-cells $M: A \tobar B$ and $N : C \tobar D$, $M \otimes N $ is a $1$-cell $A\otimes C \tobar B \otimes D$.
\end{definition}
\begin{example}
\begin{itemize}
    \item The standard change of coefficients functors for modules turn $\Bimod$ into a framed bicategory. Explicitly, for any $B$-$B'$-bimodule $M$ and $R$-algebra homomorphisms $f: A \to B$ and $g: B \to C$, the left restriction $f^*M$ is the $A$-$B'$-bimodule with underlying set $M$, right action $\act_{B'}$, and left action $(a,m) \mapsto f(a) \act_B m$. The extension $g_!M$ is given by $U_Cg^* \otimes_B M$, and the coextension $g_*M$ by $\Bimod(g^*U_C,M)$. In addition, $\Bimod$ is closed by the tensor-hom adjunction, and monoidal via the external tensor product of bimodules $\otimes_R$, associating to an $A$-$B$-bimodule $M$ and a $C$-$D$-bimodule $N$ the $A \otimes_R C$-$B\otimes_R D$-bimodule $M \otimes_R N$.
     \item The double category $\mathcal V$-$\Dist$ of $\mathcal V$-distributors is also a closed framed bicategory. Restriction is given by precomposition, and extension and coextension by left and right Kan extension, respectively. The tensor product of $\mathcal V$-categories makes $\mathcal V$-$\Dist$ monoidal. For more details, see~\cite[Example 2.6]{shulman2007framed} where this example of framed bicategory is presented in more details.
    \item We are in a similar situation for bimodules in absorption monoids, see Example \ref{ex:absorptionmonoid}. It is shown in \cite{goubault2025homotopy} that the restriction of scalar functor is right-adjoint to the extension of scalar functor, and left-adjoint to the co-extension of scalar functor. 
\end{itemize}
\end{example}
By~\cite[Theorem 4.1]{shulman2007framed}, the framed structure on a double category $\D$ can equivalently be expressed as the data of $1$-cells $\lus {f \,} B: A \tobar B$ and $B_f : B \tobar A$ called \look{special objects}, for all vertical arrows $f: A \to B$, along with certain $2$-cells satisfying the diagrammatic conditions of~\cite[Theorem 4.1]{shulman2007framed}. If a double category is framed, then special objects can be defined as follows:
\begin{definition}\label{def:special_objects}
For all framed bicategories $\D$, we define \look{special objects} ${}_f B = f^*U_B$ and $B_f = U_Bf^*$ for all vertical arrows $f: A \to B$. 
\end{definition}
Conversely, if a double category contains special objects, then one can show that for all $1$-cells $M: B \tobar D$, vertical arrows $f: A \to B$, $g: B \to A$ and vertical objects $C$, the following $2$-cells are cartesian and cocartesian, respectively:
\[\begin{tikzcd}
	A & B & C \\
	B & B & C \\
	B && C
	\arrow[""{name=0, anchor=center, inner sep=0}, "{{}_fB}"{inner sep=.8ex}, "\shortmid"{marking}, from=1-1, to=1-2]
	\arrow["f"', from=1-1, to=2-1]
	\arrow[""{name=1, anchor=center, inner sep=0}, "M"{inner sep=.8ex}, "\shortmid"{marking}, from=1-2, to=1-3]
	\arrow[equals, from=1-2, to=2-2]
	\arrow[equals, from=1-3, to=2-3]
	\arrow[""{name=2, anchor=center, inner sep=0}, "{U_B}"'{inner sep=.8ex}, "\shortmid"{marking}, from=2-1, to=2-2]
	\arrow[equals, from=2-1, to=3-1]
	\arrow[""{name=3, anchor=center, inner sep=0}, "M"'{inner sep=.8ex}, "\shortmid"{marking}, from=2-2, to=2-3]
	\arrow[equals, from=2-3, to=3-3]
	\arrow[""{name=4, anchor=center, inner sep=0}, "M"'{inner sep=.8ex}, "\shortmid"{marking}, from=3-1, to=3-3]
	\arrow["\cart"{description}, draw=none, from=0, to=2]
	\arrow["{\id_M}"{description}, draw=none, from=1, to=3]
	\arrow["{\mathfrak l_B}"{description}, draw=none, from=2-2, to=4]
\end{tikzcd}~~~~~~~~
\begin{tikzcd}
	B && C \\
	B & B & C \\
	A & B & C
	\arrow[""{name=0, anchor=center, inner sep=0}, "M"{inner sep=.8ex}, "\shortmid"{marking}, from=1-1, to=1-3]
	\arrow[equals, from=1-1, to=2-1]
	\arrow[equals, from=1-3, to=2-3]
	\arrow[""{name=1, anchor=center, inner sep=0}, "{U_B}"{inner sep=.8ex}, "\shortmid"{marking}, from=2-1, to=2-2]
	\arrow["g"', from=2-1, to=3-1]
	\arrow[""{name=2, anchor=center, inner sep=0}, "M"{inner sep=.8ex}, "\shortmid"{marking}, from=2-2, to=2-3]
	\arrow[equals, from=2-2, to=3-2]
	\arrow[equals, from=2-3, to=3-3]
	\arrow[""{name=3, anchor=center, inner sep=0}, "{A_g}"'{inner sep=.8ex}, "\shortmid"{marking}, from=3-1, to=3-2]
	\arrow[""{name=4, anchor=center, inner sep=0}, "M"'{inner sep=.8ex}, "\shortmid"{marking}, from=3-2, to=3-3]
	\arrow["{\mathfrak l^{-1}_B}"{description}, draw=none, from=0, to=2-2]
	\arrow["\chi_g"{description}, draw=none, from=1, to=3]
	\arrow["{\id_M}"{description}, draw=none, from=2, to=4]
\end{tikzcd}\]
where the cocartesian $2$-cell $U_B \ccs{g}{\id_B} A_g$ denoted $ \chi_g$ is defined as the factorization of $Ug$ trough the cartesian $2$-cell $U_Ag^* \to U_A$, and shown to be cartesian.
Symmetrically, there is a cocartesian $2$-cell ${}_g \chi: U_B \ccs{\id_B}{g} {}_g A$, as well as cartesian and cocartesian $2$-cells for horizontal composition on the right.
Consequently, there are unique isomorphisms $\lambda^*: {}_fB\odot M \to f^*M$ and $\lambda_! : B_g \odot M \to g_!M$ natural in $M$ that commute with the relevant (co)cartesian cells, and symmetrically natural isomorphisms $\rho^*$ and $\rho_!$.
In the rest of this document, we use the notations $B_f$ and ${}_f B$ as aliases for $f^*U_B$ and $U_Bf^*$ to emphasize the above cartesian $2$-cells and natural isomorphisms. 
\section{Abelian framed bicategories}
\label{section:afb}
An abelian category is an abstract categorical generalization of the concrete category of vector spaces and linear maps. 
Intuitively, the properties of vector spaces and linear maps that are used to defined homology vector spaces are abstracted into axioms: morphisms have kernels and cokernels; there is a zero object and zero morphisms between any two objects; and finite products and coproducts coincide.
Functors between abelian categories are called \look{additive} when they preserve direct sums and the zero object, and \look{exact} (resp. left, right exact) when they preserve exact sequences (resp. exactness on the left, right). 
See~\cite{weibel1994introduction, kashiwara2006categories} for textbook references.

In this section, we define abelian framed bicategories, which are framed bicategories that are locally abelian, and such that horizontal composition is additive, a condition that is satisfied as soon as a framed bicategory is closed. The motivation is to capture the fact that while  the category of modules over varying ring is not abelian, for fixed rings $A$ and $B$ the category of $A$-$B$-bimodules \emph{is} abelian. 
\begin{definition}\label{def:afb}
An \look{abelian framed bicategory} is a framed bicategory $\D$ that is locally abelian, i.e. all local categories $\DD(A, B)$ are abelian, and such that horizontal composition yields additive functors: $$\DD(A,B) \times \DD(B,C) \to \DD(A,C), (M,N) \mapsto M \odot N$$ for all objects $A,B,C \in \D_0$. 
\end{definition}
\begin{remark}
Alternative definitions might require either the category of objects $\D_0$ or the entire horizontal category $\D_1$ to be abelian, but such definitions would not accommodate our examples. Indeed, if $R$ is a ring, then the category of $R$-algebras is not abelian because the initial and terminal objects $R$ and $0$ are not isomorphic. Similarly, in the category of bimodules over varying $R$-algebras, the initial and terminal objects $0: R \tobar R$ and $0: 0 \tobar 0$ are not isomorphic.
\end{remark}
\begin{remark}\label{remark:related}
To the best of our knowledge, the notion of abelian framed bicategory has not previously appeared in the literature:
\begin{itemize}
    \item Semi-abelian categories~\cite{borceux2004mal, van2006homology} are weakened versions of abelian category in which homological algebra can still be carried out. Incidentally, they have also been used by Goubault to define directed homology theories~\cite{goubault2024semi}. Similarly, a homology theory can be defined for abelian framed bicategories even though the horizontal category is only locally abelian. However, there are two main differences: abelian framed bicategories are a $2$-categorical notion, and being locally abelian is orthogonal to being semi-abelian in the sense that one could equally consider locally semi-abelian categories, but we focus on abelian categories for simplicity.
   \item Several variants of abelian $2$-categories have been proposed in the literature~\cite{dupont2008abelian, nakaoka2011comparison}, but they are all {abelian Gpd-categories} in the sense of~\cite{dupont2008abelian}. They bear some resemblance to abelian framed bicategories (e.g., local zero $1$-cells), but they are understood as $2$-categories and not as double categories, and moreover all $2$-cells are invertible. 
   The closest connection to abelian framed bicategories would be a notion of abelian $2$-category with equipment, which does not seem to have been explored.
   \item The notion of monoidal abelian category is well studied under the name of tensor category, see for example~\cite{rivano1972categories,etingof2015tensor}. 
   However, in an abelian framed bicategory $\AAA$, neither the horizontal category $\AAA_1$ nor the local categories $\AAA(A,B)$ (for $A \neq B$) are typically monoidal abelian; only the categories $\AAA(A,A)$ carry a monoidal structure.
\end{itemize}
\end{remark}
In the remainder of this section, we deduce immediate consequences of Definition~\ref{def:afb}, introduce further definitions, and present our main examples of abelian framed bicategories.
In an abelian framed bicategory $\AAA$, 
since the change of coefficients functors are left or right adjoints, they are automatically right or left exact. Similarly, additivity of horizontal composition is automatic when $\AAA$ is closed, i.e., when the horizontal composition functors have both adjoints.
\begin{corollary}\label{cor:exactness}
Restriction functors are left exact, and extension functors right exact. If coextension functors also exist, then restriction functors are exact, coextension functors are left exact, and horizontal composition functors are right exact. 
In particular, all these functors are additive.
\end{corollary}
\begin{proof}
Right (resp. left) adjoint functors between abelian categories are left (resp. right) exact.
\end{proof}
\begin{corollary}\label{cor:extension_zero}
Restricting or (co)extending zero objects yields zero objects. For all vertical arrows $f: A \to C$ and $g: B \to D$ and $1$-cells $M: A \tobar B$ and $N: C \tobar D$, there are unique $2$-cells $0 \cc{f}{g} N$ and $M \cc{f}{g} 0$, unique up to postcomposition (resp. precomposition) by a unique isomorphism.
\end{corollary}
\begin{proof}
This follows from the additivity of the functors and the properties of (co)cartesian $2$-cells.
\end{proof}
The additional structures on framed bicategories carries on to abelian framed bicategories. A \look{closed abelian framed bicategory} is a framed bicategory that is both abelian and closed. For monoidal framed bicategories, an additional compatibility condition is required. 
\begin{definition}
A \look{monoidal abelian framed bicategory} is an abelian framed bicategory $\D$ that is also monoidal, and such that for all objects $A, B, C$, and $D$, the functor $\otimes: \DD(A,B) \times \DD(C,D) \to \DD(A \otimes C, B \otimes D)$ induced by the monoidal product in $\D_1$ is additive.
\end{definition}
\begin{example}\label{ex:abelian_monoidal}
    \begin{itemize}[noitemsep]
        \item The category of $A$-$B$-bimodules, where $A$ and $B$ are fixed $R$-algebras, is abelian. In addition, horizontal composition is additive since tensor product distributes over finite direct sums. Therefore, the framed bicategory $\Bimod$ is abelian. For the same reason, $\Bimod$ is also abelian monoidal. 
         \item The framed bicategory of $\mathcal V$-distributors is closed and if $\mathcal V$ is abelian, then the local category of $\mathcal V$-distributors over two fixed small categories is also abelian. In that case, $\mathcal V$-$\Dist$ is therefore an abelian framed bicategory. In particular, $\Repr$ is. 
        \item The framed bicategory of bimodules over absorption monoids (Example \ref{ex:absorptionmonoid}) is not abelian. 
    \end{itemize}
\end{example}
\section{Homology and cohomology in abelian framed bicategories}\label{section:homology}
In this section, we adapt classical notions of homological algebra to the setting of abelian framed bicategories. Since abelian framed bicategories are only locally abelian, we work as much as possible within local categories, using extension or restriction functors when necessary. In what follows, let $\AAA$ denote an abelian framed bicategory.
\subsection{Chain complexes in abelian framed bicategories}
We define chain complexes in $\AAA$ as chain complexes within the local abelian categories of $\AAA$, so that their homology is well-defined. Morphisms of chain complexes, however, are not necessarily local and satisfy properties similar to those of $2$-cells. In the remainder of this document, chain complexes are understood to be positively graded. In particular, we do not consider shifted chain complexes as in~\cite{goubault2025directed}.
\begin{definition}
\label{def:chaincomplex}
A \look{chain complex} in $\AAA$ is a chain complex in a local abelian category of $\AAA$, i.e., a chain complex in $\A(A,B)$ for some objects $A$ and $B$. 
A \look{morphism} $(\alpha_i)_{i\geq 0}$ from a chain complex $(M_i)_{i \geq 0}$ in $\A(A,B)$ to a chain complex $(N_i)_{i \geq 0}$ in $\A(C,D)$ is a pair of vertical arrows $f: A \to C$ and $g: B \to D$ along with $2$-cells $\alpha_i:  M_i \cc{f}{g} N_i$ for all $i \geq 0$, such that the following diagram commutes for all $i \geq 0$:
\begin{equation*}
\begin{tikzcd}
	A && B \\
	\\
	C && D
	\arrow[""{name=0, anchor=center, inner sep=0}, "{M_{i+1}}"', "\shortmid"{marking}, curve={height=12pt}, from=1-1, to=1-3]
	\arrow[""{name=1, anchor=center, inner sep=0}, "{M_i}", "\shortmid"{marking}, curve={height=-12pt}, from=1-1, to=1-3]
	\arrow["f"', from=1-1, to=3-1]
	\arrow["g", from=1-3, to=3-3]
	\arrow[""{name=2, anchor=center, inner sep=0}, "{N_{i+1}}"', "\shortmid"{marking}, curve={height=12pt}, from=3-1, to=3-3]
	\arrow[""{name=3, anchor=center, inner sep=0}, "{N_i}", "\shortmid"{marking, text={rgb,255:red,128;green,128;blue,128}}, color={rgb,255:red,128;green,128;blue,128}, curve={height=-12pt}, from=3-1, to=3-3]
	\arrow["{ \partial_{i+1}}", shorten <=5pt, shorten >=5pt, Rightarrow, from=0, to=1]
	\arrow["{\alpha_i}"{description, pos=0.7}, shift left=5, color={rgb,255:red,128;green,128;blue,128}, shorten <=4pt, shorten >=4pt, Rightarrow, from=1, to=3]
	\arrow["{\alpha_{i+1}}"{description, pos=0.3}, shift right=5, shorten <=4pt, shorten >=4pt, Rightarrow, from=0, to=2]
	\arrow["{\partial_{i+1}}"', color={rgb,255:red,128;green,128;blue,128}, shorten <=5pt, shorten >=5pt, Rightarrow, from=2, to=3]
	\arrow[curve={height=12pt}, from=1-1, to=1-3]
\end{tikzcd}
\end{equation*}
\end{definition}

We denote by $\Ch(\AAA)$ the category of chain complexes and morphisms of chain complexes in $\AAA$.
Chain complexes are stable under change of coefficients functors, and there are (co)cartesian morphisms of chain complexes:
\begin{definition}
The \look{restriction} along vertical arrows $f: A \to B$ and $g: C \to D$ of a chain complex $M_* : B \tobar D$ in $\AAA$ is the pointwise restriction of $M_*$. Since restriction is additive, this yields another chain complex, denoted $f^*Mg^*: A \tobar C$.
Likewise, we define \look{(co)extension} along vertical arrows and \look{horizontal composition} along $1$-cells.
\end{definition}
\begin{proposition}
For all chain complexes $M$ and $N$ in $\A(A,B)$ and $\A(C,D)$, respectively, and for all vertical arrows $f: A\to C$ and $g: B \to D$, there exists a morphism of chain complexes $\gamma: f^*Ng^* \cc{f}{g} N$, unique up to precomposition by a unique isomorphism.
This morphism is cartesian in the sense that for any morphism of chain complexes $\alpha: M \cc{f}{g} N$, there exists a unique morphism $\fact{\alpha}: M \to f^*Ng^*$ in $\Ch(\DD(A,B))$ such that $\alpha = \gamma \circ \fact \alpha$.
\end{proposition}
\begin{proof}
The cartesian $2$-cells $(\gamma_i: f^*N_ig^* \cc{f}{g} N_i)_{i\geq 0}$ automatically assemble into a morphism of chain complexes by definition of $f^*\partial_i g^*$ for $i\geq 0$. 
Given another morphism of chain complexes $\alpha: M \ccs{f}{g} N$, factor $\alpha$ pointwise along each cartesian $2$-cell $\gamma_i$ to obtain $2$-cells $\fact \alpha_i: M_i \to f^*N_ig^*$ for all $i\geq 0$. These again assemble into a morphism of chain complexes $\fact \alpha$ that satisfies the required property. Conversely, any other factorization must coincide with this one, by applying pointwise the uniqueness of each $\fact \alpha_i$. Similarly, to see that $\gamma$ is unique up to precomposition by a unique isomorphism, apply pointwise the uniqueness property of each $2$-cell $\gamma_i$.
\end{proof}
The horizontal category $\AAA_1$ is not necessarily abelian, nor is $\Ch(\AAA)$. However, for all objects $A$ and $B$, the categories of chain complexes $\Ch(\A(A,B))$ in the local categories $\A(A,B)$ are abelian.
Therefore, factorization through (co)cartesian morphisms is required to compute, for example, the cokernel of a morphism of chain complexes.
\begin{lemma}\label{lemma:cokernel}
Let $M$ and $N$ be chain complexes in $\A(A,B)$ and $\A(C,D)$, respectively, and let $\alpha: M \cc{f}{g} N$ be a morphism of chain complexes. Denote by $\fact{\alpha}: f_!Mg_! \to N$ the factorization of $\alpha$ through the cocartesian morphism $M \xrightarrow{\cocart} f_!Mg_!$. Then the cokernel of $\fact{\alpha}$ exists.
\end{lemma}
\begin{proof}
The category $\Ch(\A(C,D))$ is abelian.
\end{proof}
Fortunately, changing coefficients is not overly restrictive. 
For example, given a cone $(N: A \tobar B, (\phi_x)_{x\in I})$ of a diagram $F: I \to \D_1$, one can apply restriction of scalars functor to obtain a cone of the `restricted diagram' in the local category $\A(A,B)$.
This is the content of the following lemma.
\begin{lemma}\label{lemma:cone_change}
Let $(N: A \tobar B, (\phi_x)_{x\in I})$ be a cone of a diagram $F: I \to \D_1$. There exists a restricted diagram $\fact F: I \to \A(A,B)$ defined on objects $x$ in $I$ by $\fact F (x) = L(\phi_x)^*F(x)R(\phi_x)^*$. Furthermore, $N$ together with the factorizations of the $\phi_x$ for $x\in I$ is a cone of this diagram.
\end{lemma}
\begin{proof}
Define $\fact F$ on an arrow $f: x \to y$ by the factorization of $\fact F(x) \xrightarrow{\cart} F(x) \xrightarrow{F(f)} F(y)$ through $\fact F(y) \xrightarrow{\cart} F(y)$, as depicted in the diagram below. By uniqueness of the factorization, $\fact F$ is functorial. For all $x \in I$, we denote by $\fact \phi_x$ the factorization of $\phi_x$ through $\fact F(x) \xrightarrow{\cart} F(x)$. Then, $(N, (\fact \phi_x)_{x \in I})$ is a cone of $\fact F$, as shown for all arrows $x \to y$ in $I$ by the following commutative diagram:
\[
\begin{tikzcd}[sep=small]
	&& {\fact F(x)} && {F(x)} \\
	N \\
	&& {\fact F(y)} && {F(y)}
	\arrow["\cart", from=1-3, to=1-5]
	\arrow["{\fact F(f)}", from=1-3, to=3-3]
	\arrow["{F(f)}", from=1-5, to=3-5]
	\arrow["{\fact \phi_x}", from=2-1, to=1-3]
	\arrow["{\fact \phi_y}"', from=2-1, to=3-3]
	\arrow["\cart", from=3-3, to=3-5]
\end{tikzcd}
\]
where the left triangle commutes because the right square and the outer diagram do, and by uniqueness of factorization through the cartesian morphism $\fact F (y) \xrightarrow{\cart} F(y)$.
\end{proof}
\subsection{Homology}
Let $\C$ be any category, $\AAA$ an abelian framed bicategory, and $M: \C \to \Ch(\AAA)$ a functor. The image under $M$ of an object $X$ is denoted $M(X): L(X) \tobar R(X)$. 
That is, $M(X)$ is a chain complex of `$L(X)$-$R(X)$-bimodules'.
\begin{definition}\label{def:homology}
The \look{homology} of an object $X$ in $\C$ is the homology $\HM_*(X): L(X) \tobar R(X)$ of the chain complex $M_*(X)$ in the local abelian category $\A(L(X), R(X))$. This construction is functorial.
\end{definition}
\begin{example}\label{ex:homology_precubical}
Except for the shift-by-one, the homology theory of~\cite[Definition 5.4]{goubault2025directed}  corresponds in the framework of abelian framed bicategories to a functor $C: \Cub_I \to \Ch(\Bimod)$ from the category $\Cub_I$ of precubical sets with proper non-looping length covering, with maps which are injective on vertices.
\end{example}
Just as in abelian categories, one can consider the homology of an object $X$ in $\C$ relative to an arrow $f: Y \to X$. We denote the extension of scalars functors $L(f)_!(\mathunderscore)R(f)_!$ by ${}^{f}(\mathunderscore)$, or ${}^{X}(\mathunderscore)$ by when there is no ambiguity.
\begin{definition}\label{def:relative_homology}
Let $f: Y \to X$ be a morphism in $\C$. By Lemma~\ref{lemma:cokernel}, the cokernel $(M_i(X)/{}^X M_i(Y))_{i \geq 0}$ is well-defined. We denote this chain complex by $M(X,Y)$ and its homology by $HM(X,Y)$, and call it the \look{relative homology} of $f: Y \to X$.
\end{definition}
Observe that the cokernel $M(X,Y)$ is computed in the same local category $\A(L(X),R(X))$ as $M(X)$, whereas $M_*(Y)$ is defined in the local category $\A(L(Y), R(Y))$.
Since $\AAA_1$ is not assumed to be abelian, the notion of exact sequences in $\AAA_1$ is ill-defined. Hence, one cannot directly obtain a long exact sequence that computes relative homology in terms of the homology of $X$ and $Y$.   
Nevertheless, in Section~\ref{section:les_rh} we exhibit a long exact sequence involving extension of scalars that computes relative homology along certain well-behaved morphisms $f: Y \to X$, called relative pairs.

\subsection{Cohomology}

As in Definition~\ref{def:chaincomplex}, we can define cochains a and cohomology theory: 
\begin{definition}
\label{def:cochaincomplex}
A \look{cochain complex} in $\AAA$ is a cochain complex in a local abelian category of $\AAA$, i.e., a cochain complex in $\A(A,B)$ for some objects $A$ and $B$.  
A \look{morphism} $(\alpha_i)_{i\geq 0}$ from a cochain complex $(M_i)_{i \geq 0}$ in $\A(A,B)$ to a cochain complex $(N_i)_{i \geq 0}$ in $\A(C,D)$ is a pair of vertical arrows $f: A \to C$ and $g: B \to D$ together with $2$-cells $\alpha_i: M_i \cc{f}{g} N_i$ for all $i \geq 0$, such that the following diagram commutes for all $i \geq 0$:
\begin{equation*}
\begin{tikzcd}
	A && B \\
	\\
	C && D
	\arrow[""{name=0, anchor=center, inner sep=0}, "{M_{i}}"', "\shortmid"{marking}, curve={height=12pt}, from=1-1, to=1-3]
	\arrow[""{name=1, anchor=center, inner sep=0}, "{M_{i+1}}", "\shortmid"{marking}, curve={height=-12pt}, from=1-1, to=1-3]
	\arrow["f"', from=1-1, to=3-1]
	\arrow["g", from=1-3, to=3-3]
	\arrow[""{name=2, anchor=center, inner sep=0}, "{N_{i}}"', "\shortmid"{marking}, curve={height=12pt}, from=3-1, to=3-3]
	\arrow[""{name=3, anchor=center, inner sep=0}, "{N_{i+1}}", "\shortmid"{marking, text={rgb,255:red,128;green,128;blue,128}}, color={rgb,255:red,128;green,128;blue,128}, curve={height=-12pt}, from=3-1, to=3-3]
	\arrow["{ \partial_{i+1}}", shorten <=5pt, shorten >=5pt, Rightarrow, from=0, to=1]
	\arrow["{\alpha_{i+1}}"{description, pos=0.7}, shift left=5, color={rgb,255:red,128;green,128;blue,128}, shorten <=4pt, shorten >=4pt, Rightarrow, from=1, to=3]
	\arrow["{\alpha_{i}}"{description, pos=0.3}, shift right=5, shorten <=4pt, shorten >=4pt, Rightarrow, from=0, to=2]
	\arrow["{\partial_{i+1}}"', color={rgb,255:red,128;green,128;blue,128}, shorten <=5pt, shorten >=5pt, Rightarrow, from=2, to=3]
	\arrow[curve={height=12pt}, from=1-1, to=1-3]
\end{tikzcd}
    \end{equation*}
\end{definition}
We denote by $\coCh(\AAA)$ the category of cochain complexes and morphisms of cochain complexes in $\AAA$.
Cochain complexes are also stable under change of coefficients functors, and there are (co)cartesian morphisms of cochain complexes.

Similarly to the homological case, let $\C$ be any category, $\AAA$ an abelian framed bicategory, and $M: \C \to \coCh(\AAA)$ a functor. The image under $M$ of an object $X$ is denoted $M(X): L(X) \tobar R(X)$, i.e., $M(X)$ is a cochain complex of `$L(X)$-$R(X)$-bimodules'.
We can then define:

\begin{definition}\label{def:cohomology}
The \look{cohomology} of an object $X$ in $\C$ is the cohomology $\HM^*(X): L(X) \tobar R(X)$ of the cochain complex $M^*(X)$ in the local abelian category $\A(L(X), R(X))$. This construction is functorial.
\end{definition}

As in Definition~\ref{def:relative_homology}, relative cohomology can also be defined.

In many interesting cases, we can define cochain complexes from chain complexes and get cohomology theories that way. As in the homology case, we again consider a category $\C$, $\AAA$ an abelian framed bicategory and $M$ a functor $M: \C \to \Ch(\AAA)$.


In order to dualize chain complexes into cochain complexes, we are going to use dual pairs in $\AAA$. Recall \cite{shulman2007framed} that a 
dual pair in a double category $\D$ is a pair of $1$-cells $(M,N)$, with $M: \ A \tobar B$,
$N: \ B \tobar A$, together with `evaluation' and `coevaluation' $2$-cells
$N \odot M \rightarrow U_B$ and $U_A \rightarrow M \odot N$ 
satisfying the triangle identities. In that case, $N$ is called the right dual of $M$. 

Suppose now that $\AAA$ is a symmetric closed monoidal abelian framed bicategory, i.e., we have
$\AAA(M \odot N,P)\simeq \AAA(M,N \vartriangleright P)\simeq \AAA(N,P \vartriangleleft M)$ for all $1$-cells $M$, $N$ and $P$. 
We know again from \cite{shulman2007framed} that when the right dual of a $1$-cell $M: A \tobar B$ exists, it is always isomorphic to $M \vartriangleright U_B$. 
We now suppose that for all $X$ in $\C$, and $i \geq 0$, $M_i(X)$ is right dualizable, and we set $M^i(X)$ to be the dual of $M_i(X)$. Similarly, consider the dual globular $2$-cells $\partial^{i+1}: \ M^{i}(X)\rightarrow M^{i+1}(X)$ of the boundary operators of $M_*(X)$. It is immediate to see that this forms a cochain complex $M^*(X) : R($. 

\begin{example}
We consider the abelian framed bicategory $\Bimod$ of bimodules of algebras, see Example \ref{ex:bimod}, and the homological construction of~\cite{goubault2025directed}, see Example~\ref{ex:homology_precubical}. $\Bimod$ is symmetric and closed monoidal, therefore we can apply the above construction to any right dualizable chain complex. In~\cite{goubault2025directed}, to all precubical set $X$ is associated a chain complex $M(X)$ of $R(X)$-$R(X)$-bimodules, with $R(X)$ the path algebra of the underlying quiver of $X$. Recall that $M_0(X)$ is $U_{R(X)}$, i.e., $R(X)$ regarded as a $R(X)$-$R(X)$-bimodule, and that $M_i(X)$ is free for $i\geq 1$.
It is easy to see that they are both right dualizable and that the dual structure can be described as follows:

Let $X$ be a precubical set. 
Consider $R(X)^{\mathrm{op}}$ the opposite algebra of $R(X)$, i.e., the algebra which has the same underlying $R$-module as $R[X]$ but has as internal multiplication $\times^*$ the opposite of the internal multiplication $\times$ of $R(X)$: for $p$ and $q$ in $R(X)^{\mathrm{op}}$, $p\times^* q=q\times p$.
For all $i \geq 1$, let $M^i(X)$ be the $R$-module $\Hom_{{ }_R \mmod}(M_n(X),R)$. 
We define coboundary maps $\partial^{i+1}: M^i(X) \rightarrow M^{i+1}(X)$. For all $f \in M^i(X)$, let 
$\partial^{i+1}(f)$ be $f \circ \partial_{i+1}$, which is an element of $M^{i+1}(X)$ given $\partial_{i+1}:  M_{i+1}(X) \rightarrow M_i(X)$. 
%
Let us now define a $R(X)^{\mathrm{op}}$-$R(X)^{\mathrm{op}}$-bimodule structure on $M^*(X)$. For all $p,q$ in $R(X)^{\mathrm{op}}$, $f \in M^i(X)$ and $x \in M_i(X)$, $i \geq 0$, let $p\act f \act q (x)$ be $f(p\act x \act q)$. We have: 
\begin{align*}
p'\act (p\act f \act q)\act q'(x) & =  p\act f(p'\act x \act q') \act q\\
& =  f(p\act (p'\act x \act q')\act q) \\
& =  f((p\times p')\act x \act (q'\times q) )\\
& =  f((p'\times^* p)\act x \act (q\times^* q')) \\
& =  (p'\times^* p) \act f \act (q\times^* q') (x)
\end{align*}
and the coboundary map we defined is indeed a morphism of $R(X)^{\mathrm{op}}$-bimodules: 
\begin{align*}
\partial^{i+1}(p\act f \act q)(x) & =  (p\act f \act q)(\partial_{i+1}(x)) \\
& =  f(p\act \partial_{i+1}(x) \act q) \\
& =  f(\partial_{i+1}(p\act x \act q)) \\
& =  p\act (f \circ \partial_{i+1}) \act q(x) \\
& =  p \act \partial^{i+1}(f) \act q (x).
\end{align*}

As we know, $B^{\mathrm{op}}$-$A^{\mathrm{op}}$-bimodules can be equivalently seen as $A$-$B$-bimodules, therefore this construction makes $M^*(X)$ a $R(X)$-$R(X)$-bimodule, as expected. 
\end{example}

\section{Exact sequences through relative pairs}\label{section:exact_sequences}
In this section, we argue that in the setting of abelian framed bicategories, one can still obtain adaptations of the usual exact sequences involving extension of scalars. Long exact sequences for relative homology and \MV{} sequences require the notion of relative pairs. In the following, let $\AAA$ denote an abelian framed bicategory, $\C$ a category, and $M : \C \to \Ch(\AAA)$ a functor.
\subsection{Relative pairs}
The goal of this section is to generalize the notion introduced in~\cite[Definition 5.17]{goubault2025directed}. 
The original definition is combinatorial and induces a factorization property~\cite[Lemma 5.18]{goubault2025directed} that ensures the existence of certain exact sequences.
By abstracting proof~\cite[Theorem 8.18]{goubault2025directed}, we obtain the following two sufficient conditions.
\begin{definition}\label{def:relative_pair}
A \look{relative pair} $(X,Y)$ is an arrow $f: Y \to X$ in $\C$ such that:
\begin{enumerate}
    \item the extension of scalars functor $\lex X (\mathunderscore)$ is exact,\label{condition:1}
    \item the factorization  $\lex X M(Y) \to M(X)$ of the morphism of chain complexes $M(f): M(Y) \to M(X)$ through the cocartesian morphism $M(Y) \xrightarrow{\cocart} \lex X M(Y)$ is monic. \label{condition:2}
\end{enumerate}
\end{definition}
\begin{remark}
  To satisfy condition~\ref{condition:2}, it suffices that $f$ be monic and that both $M$ and the factorization of $2$-cells $M(Y) \ccs{L(f)}{R(f)} N$ through $M(Y) \xrightarrow{\cocart} \lex X M(Y)$ preserve monomorphisms. This last property and condition~\ref{condition:1} can be stated for any pair of vertical arrows $f: A \to C$ and $g: B \to D$ and any $1$-cell $M: A \tobar B$. This gives a notion of relative pair that is independent of the category $\C$.
\end{remark}
As in~\cite{goubault2025directed}, relative pairs compose, although the argument here is entirely different.
\begin{proposition}\label{prop:rel_pair_compose}
Let $g: Z \to Y$ and $f: Y \to X$ be relative pairs. Then $f\circ g: Z \to X$ is a relative pair, and there is a natural isomorphism $\lex {fg} (\mathunderscore) \cong \lex f (\lex g (\mathunderscore)): \A(L(Z), R(Z)) \to \A(L(X),R(X)) $.
\end{proposition}
\begin{proof}
By the pseudofunctoriality~\cite[\S 4]{shulman2007framed} of change of coefficient functors, we have $\lex {fg} (\mathunderscore) \cong \lex f (\lex g (\mathunderscore))$. Exact functors are stable under composition, hence $\lex {fg} (\mathunderscore)$ 
is exact. It remains to see that the factorization of $M(fg)$ is monic.
We have the following equality of $2$-cells, by functoriality of $M$, by relativeness of the pairs $(X,Y)$ and $(Y,Z)$, by definition of $\lex f \alpha$, and because extension $\lex f (\mathunderscore)$ preserves monomorphisms by exactness. Monomorphisms are drawn in blue.
\[\begin{tikzcd}[row sep=3em]
	{L(Z)} & {R(Z)} \\
	{L(X)} & {R(X)}
	\arrow[""{name=0, anchor=center, inner sep=0}, "{M(Z)}", "\shortmid"{marking}, from=1-1, to=1-2]
	\arrow["{L[fg]}"', from=1-1, to=2-1]
	\arrow["{R[fg]}", from=1-2, to=2-2]
	\arrow[""{name=1, anchor=center, inner sep=0}, "{M(X)}"', "\shortmid"{marking}, from=2-1, to=2-2]
	\arrow["{M(fg)}"{description}, shorten <=8pt, shorten >=8pt, Rightarrow, from=0, to=1]
\end{tikzcd}
=
\begin{tikzcd}[row sep=3em]
	{L(Z)} & {R(Z)} \\
	{L(Y)} & {R(Y)} \\
	{L(X)} & {R(X)}
	\arrow[""{name=0, anchor=center, inner sep=0}, "{M(Z)}", "\shortmid"{marking}, from=1-1, to=1-2]
	\arrow["{L[g]}"', from=1-1, to=2-1]
	\arrow["{R[g]}", from=1-2, to=2-2]
	\arrow[""{name=1, anchor=center, inner sep=0}, "{M(Y)}"{description}, from=2-1, to=2-2]
	\arrow["{L[f]}"', from=2-1, to=3-1]
	\arrow["{R[f]}", from=2-2, to=3-2]
	\arrow[""{name=2, anchor=center, inner sep=0}, "{M(X)}"', "\shortmid"{marking}, from=3-1, to=3-2]
	\arrow["{M(g)}"{description}, shorten <=8pt, shorten >=8pt, Rightarrow, from=0, to=1]
	\arrow["{M(f)}"{description}, shorten <=8pt, shorten >=8pt, Rightarrow, from=1, to=2]
\end{tikzcd}\]
\[
=
\begin{tikzcd}[row sep=3em]
	{L(Z)} & {R(Z)} \\
	{L(Y)} & {R(Y)} \\
	{L(Y)} & {R(Y)} \\
	{L(X)} & {R(X)} \\
	{L(X)} & {R(X)}
	\arrow[""{name=0, anchor=center, inner sep=0}, "{M(Z)}"{inner sep=.8ex}, "\shortmid"{marking}, from=1-1, to=1-2]
	\arrow[from=1-1, to=2-1]
	\arrow[from=1-2, to=2-2]
	\arrow[""{name=1, anchor=center, inner sep=0}, "{\lex g M(Z)}"{description}, from=2-1, to=2-2]
	\arrow[equals, from=2-1, to=3-1]
	\arrow[equals, from=2-2, to=3-2]
	\arrow[""{name=2, anchor=center, inner sep=0}, "{M(Y)}"{description}, from=3-1, to=3-2]
	\arrow[from=3-1, to=4-1]
	\arrow[from=3-2, to=4-2]
	\arrow[""{name=3, anchor=center, inner sep=0}, "{\lex f M(Y)}"{description}, from=4-1, to=4-2]
	\arrow[equals, from=4-1, to=5-1]
	\arrow[equals, from=4-2, to=5-2]
	\arrow[""{name=4, anchor=center, inner sep=0}, "{M(X)}"'{inner sep=.8ex}, "\shortmid"{marking}, from=5-1, to=5-2]
	\arrow["\cocart"{description}, draw=none, from=0, to=1]
	\arrow["\alpha"{description}, draw={rgb,255:red,92;green,92;blue,214}, between={0.2}{0.8}, Rightarrow, from=1, to=2]
	\arrow["\cocart"{description}, draw=none, from=2, to=3]
	\arrow["\beta"{description}, draw={rgb,255:red,92;green,92;blue,214}, between={0.2}{0.8}, Rightarrow, from=3, to=4]
\end{tikzcd}
=
\begin{tikzcd}[row sep=3em]
	{L(Z)} & {R(Z)} \\
	{L(Y)} & {R(Y)} \\
	{L(Y)\ } & {\ R(Y)} \\
	{L(X)} & {R(X)} \\
	{L(X)} & {R(X)}
	\arrow[""{name=0, anchor=center, inner sep=0}, "{M(Z)}"{inner sep=.8ex}, "\shortmid"{marking}, from=1-1, to=1-2]
	\arrow[from=1-1, to=2-1]
	\arrow[from=1-2, to=2-2]
	\arrow[""{name=1, anchor=center, inner sep=0}, "{\lex f M(Z)}"{description}, from=2-1, to=2-2]
	\arrow[from=2-1, to=3-1]
	\arrow[from=2-2, to=3-2]
	\arrow[""{name=2, anchor=center, inner sep=0}, "{\lex f (\lex g M(Z))}"{description}, from=3-1, to=3-2]
	\arrow[equals, from=3-1, to=4-1]
	\arrow[equals, from=3-2, to=4-2]
	\arrow[""{name=3, anchor=center, inner sep=0}, "{\lex f M(Y)}"{description}, from=4-1, to=4-2]
	\arrow[equals, from=4-1, to=5-1]
	\arrow[equals, from=4-2, to=5-2]
	\arrow[""{name=4, anchor=center, inner sep=0}, "{M(X)}"'{inner sep=.8ex}, "\shortmid"{marking}, from=5-1, to=5-2]
	\arrow["\cocart"{description}, draw=none, from=0, to=1]
	\arrow["\cocart"{description}, draw=none, from=1, to=2]
	\arrow["{\lex f \alpha}"{description}, draw={rgb,255:red,92;green,92;blue,214}, between={0.2}{0.8}, Rightarrow, from=2, to=3]
	\arrow["\beta"{description}, draw={rgb,255:red,92;green,92;blue,214}, between={0.2}{0.8}, Rightarrow, from=3, to=4]
\end{tikzcd}
\]
In addition, $M(fg)$ factors through the cocartesian morphism $M(Z) \xrightarrow{\cct} \lex X M(Z)$:
\begin{equation*}
\begin{tikzcd}[row sep=2.5em]
	{L(Z)} & {R(Z)} \\
	{L(X)} & {R(X)}
	\arrow[""{name=0, anchor=center, inner sep=0}, "{M[Z]}", "\shortmid"{marking}, from=1-1, to=1-2]
	\arrow["{L(fg)}"', from=1-1, to=2-1]
	\arrow["{R(fg)}", from=1-2, to=2-2]
	\arrow[""{name=1, anchor=center, inner sep=0}, "{M[X]}"', "\shortmid"{marking}, from=2-1, to=2-2]
	\arrow["{M(fg)}"{description}, shorten <=8pt, shorten >=8pt, Rightarrow, from=0, to=1]
\end{tikzcd}
=
\begin{tikzcd}[row sep=2.5em]
	{L(Z)} & {R(Z)} \\
	{L(X)} & {R(X)} \\
	{L(X)} & {R(X)}
	\arrow[""{name=0, anchor=center, inner sep=0}, "{M(Z)}", "\shortmid"{marking}, from=1-1, to=1-2]
	\arrow[from=1-1, to=2-1]
	\arrow[from=1-2, to=2-2]
	\arrow[""{name=1, anchor=center, inner sep=0}, "{\lex {fg} M(Z)}"{description}, from=2-1, to=2-2]
	\arrow[equals, from=2-1, to=3-1]
	\arrow[equals, from=2-2, to=3-2]
	\arrow[""{name=2, anchor=center, inner sep=0}, "{M(X)}"', "\shortmid"{marking}, from=3-1, to=3-2]
	\arrow["\cocart"{description}, draw=none, from=0, to=1]
	\arrow["\gamma"{description}, shorten <=8pt, shorten >=8pt, Rightarrow, from=1, to=2]
\end{tikzcd}
\end{equation*}
By a general result about fibrations, the composition of two cartesian morphisms is cartesian, and cartesian morphisms are unique up to composition by a unique isomorphism. Hence, by uniqueness of factorization along cartesian $2$-cells, 
$\gamma$ is the composition of $\beta$, $\lex X \alpha$, and of an isomorphism, all of which are monic $2$-cells. Therefore, $\gamma$ is monic.
\end{proof}
\begin{example}\label{ex:rel_pairs_1}
As shown in~\cite{goubault2025directed}, under the homology theory of Example~\ref{ex:homology_precubical}, for all inclusions of precubical sets $Y \hookrightarrow X$ such that the directed paths in $X$ can only enter $Y$ once, and exit $Y$ once, then $(X,Y)$ is a relative pair. 
A similar property holds for relative pairs of d-spaces, as shown in~\cite{goubault2026homological}.

This situation occurs for examples in Conley theory for dynamical systems: given an isolated invariant set $S$ in a flow $\varphi: \R \times X \to X$, consider an index pair $(N, L)$ for $S$. The flow $\varphi$ induces a structure of d-space on $X$, and therefore on the subspaces $N_1$ and $N_2$, with as directed paths the cloture under reparametrization of maps:
\begin{align*}
    p_{a,b,x}: [0,1] &\to X \\
    t &\mapsto \varphi(a + (b-a) t, x)
\end{align*}
for all points $x \in X$ and closed intervals $[a,b]$ in $\R$. By definition of an index pair, we have $L \subseteq N$. Moreover, for all directed path in $N$, $L$ is positively invariant in $N$ so a directed path in $N$ may enter $L$ at most once, and cannot exit.
This observation allows to consider the directed relative homology, as defined in~\cite{goubault2026homological}, of an index pair, mimicking the definition of the homological Conley index. Its properties will be studied elsewhere.
\end{example}
More generally, it is interesting to determine when the first condition of Definition~\ref{def:relative_pair} holds, i.e., when extension of scalars is exact.
\begin{example}\label{ex:dist_galois}
Consider the framed bicategory $\Repr$ of representations over posets, i.e., ${}_R\mmod$-distributors of the form $M : P \tobar \{*\}$, see Example~\ref{ex:dist}. For any Galois connection of posets $\alpha: C \to A \dashv \gamma : A \to C $, we already know that extension along $\alpha$ is right exact. Let us show that it is also left exact, hence exact. It suffices to show that extension along $\alpha$ is naturally isomorphic to restriction along $\gamma$, and more precisely that there are isomorphisms of special objects ${}_\alpha A \cong C_\gamma$ and $A_\alpha \cong {}_\gamma C$. This follows immediately from the definition of $U$ in terms of homsets and the adjunction $\alpha \dashv \gamma$.
\end{example}
\subsection{Long exact sequences for relative homology}\label{section:les_rh}
Almost by definition, there are long exact sequences for relative homology along relative pairs.
\begin{theorem}\label{th:exact_seq_rel_hom}
For all relative pairs $(X,Y)$, there is a long exact sequence: 
\begin{center}
    \begin{tikzcd}[arrow style=math font,cells={nodes={text height=2ex,text depth=0.75ex}}]
    & & 0 \arrow[draw=none]{d}[name=X,shape=coordinate]{} \\
       \HM_0(X,Y) \arrow[curarrow=X]{urr}{} 
       & \HM_{0}(X) \arrow[l] \arrow[draw=none]{d}[name=Y, shape=coordinate]{} & \arrow[l] \cdots \\
       \HM_{i}(X,Y) \arrow[curarrow=Y]{urr}{} & \HM_{i}(X) \arrow[l] \arrow[draw=none]{d}[name=Z,shape=coordinate]{} & \lex X \HM_i(Y) \arrow[l] \\
       \HM_{i+1}(X,Y) \arrow[curarrow=Z]{urr}{} & \HM_{i+1}(X) \arrow[l] & \cdots \arrow[l]
   \end{tikzcd}
\end{center}
\end{theorem}
\begin{proof}
Since $(X,Y)$ is a relative pair, the factorization $\lex X M(Y) \to M(X)$ is monic. In addition, $M(X,Y)$ is defined as a cokernel. Therefore, the following short sequence in $\A(L(X),R(X))$ is exact:
$$
0 \to \lex X M(Y) \to M(X) \to M(X,Y) \to 0 
$$
It induces a long exact sequence, which gives the required sequence after using the exactness of the extension functor $\lex X (\mathunderscore)$ to get the isomorphism $H(\lex X M(Y)) \cong \lex X \HM(Y)$. 
\end{proof}
Theorem~\ref{th:exact_seq_rel_hom} is a direct generalization of~\cite[Theorem 8.18]{goubault2025directed}.
\begin{example}[Example 8.19 of \cite{goubault2025directed}]
Consider $X=\overrightarrow{D}^n=\K^n$ the tensor product of $n$ copies of the precubical set $\K$ that has two vertices $0$ and $1$, and a unique 1-cell $a$ with $d^0(a)=0$ and $d^1(a)$. $\K$ ``represents'' a directed segment, and $\overrightarrow{D}^n$ is a precubical version of a directed $n$-disc. We write $Y=\overrightarrow{S}^{n-1}$ for the boundary of $\overrightarrow{D}^n$, which is a version of a directed $(n-1)$-sphere, or ``hollow $n$-hypercube".  

Let us illustrate the relative homology sequence for $(\overrightarrow{D}^n,\overrightarrow{S}^{n-1})$, when $n=2$. The homology between any pair of points in $\overrightarrow{D}^2$ and $\overrightarrow{S}^{1}$ is trivial except between the start and end points, $00$ and $11$ respectively. 
In that case, $\HM_0(\overrightarrow{D}^2)=R$, $\HM_1(\overrightarrow{D}^2)=R$, $\HM_0(\overrightarrow{S}^{1})=R^2$, $\HM_1(\overrightarrow{S}^{1})=R$ and it is easy to check that $\HM_0(\overrightarrow{D}^2,\overrightarrow{S}^{1})=0$, $\HM_1(\overrightarrow{D}^2,\overrightarrow{S}^{1})=R$ and the relative exact sequence of Theorem \ref{th:exact_seq_rel_hom} boils down to: 
\begin{center}
\begin{tikzcd}
0 \arrow[r] & R \arrow[r] & R \arrow[r] & R^2 \arrow[r] & R\arrow[r] & 0 \arrow[r] & 0
\end{tikzcd}
\end{center}
\end{example}
\subsection{Mayer–Vietoris sequence}
\label{sec:MayerVietoris}
To formulate our Mayer-Vietoris theorem, we need the following adapted definition of good cover. 
\begin{definition}\label{def:good_cover}
A \look{good cover} of an object $X$ in $\C$ is a commutative square of relative pairs as depicted below such that the canonical morphism of Proposition~\ref{prop:canonical_excision} below induces an isomorphism in homology $\lex X \HM(X_1, X_1 \cap X_2) \cong \HM(X, X_2)$, i.e. an excision property.
\begin{equation*}
\begin{tikzcd}[sep=scriptsize]
	{X_1 \cap X_2} & {X_1} \\
	{X_2} & X
	\arrow[from=1-1, to=1-2]
	\arrow[from=1-1, to=2-1]
	\arrow[from=1-2, to=2-2]
	\arrow[from=2-1, to=2-2]
\end{tikzcd}
\end{equation*}
\end{definition}
\begin{proposition}\label{prop:canonical_excision}
Consider a commutative square of relative pairs as in Definition~\ref{def:good_cover} above.
The factorization $\lex X M(X_1) \to M(X)$ of $M(X_1 \to X)$ descends to the quotient, i.e. there is a unique morphism $\lex X M(X_1, X_1 \cap X_2) \to M(X, X_2)$ such that the following diagram commutes.
\begin{equation}\label{diagram:canonical}
\begin{tikzcd}[sep=scriptsize]
	0 & {\lex X M(X_1 \cap X_2)} & {\lex X M(X_1)} & {\lex X M(X_1, X_1 \cap X_2)} & 0 \\
	0 & {\lex X M(X_2)} & {M(X)} & {M(X,X_2)} & 0
	\arrow[from=1-1, to=1-2]
	\arrow[from=1-2, to=1-3]
	\arrow[from=1-2, to=2-2]
	\arrow[from=1-3, to=1-4]
	\arrow[from=1-3, to=2-3]
	\arrow[from=1-4, to=1-5]
	\arrow["{\exists !}", dashed, from=1-4, to=2-4]
	\arrow[from=2-1, to=2-2]
	\arrow[from=2-2, to=2-3]
	\arrow[from=2-3, to=2-4]
	\arrow[from=2-4, to=2-5]
\end{tikzcd}
\end{equation}
\end{proposition}
\begin{proof}
Since relative pairs compose, and because the extension functor $\lex X (\mathunderscore): \A(L(X_1), R(X_1)) \to \A(L(X),R(X))$ is exact, each row composes to zero. Moreover, the leftmost square commutes by functoriality of $M$ and by the dual of Lemma~\ref{lemma:cone_change}. The result then follows by the universal property of the cokernel $\lex X M(X_1, X_1 \cap X_2)$, which is a cokernel since $\lex X (\mathunderscore)$ is exact, hence preserves cokernels.
\end{proof}
\begin{theorem}\label{th:mayer_vietoris}
Consider a good cover of an object $X$ in $\C$, as in Definition~\ref{def:good_cover}. Then, we have the following long exact sequence in homology:
\begin{center}
    \begin{tikzcd}[arrow style=math font,cells={nodes={text height=2ex,text depth=0.75ex}}, column sep=small]
    & & 0 \arrow[draw=none]{d}[name=X,shape=coordinate]{} \\
       \HM_1(X) \arrow[curarrow=X]{urr}{} 
       & \lex X \HM_{1}(X_1) \oplus \lex X \HM_1(X_2) \arrow[l] \arrow[draw=none]{d}[name=Y, shape=coordinate]{} & \arrow[l] \cdots \\
       \HM_{i}(X) \arrow[curarrow=Y]{urr}{} & \lex X \HM_{i}(X_1) \oplus \lex X HMi(X_2) \arrow[l] \arrow[draw=none]{d}[name=Z,shape=coordinate]{} & \lex X \HM_i(X_1 \cap X_2) \arrow[l] \\
       \HM_{i+1}(X) \arrow[curarrow=Z]{urr}{} & \lex X \HM_{i+1}(X_1) \oplus \lex X \HM_{i+1}(X_2) \arrow[l] & \cdots \arrow[l]
   \end{tikzcd}
\end{center}
\end{theorem}
\begin{proof}
Recall that by Proposition~\ref{prop:rel_pair_compose}, $(X,X_1 \cap X_2)$ is a relative pair, and thus that the extension of scalars functor $\lex X (\mathunderscore)$ commutes with homology. Hence, by applying Theorem~\ref{th:exact_seq_rel_hom} to the two exact rows of diagram~\ref{diagram:canonical}, we get the following two long exact sequences in homology:
\begin{equation*}
\begin{tikzcd}[ column sep=tiny,row sep=scriptsize]
	\cdots & {\lex X \HM_{i+1}(X_1, X_1 \cap X_2)} & {\lex {X} \HM_i(X_1 \cap X_2)} & {\lex X \HM_i(X_1)} & {\lex X \HM_i(X_1, X_1 \cap X_2)} & \cdots \\
	\cdots & {\HM_{i+1}(X, X_2)} & { \lex {X} \HM_i(X_2)} & {\HM_i(X)} & {\HM_i(X, X_2)} & \cdots
	\arrow[from=1-1, to=1-2]
	\arrow[from=1-2, to=1-3]
	\arrow[from=1-3, to=1-4]
	\arrow[from=1-4, to=1-5]
	\arrow[from=1-5, to=1-6]
	\arrow[from=2-1, to=2-2]
	\arrow[from=2-2, to=2-3]
	\arrow[from=2-3, to=2-4]
	\arrow[from=2-4, to=2-5]
	\arrow[from=2-5, to=2-6]
\end{tikzcd}
\end{equation*}
In addition, by the functoriality of this construction, see for example~\cite[Proposition 1.3.6]{kashiwara2013sheaves}, we get the following commutative diagram with exact rows. By definition of a good cover, some of the vertical maps are isomorphisms:
\begin{equation*}
\begin{tikzcd}[column sep=tiny,row sep=scriptsize]
	\cdots & {\lex X \HM_{i+1}(X_1, X_1 \cap X_2)} & {\lex {X} \HM_i(X_1 \cap X_2)} & {\lex X \HM_i(X_1)} & {\lex X \HM_i(X_1, X_1 \cap X_2)} & \cdots \\
	\cdots & {\HM_{i+1}(X, X_2)} & { \lex {X} \HM_i(X_2)} & {\HM_i(X)} & {\HM_i(X, X_2)} & \cdots
	\arrow[from=1-1, to=1-2]
	\arrow[""{name=0, anchor=center, inner sep=0}, draw=none, from=1-1, to=2-1]
	\arrow[from=1-2, to=1-3]
	\arrow[""{name=1, anchor=center, inner sep=0}, "\cong"', from=1-2, to=2-2]
	\arrow[from=1-3, to=1-4]
	\arrow[""{name=2, anchor=center, inner sep=0}, from=1-3, to=2-3]
	\arrow[from=1-4, to=1-5]
	\arrow[""{name=3, anchor=center, inner sep=0}, from=1-4, to=2-4]
	\arrow[from=1-5, to=1-6]
	\arrow[""{name=4, anchor=center, inner sep=0}, "\cong"', from=1-5, to=2-5]
	\arrow[""{name=5, anchor=center, inner sep=0}, draw=none, from=1-6, to=2-6]
	\arrow[from=2-1, to=2-2]
	\arrow[from=2-2, to=2-3]
	\arrow[from=2-3, to=2-4]
	\arrow[from=2-4, to=2-5]
	\arrow[from=2-5, to=2-6]
\end{tikzcd}
\end{equation*}
The result then follows from the algebraic \MV{} theorem~\cite[Lemma 4.1.8]{selick1997introduction}, see Appendix~\ref{annex:algebraic}.
\end{proof}
Theorem~\ref{th:mayer_vietoris} is also a direct generalization of~\cite[Theorem 8.25]{goubault2025directed}.

\section{Künneth and Eilenberg–Zilber theorems}\label{sec:Kunneth}

In this section, we present a Künneth theorem, involving restriction of scalars, for abelian framed bicategories $\AAA$ and functors $M : \C \to \Ch(\AAA)$. 
In contrast to the exact sequence theorems in Section~\ref{section:exact_sequences}, this theorem does not rely on relative pairs, nor does it immediately generalize the Künneth theorem from~\cite{goubault2025directed}.
Nevertheless, in the concrete setting of Example~\ref{ex:homology_precubical}, we show that our Künneth theorem applies to a different, albeit still restricted, set of precubical sets than~\cite[Theorem 8.14]{goubault2025directed}.

\subsection{A Künneth theorem}

We require the following hypotheses on $\AAA$ and $\C$:
\begin{itemize}
    \item To formulate our Künneth theorem, we need $\AAA$ and $\C$ to be monoidal, with products both denote $\otimes$. This induces a product for chain complexes in $\AAA$: for all chain complexes $C_*$ in $\AAA(A,B)$ and $D_*$ in $\AAA(C,D)$, $C \otimes D$ is the chain complex in $\AAA(A \otimes C, B \otimes D) $ with:
\begin{itemize}
    \item $(C \otimes D)_i = \bigoplus_{i = j+k}C_j \otimes D_k$ for all $i \geq 0$;
    \item $\partial_i : (C \otimes D)_i \to (C \otimes D)_{i-1}$ is defined on the summand $C_j\otimes D_k$ by $\partial_j \otimes \id_{D_k} + (-1)^j \id_{C_j} \otimes \partial_k$, for all $i > 0$.
\end{itemize}
Now, for all $X, Y$ in $\C$, we can consider the two chain complexes $M_*(X \otimes Y)$ and $M_*(X) \otimes M_*(Y)$.
\item To relate the two chain complexes above, we require natural transformations 
$l_{X,Y}: L(X\otimes Y) \to L(X) \otimes L(Y)$ and $r_{X,Y}: R(X\otimes Y) \to R(X) \otimes R(Y)$ at the vertical level, along with an \EZ{}-like result, i.e.,  a natural isomorphism $\HM_*(X\otimes Y) \cong l^*(H_*(M(X) \otimes M(Y)){r}^*$ at the horizontal level.
\item Finally, we require $\AAA$ to locally have enough projectives, enough injectives and coproducts in order to apply the algebraic Künneth theorem~\ref{th:algebraic_kunneth}, and we require $\AAA$ to be closed in order to use the above \EZ{} result. 
\end{itemize}
\begin{theorem}\label{th:kunneth}
For all objects $X, Y$ in $\C$, if the tensor product functor $\otimes: \A(L(X), R(X)) \times \A(L(Y), R(Y)) \to \A(L(X) \otimes L(Y), R(X) \otimes R(Y))$ satisfies the conditions of the algebraic Künneth theorem~\cite[Theorem 3.19]{fluch2004kunneth}~\ref{th:algebraic_kunneth}, then we have the following short exact sequence, where $\Tor_1$ denotes the first left derived functor of the tensor product functor:
$$ 0 \to l^*(HM(X) \otimes \HM(Y)){r}^* \to \HM(X \otimes Y) \to l^*\Tor_1(HM(X),HM(Y)){r}^* \to 0.$$
\end{theorem}

\begin{proof}
The algebraic Künneth theorem yields the following exact sequence:
$$ 0 \to \HM(X) \otimes \HM(Y) \to H_*(M(X) \otimes M(Y)) \to \Tor_1(\HM(X),\HM(Y)) \to 0.$$ 
The appropriate restriction of scalars functors can be applied on this sequence, preserving exactness. Finally, the middle term is replaced using the provided \EZ{} formula.
\end{proof}
An \EZ{}-type formula is typically stronger in the sense that it requires the existence of a chain homotopy equivalence, not merely an isomorphism in homology. In Section~\ref{section:ezf}, we show that for the concrete directed homology theory of~\cite{goubault2025directed}, the strong \EZ{} formula does not hold, whereas the weak version does. Nevertheless, as noted above, this weaker version suffices for our purposes.

\subsection{Eilenberg–Zilber formulas for precubical sets}\label{section:ezf}
In this section, we consider the concrete homology theory of~\cite{goubault2025directed}. In the framework of abelian framed bicategories, this corresponds to a functor $C: \Cub \to \Ch(\Bimod)$ from a certain category $\Cub$ of precubical sets, with tensor products the usual tensor product of precubical sets~\cite[Definition 8.11]{goubault2025directed}, into the monoidal abelian framed bicategory of bimodules over $R$-algebras, with $R$ a fixed commutative ring. 
In particular we do not use the shifted convention of~\cite{goubault2025directed}. 

In more detail, $C$ maps a precubical set $X$ to a chain complex $C(X)$ of $R(X)$-bimodules, where the $R$-algebra $R(X)$ denotes the path algebra of $X$. Recall that $R(X)$ is unital, with unit denoted $1$. The homology of $C(X)$ is denoted $H(X)$. For $C$ to be functorial, morphisms in $\Cub$ must be injective on vertices, see~\cite{goubault2024semi}. 
We take $\Cub$ to be the category of finite precubical sets without cycles, and morphisms of precubical sets which are injective on vertices.
We restrict the objects of $\Cub$ to cycle-less precubical set in order to get Lemma~\ref{lemma:eq_prec_chains}, see Remark~\ref{remark:funct_cycle}.  

The goal of this section is to show that in the setting of bimodules over $R$-algebras, the classical \EZ{} formula does not hold, but that a weaker version in homology does.
This implies a Künneth theorem, see~\ref{th:kunneth}, which applies to cases where the Künneth theorem from~\cite{goubault2025directed} does not, for example the second case of~\cite[Example 8.17]{goubault2025directed}.
Let us start with the following natural algebra homomorphism $h$, introduced in~\cite[Remark 8.15]{goubault2025directed}:
\begin{definition}
For all precubical sets $X,Y$, the $R$-algebra map $h_{X,Y}: R(X\otimes Y) \to R(X) \otimes R(Y)$ is defined on generators $(u,v) \in (X_0\times Y_1) \cup (X_1\times Y_0)$ of $R(X\otimes Y)$ by $(u, v)$.
\end{definition}
In the following, we denote by $h^*M$ the restriction on the left and right along $h_{X,Y}$ of a $R(X) \otimes R(Y)$-bimodule $M$. 
Therefore, we can consider for all precubical sets $X,Y \in \Cub$, the chain complex of $R(X\otimes Y)$-bimodules $h^*{C(X) \otimes C(Y)}$.
\begin{proposition}\label{prop:no_homotopy}
There is no natural chain homotopy equivalence $C(X \otimes Y) \cong h^* ({C(X) \otimes C(Y)})$ of modules over $R$-algebras
\end{proposition}
\begin{proof}
Consider the precubical sets $X = Y = \lineone$ and the tensor product $X \otimes Y$, i.e., the full square. Via restriction along $h$, all the bimodules involved are regarded as $R(X \otimes Y)$-bimodules.
A chain homotopy equivalence between $C(X\otimes Y)$ and $C(X)\otimes C(Y)$ is the data of chain maps $f: C(X)\otimes C(Y) \to C(X\otimes Y)$ and $g: C(X\otimes Y) \to C(X)\otimes C(Y)$ along with chain homotopies $gf \cong \id_{C(X) \otimes C(Y)}$ and $fg \cong \id_{C(X \otimes Y)}$.
The latter is the data of a bimodule map $\psi_0: C_0(X\otimes Y) \to C_1(X\otimes Y)$ such that $f_0 \circ g_0 - \id_{C_0(X\otimes Y)} = \partial_1 \circ \psi_0$, as seen in the diagram below:
\[\begin{tikzcd}
	0 & {C_0(X) \otimes C_0(Y)} & 0 & 0 \\
	0 & {C_0(X\otimes Y)} & {C_1(X\otimes Y)} & 0
	\arrow[from=1-2, to=1-1]
	\arrow["{f_0}", shift left=2, from=1-2, to=2-2]
	\arrow[from=1-3, to=1-2]
	\arrow[from=1-4, to=1-3]
	\arrow["{g_0}", shift left=2, from=2-2, to=1-2]
	\arrow[from=2-2, to=2-1]
	\arrow["{\psi_0}", curve={height=-6pt}, from=2-2, to=2-3]
	\arrow["{\partial_1}", from=2-3, to=2-2]
	\arrow[from=2-4, to=2-3]
\end{tikzcd}\]
$\psi_0$ sends $0$-chains to $1$-chains, i.e., to the unique $1$-chain in $C_1(X\otimes Y)$ up to scalar multiplication. $\psi_0$ must respect the action of $R(X\otimes Y)$ on $C_0(X\otimes Y)$ and $C_1(X\otimes Y)$. In particular, for all $0$-chains $c \in C_0(X\otimes Y)$, we have $c \act \psi_0(1) = \psi_0(c) = \psi_0(1) \act c$, hence $\psi_0(c)$ is $0$. Therefore, we must have $fg = \id_{C_0(X\otimes Y)}$, hence $C_0(X\otimes Y) \cong C_0(X) \otimes C_0(Y)$. However, $C_0(X\otimes Y)$ is freely generated by the eleven $0$-chains in the filled square $X \otimes Y$ while $(C(X)\otimes C(Y))_0 = C_0(X)\otimes C_0(Y)$ is freely generated by the nine pure tensors of the three $0$-chains in $X = Y$. This contradict the previous isomorphism of $R$-module.
\end{proof}
A convenient source of natural chain homotopy equivalences is the acyclic model theorem (see~\ref{section:amt}). In the following sections, we show that even if the acyclic model theorem cannot be directly generalized to modules over $R$-algebras by Proposition~\ref{prop:no_homotopy}, it still applies to our case when modules are regarded as $R$-modules. Moreover, we show that at the level of homology, the induced natural isomorphism respects the actions of the algebras.

Let us now show how to apply the acyclic model theorem to our situation, using wedges of cubes~\cite{ziemianski2020spaces} as models. Recall that we are working with the category $\Cub$, hence all the precubical sets we consider are cycle-free. 
\begin{definition}[\cite{ziemianski2020spaces}]
For all $n \in \N$, the \look{standard $n$-dimensional cube} is the precubical set $\mathrm C^n$ where for all $k \in \N$, $\mathrm C^n_k$ is the set of $k$-dimensional faces of the unit $n$-dimensional cube in $\R^n$, and where $d^\epsilon_i$ is the projection along the $i$-th coordinate onto the hyperplane $x_i = \epsilon$. We denote by $s$ and $e$ its endpoints.
\end{definition}

We now define the realization of a sequence $s = (n_1, \ldots, n_l) \in \N_*^l$, which is the wedge $s$-cubes in the sense of~\cite{ziemianski2020spaces}, along with $i$-precubical chains.

\begin{definition}
For all $i \geq 0$, and for a sequence $s = (n_1, \ldots , n_l)$ with $n_k \geq 1$ for $1 \leq k \leq l$ and \look{dimension} $\sum_{1 \leq k\leq l} n_k - l = i$, the \look{realization} of $s$, denoted $\rea s$, is the precubical set: $${\bigsqcup_{1 \leq k \leq l} \mathrm C^{n_k}}/\{d^1(\mathrm C^{n_k}) \sim d^0( \mathrm C^{n_{k+1}}), \ 1 \leq k < l \}.$$
We denote by $s$ the vertex $d^0(x)$ with $x \in \mathrm C^{n_1}_{n_1} \hookrightarrow {\rea s}_{n_1}$. Symmetrically, $e$ denotes the vertex $d^1(x)$ with $x \in \mathrm C^{n_l}_{n_l} \hookrightarrow {\rea s}_{n_l}$.
A \look{$i$-precubical chain} is the realization of a sequence of dimension $i$. The realization of an empty sequence is the point $X_0 =\{*\}$. For every cube chain $c$ in a precubical set $X$, the \look{realization} of $c$ is the realization of its type.
\end{definition}

\begin{example}
Consider the standard $0$, $1$ and $2$-dimensional cubes, along with the realization of the sequence $(1,2,2)$:
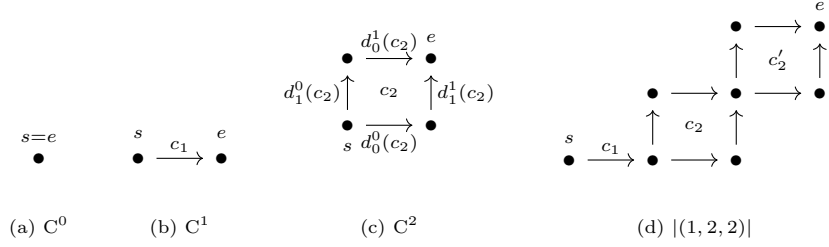
\vspace{-\baselineskip}
\begin{figure}[htbp]
    \centering
    \begin{subfigure}[b]{0.1\textwidth}
        \centering
\[\begin{tikzcd}
	\bullet
	\arrow["{s= e}"{description}, shift right=3, draw=none, from=1-1, to=1-1, loop, in=60, out=120, distance=5mm]
\end{tikzcd}\]
        \caption{$\mathrm C^0$}
    \end{subfigure}
    \begin{subfigure}[b]{0.19\textwidth}
        \centering
\[\begin{tikzcd}[sep=scriptsize]
	\bullet & \bullet
	\arrow["{s}"{description}, shift right=2.5, draw=none, from=1-1, to=1-1, loop, in=60, out=120, distance=5mm]
	\arrow["{c_1}", from=1-1, to=1-2]
	\arrow["{e}"{description}, shift right=2.5, draw=none, from=1-2, to=1-2, loop, in=60, out=120, distance=5mm]
\end{tikzcd}\]
        \caption{$\mathrm C^1$}
    \end{subfigure}
    \begin{subfigure}[b]{0.25\textwidth}
        \centering
\[\begin{tikzcd}[sep=scriptsize]
	\bullet & \bullet \\
	\bullet & \bullet
	\arrow["{d^1_0(c_2)}", from=1-1, to=1-2]
	\arrow["e"{description}, shift right=3, draw=none, from=1-2, to=1-2, loop, in=60, out=120, distance=5mm]
	\arrow["{d^0_1(c_2)}", from=2-1, to=1-1]
	\arrow["{c_2}"{description}, draw=none, from=2-1, to=1-2]
	\arrow["s"{description}, shift left=3, draw=none, from=2-1, to=2-1, loop, in=300, out=240, distance=5mm]
	\arrow["{d^0_0(c_2)}"', from=2-1, to=2-2]
	\arrow["{d^1_1(c_2)}"', from=2-2, to=1-2]
\end{tikzcd}\]
        \caption{$\mathrm C^2$}
    \end{subfigure}
    \begin{subfigure}[b]{0.4\textwidth}
        \centering
\[\begin{tikzcd}[sep=scriptsize]
	&& \bullet & \bullet \\
	& \bullet & \bullet & \bullet \\
	\bullet & \bullet & \bullet
	\arrow[from=1-3, to=1-4]
	\arrow["e"{description}, shift right=3, draw=none, from=1-4, to=1-4, loop, in=60, out=120, distance=5mm]
	\arrow[from=2-2, to=2-3]
	\arrow[from=2-3, to=1-3]
	\arrow["{c_2'}"{description}, draw=none, from=2-3, to=1-4]
	\arrow[from=2-3, to=2-4]
	\arrow[from=2-4, to=1-4]
	\arrow["s"{description}, shift right=3, draw=none, from=3-1, to=3-1, loop, in=60, out=120, distance=5mm]
	\arrow["{c_1}", from=3-1, to=3-2]
	\arrow[from=3-2, to=2-2]
	\arrow["{c_2}"{description}, draw=none, from=3-2, to=2-3]
	\arrow[from=3-2, to=3-3]
	\arrow[from=3-3, to=2-3]
\end{tikzcd}\]
        \caption{$\rea{(1,2,2)}$}
    \end{subfigure}
\end{figure}
\end{example}

Precubical chains can be regarded as cubical complexes in the sense of~\cite{dubut2017directed}. In particular, they are cycle-free.
Intuitively, the directed geometric realization of a precubical chain is a sequence of cubes where consecutive cubes $c$ and $c'$ are glued along their extremities $d^1(c)$ and $d^0(c')$.

In the following, we denote by $\mathcal M_i$ the set of $i$-precubical chains, and by $\mathcal M$ the set of all precubical chains. We show that cube chains in a precubical set $X$ are characterized by morphisms of precubical chains into $X$, and similarly for the tensor product of precubical sets. If $\mathrm C$ is a precubical chain, we denote by $C'(C)$ the chain complex of cube-chains in $\mathrm C$ whose endpoints are the endpoints $s$ and $e$ of $\mathrm C$.

\begin{lemma}\label{lemma:eq_prec_chains}
For every $i$-precubical chain $\mathrm C$, there is a unique cube chain in $C'_i(\mathrm C)$, and its realization is $\mathrm C$. Moreover, for all precubical sets $X \in \Cub$ and $i$-cube chains $c = (c_1, \ldots, c_l)$ of type $(n_1, \ldots, n_l)$ in $X$, there is a unique morphism of precubical sets $\rea c \to X$ in $\Cub$ mapping $x \in \mathrm C^{n_k}_{n_k} \hookrightarrow {\rea c}_{n_k}$ to $c_k \in X_{n_k}$ for all $1 \leq k \leq l$.\end{lemma}
\begin{proof}
Let $\mathrm C$ be the realization of a sequence $s = (n_1, \ldots , n_l)$ of dimension $i$. A cube chain in $C'(\mathrm C)$ is a concatenation of cube chains $(c_k)$ in $C'(\mathrm C^{n_k})$, $1 \leq k \leq l$. The dimension of such a cube chain is $\sum_{1 \leq k \leq l} \dim c_k$, and the maximum dimension of a cube chain $c_k$ is $n_k - 1$, attained by a unique cube chain: $(x)$ where $x$ is the only element of $\mathrm C^{n_k}_{n_k}$. Therefore, there is a unique choice of cube chains $c_k$, $1 \leq k \leq l$, yielding a cube chain in $C'(\mathrm C)$ of dimension $i$, and the type of this cube chain is $s$.

Moreover, for all precubical sets $X$ and $i$-cube chains $c = (c_1, \ldots, c_l)$ of type $(n_1, \ldots, n_l)$ in $X$, we show that mapping $x \in \mathrm C^{n_k}_{n_k} \hookrightarrow {\rea c}_{n_k}$ to $c_k \in X_{n_k}$ for all $1 \leq k \leq l$ completely defines a morphism $\varphi$ of precubical sets $\rea c \to X$. Indeed, for all $1 \leq k \leq l$ and $1 \leq i \leq n_k$, the definition of a morphism of precubical sets determines the value of $\varphi(x)$ for all $x \in \mathrm C^{n_k}_i \hookrightarrow \rea c$. For $i=0$, we must additionally check that $d^1(c_k) = d^0(c_{k+1})$ for all $1 \leq k < l$, which holds by definition of a cube chain. In addition, since $X$ does not have cycles, then the vertices $(d^0(c_1), \ldots, d^0(c_l), d^1(c_l))$ in $c$ are distinct from each others. Hence, the cube chain map we have constructed is injective on vertices.
\end{proof}

\begin{remark}\label{remark:funct_cycle}
The acyclic models theorem applies to functors. To guarantee functoriality, we restricted the morphisms of precubical sets in $\Cub$ to those that are injective on vertices. Therefore, we had to restrict the objects of $\Cub$ to cycle-free precubical sets in order for the above lemma to hold. 
\end{remark}

\begin{lemma}\label{lemma:bij_cube_chain1}
For all $i \geq 0$ and precubical sets $X \in \Cub$, the function $\varphi$ mapping morphisms $f: \mathrm C \to X$ in $\Cub$, $\mathrm C \in \mathcal M_i$, to $C_i(f)(c)$ in $C_i(X)$, where $c$ is the unique cube chain in $C'_i(\mathrm C)$, is a bijection.
\end{lemma}
\begin{proof}
Let $\psi$ be the function mapping $i$-cube chains $c$ in $X$ to the morphism $f: \rea c  \to X$ from Lemma~\ref{lemma:eq_prec_chains}. By Lemma~\ref{lemma:eq_prec_chains}, $\psi$ and $\varphi$ are inverses of each other.
\end{proof}

\begin{lemma}\label{lemma:bij_cube_chain2}
For all $i \geq 0$ and precubical sets $X, Y \in \Cub$, the function mapping triples $(c, f: \mathrm C \to X, g: \mathrm C' \to Y)$ with $c \in C_i(\mathrm C \otimes \mathrm C')$ a dimension $i$ shuffling of a $j$-cube chain in $C'_j(\mathrm C)$, $\mathrm C \in \mathcal M_j$ and a $k$-cube chain $C'_k(\mathrm C')$, $\mathrm C' \in \mathcal M_k$, 
to $C_i(f\otimes g)(c)$ in $C_i(X \otimes Y)$, is a bijection.
\end{lemma}
\begin{proof}
Recall that the dimension of a shuffle of two cube chains $c$ and $c'$ is greater or equal to $\dim(c) + \dim(c')$. Therefore, the triples of the statement must satisfy $i \geq j + k$.
For all precubical sets $X$ and $Y$, an $i$-cube chain $c \in C_i(X \otimes Y)$ 
is a unique shuffle of a $j$-cube chain $c_X$ in $X$ and a $k$-cube chain $c_Y$ in $Y$, with $i = j + k$. By Lemma~\ref{lemma:eq_prec_chains}, $c_X$ and $c_Y$ induce unique morphisms $f: \mathrm C \to X$ and $g: \mathrm C' \to Y$ in $\Cub$ for some $\mathrm C \in \mathcal M_j$ and $\mathrm C' \in \mathcal M_k$.
Consider the unique cube chains $c_{\mathrm C} \in C'_j(\mathrm C)$ and $c_{\mathrm C'} \in C'_k(\mathrm C')$. By definition, $c_{\mathrm C}$ and $c_X$ have the same type, and similarly for $c_{\mathrm C'}$ and $c_Y$. Therefore, we can shuffle $c_{\mathrm C}$ and $c_{\mathrm C'}$ as in $c$, yielding a cube chain $c' \in C_i(\mathrm C \otimes \mathrm C')$ of the same type as $c$, with $C_i(f\otimes g)(c') = c$.
\end{proof}

Consider the models $\mathcal M \times \mathcal M$ in $\C$, along with the functors $F$ and $G$:\begin{align*}
F: \Cub^2 &\to \Ch({}_R \mmod) & G: \Cub^2 &\to \Ch({}_R \mmod)\\
(X,Y) &\mapsto C_*(X \otimes Y) & (X,Y) &\mapsto C_*(X) \otimes C_*(Y).
\end{align*}
To prove the existence of a natural chain homotopy equivalence $C_*(X \otimes Y) \simeq C_*(X) \otimes C_*(Y)$ of chain complexes of $R$-modules, it suffices to show that $F$ and $G$ are both free and acyclic on $\mathcal M\times \mathcal M$, and to exhibit a natural isomorphism $H_0(F) \cong H_0(G)$.

\begin{proposition}
The functors $F,G: \Cub^2 \to \Ch({}_R\mmod)$ are free and acyclic on $\mathcal M$.
\end{proposition}
\begin{proof}
Let us show that $F$ is free.
We consider all the elements in $\mathcal M^2$ ($A = \mathcal M^2$).
Fix $i\geq 0$.
For all precubical chains $\mathrm C \in \mathcal M_j, \mathrm C' \in \mathcal M_k$, we let $B^i_{\mathrm C, \mathrm C'}$ be the set of dimension $i$ shufflings $c \in C_i(\mathrm C \otimes \mathrm C')$ of $j$-cube chains in $C'_j(\mathrm C)$ and $k$-cube chains in $C'_k(\mathrm C')$.
For all precubical sets $X,Y \in \Cub$, $C_i(X  \otimes Y)$ is freely generated by the $i$-cube chains in $X \otimes Y$, hence Lemma~\ref{lemma:bij_cube_chain2} concludes.
For the freeness of $G$, we keep $A = \mathcal M^2$. Fix $i\geq 0$. 
For all precubical chains $\mathrm C \in \mathcal M_j, \mathrm C' \in \mathcal M_k$, we let $B^i_{\mathrm C, \mathrm C'}$ be the set of pure tensors in $(C'_*(\mathrm C) \otimes C'_*(\mathrm C'))_i$.
For all precubical sets $X,Y \in \Cub$, $(C_*(X) \otimes C_*(Y))_i$ is freely generated by pure tensors of a $j$-cube chain in $X$ and a $k$-cube chain in $Y$ with $i = j + k$. Applying Lemma~\ref{lemma:bij_cube_chain1} twice yields the result.

Since a precubical chain can be regarded as a cubical complex, by~\cite[Lemma 7.1]{goubault2025directed}, its homology is isomorphic to that of its trace space. These trace spaces are contractible, hence the homology is concentrated in degree $0$. Moreover, the Künneth formula~\cite[Theorem 8.14]{goubault2025directed} holds so for all precubical chains $\mathrm C, \mathrm C'$, we have that $H_*(\mathrm C \otimes \mathrm C')$ is isomorphic to $H_*(\mathrm C) \otimes H_*(\mathrm C')$, 
hence also concentrated in degree $0$. Thus $F$ is acyclic. Similarly, the classical Künneth theorem yields $H_i(C_*(\mathrm C) \otimes C_*(\mathrm C')) \cong 0$ for all $i \geq 1$ and $\mathrm C,\mathrm C' \in \mathcal M$, hence $G$ is acyclic.\end{proof}

It then suffices to construct a natural isomorphism $H_0(X \otimes Y) \to H_0(C_*X \otimes C_*Y)$ of $R$-modules to obtain an \EZ{} formula.

\begin{proposition}\label{prop:exists_iso_deg0}
For all precubical sets $X,Y \in \Cub$, there is an isomorphism $H_0(X \otimes Y) \cong H_0(C_*X \otimes C_*Y)$, natural in $X$ and $Y$.\end{proposition}
\begin{proof}
Consider the map $f: C_0(X \otimes Y) \to C_0(C_*X \otimes C_*Y)$ defined on cube chains
$((u_1,v_1), \ldots, (u_l,v_l))$ by $(u_{i_1},\ldots, u_{i_n}) \otimes (v_{j_1}, \ldots, v_{j_m})$ where $i_1 < \cdots < i_n$ (resp. $j_1 < \cdots < j_m$) is the ordered set of indices $1 \leq k \leq l$ such that $u_k \in X_1$ (resp. $j_k \in Y_1$).
and extended by linearity. 
Similarly, consider the map $ g : C_0(C_*X \otimes C_*Y) \to C_0(X \otimes Y)$, defined on pure tensors $(u_1,\cdots, u_k) \otimes (v_1, \cdots, v_l)$ by: $$((u_1,d^0v_1), \ldots, (u_k,d^0v_1), (d^1u_k,v_1), \ldots (d^1u_k,v_l)).$$
Intuitively, $f$ separates a $0$-cube chain into its $X$ and $Y$ components, while $g$ maps a pair of $0$-cube chains to their trivial shuffle (all $X$ cubes before all $Y$ cubes). We show $f$ and $g$ induce maps $\bar f$ and $\bar g$ in homology. It suffices to check they send boundaries to boundaries.
A boundary in $C_0(X \otimes Y)$ is generated by elements of the form: 
\begin{equation}
    b = \partial((u_1,v_1), \ldots, (v,s),\ldots, (u_k,v_k))\label{eq:generator1}
\end{equation}
for some square $s \in Y_2$ and vertex $v \in X_0$, of the symmetric form for $X$, or of the form: 
\begin{equation}
    b = \partial((u_1,v_1), \ldots, (e,e'),\ldots, (u_k,v_k))\label{eq:generator2}
\end{equation}
for some edges $e \in X_1$ and $e' \in Y_1$.
In case~\eqref{eq:generator2}, $f(b) = 0$. In case~\eqref{eq:generator1}, $f(b)$ is the boundary of $f_1((u_1,v_1), \ldots, (v,s),\ldots, (u_k,v_k))$ where $f_1: C_1(X\otimes Y) \to C_1(C_*X \otimes C_*Y)$ is defined similarly to $f$, or symmetrically. 
For $g$, boundaries in $C_0(C_*X \otimes C_*Y)$ are generated by elements of the form $c  \otimes \partial c'$ for a $0$-cube chain $c$ in $X$ and a $1$-cube chain $c'$ in $Y$, or symmetrically. These are sent by $g$ to elements of the form~\eqref{eq:generator1} (or symmetric).
To see that $\bar f$ and $\bar g$ are inverse, it suffices to show that in $H_0(X \otimes Y)$, any two shuffles of $0$-cube chains $(u_1,\cdots, u_k)$ and $(v_1, \cdots, v_l)$ are in the same homology class. In $C_0(X \otimes Y)$, shuffles are related by finite sequences of \look{swaps}:
\begin{definition}
    For all precubical sets $X,Y \in \Cub$ and $0$-cube chain in $C_0(X \otimes Y)$ of the form:
    \begin{equation}
        ((u_1,v_1), \ldots, (e, v'), (v,e'),\ldots, (u_k,v_k)) \label{eq:form1}
    \end{equation}
    for some vertices $v \in X_0$, $v' _in Y_0$ and edges $e\in X_1$, $e' \in Y_1$, or symmetrically, 
    the following $0$-cube chain in $C_0(X \otimes Y)$, or its symmetric, is called a \look{swap} of $c$:
     $$((u_1,v_1), \ldots,(v,e'),(e, v'),\ldots, (u_k,v_k)).$$
\end{definition}
To complete the proof, observe that swapping preserves homology classes: for any $0$-cube chain $c$ of the form~\eqref{eq:form1}, the difference between $c$ and its swap is a boundary of the form~\eqref{eq:generator2}.
Finally, the naturality of $\bar f$ and $\bar g$ is straightforward.
\end{proof}

We can now apply the acyclic model theorem for modules over commutative rings to obtain a natural chain homotopy equivalence $\alpha_{X,Y}: C_*(X \otimes Y) \rightleftarrows C_*(X) \otimes C_*(Y): \beta_{X,Y}$ of chain complexes of $R$-modules. In particular, we obtain a natural isomorphism $\bar \alpha: H_*(X \otimes Y) \cong H_*(C(X) \otimes C(Y))$ of $R$-modules at the homology level, with $\bar \beta$ its inverse.
This raises the following question: even though such natural chain maps $\alpha$ and $\beta$ do not necessarily respect the $R$-algebra action (Proposition~\ref{prop:no_homotopy}), does it hold at the level of homology? We investigate this in our concrete case.
By following the proof of the acyclic models theorem, we can explicitly construct such natural chain maps $\alpha, \beta$ and chain homotopies $\alpha\beta \simeq \id$ and $\beta\alpha \simeq \id$.

\begin{proposition}
    Consider the natural transformation $\bar f$ of Proposition~\ref{prop:exists_iso_deg0}.
    The acyclic model theorem yields a unique natural chain map $\alpha_{X,Y}: C_*(X \otimes Y) \to C_*(X) \otimes C_*(Y)$ of chain complexes of $R$-modules.\end{proposition}
\begin{proof}
Let us consider $\alpha$.
For $i=0$, $\alpha_0$ is determined by its value on $0$-cube chains $c$ in $B^0_{\mathrm C, \mathrm C'}$, with $\mathrm C, \mathrm C' \in \mathcal M_0$. For each such $c$, $\alpha_0(c)$ must be a representative of the homology class $\bar f(c)$ in $G_0(\mathrm C, \mathrm C')$. There is only one such element: $f(c)$. Hence, $\alpha_0 = f$.
For $i > 0$, let us show by induction that for all $i$-cube chain $c$ in $B^i_\alpha$, if $c$ contains a cube of the form $(u,v)$ with $\dim(u), \dim(v) \geq 1$, then $\alpha_i(c) = 0$, and otherwise $\alpha_i((u_1,v_1), \ldots, (u_k,v_k))$ is: $$ (u_1,\cdots, u_k \ | \ 1 \leq i \leq k, u_i \in X_{\geq 1}) \otimes (v_1, \cdots, v_k \ | \ 1 \leq i \leq k, v_i \in Y_{\geq 1}).$$ 
For $i = 1$, given a $1$-cube chain $c \in B^1_\alpha$ of the first type, $f(\partial c) = 0$, so $\alpha(c)$ is the unique element in $G_1(M_\alpha)$ with boundary $0$. The reasoning is similar for the second type.
For $i > 1$, assume the result holds for $i-1$. Let $c$ be an $i$-cube chain in $B^i_\alpha$ of the first type, of the form $((u_1,v_1), \ldots, (u, v),\ldots, (u_k,v_k))$ with $\dim(u), \dim(v) \geq 1$. Using the definition of $\partial c$ from~\cite[Definition 2.11]{goubault2025directed}, we observe that $\partial c$ is a linear combination of $(i-1)$-cube chains of the first type (which are mapped by $\alpha_{i-1}$ to $0$), and up to a sign, of:
\begin{align*}
&((u_1,v_1), \ldots, (u, d^0(v),(d^1(u),v),\ldots, (u_k,v_k))\\ &- ((u_1,v_1), \ldots, (d^0(u), v),(u, d^1(v)), \ldots, (u_k,v_k)),
\end{align*}
which is also mapped by $\alpha_{i-1}$ to $0$.
The reasoning for elements of the second type is similar.
\end{proof}

Given this explicit formulation of $\alpha$, we observe that for all precubical sets $X,Y \in  \Cub$, $\alpha_{X,Y}$ respects the action of $R(X \otimes Y)$, not just in homology. $\beta$, however, is not uniquely defined by the acyclic models theorem. Nevertheless, in homology, $\bar \beta$ is. 
We will show that in homology, $\bar \beta_{X,Y}$ also respects the action of $R(X \otimes Y)$ by applying the following result to a specific $\beta$.

\begin{proposition}\label{prop:action_iso}
    Let $f: C \rightleftarrows C': g$ be chain maps of chain complexes of $R$-modules such that in homology, $\bar f$ and $\bar g$ are inverse to each other. Suppose $C$ and $C'$ are chain complexes of left $A$-modules, for an $R$-algebra $A$. Then the homology modules inherit a left $A$-module structure. If for all $i \geq 0$, $x \in \ker(\partial_i: C'_i \to C'_{i-1})$, and $a \in A$, we have $f_i\circ g_i(a \act x) = f_i(a \act g_i(x))$, then $\bar g$ respects the action of $A$.\end{proposition}
\begin{proof}
    For all $i \geq 0$, $x \in H_i(C)$, and $a \in A$, since $\bar f$ and $\bar g$ are inverses, it suffices to show $\bar g_i\circ \bar f_i(a \act x) = \bar g_i(a \act \bar f_i(x))$. This follows immediately from the hypothesis.
\end{proof}

We now show how to obtain a natural chain map $\beta_{X,Y}: C_*(X) \otimes C_*(Y) \to C_*(X \otimes Y)$ of $R$-modules via the acyclic model theorem that satisfies the conditions of Proposition~\ref{prop:action_iso}.

\begin{proposition}
    Consider the natural transformation $\bar g$ of Proposition~\ref{prop:exists_iso_deg0} and the natural chain map $\beta_{X,Y}: C_*(X) \otimes C_*(Y) \to C_*(X \otimes Y)$ defined on pure tensors $(u_1,\cdots, u_k) \otimes (v_1, \cdots, v_l) \in (C_*(X) \otimes C_*(Y))_i$ by:
    $$((u_1,d^0v_1), \ldots, (u_k,d^0v_1), (d^1u_k,v_1), \ldots (d^1u_k,v_l)) \in C_i(X \otimes Y)$$ for all $i \geq 0$.
    Then $\beta$ can be constructed by applying the proof of the acyclic model theorem to $\bar g$. Moreover, $\beta$ satisfies the conditions of Proposition~\ref{prop:action_iso}.
\end{proposition}
\begin{proof}
We show the first statement by induction. For $i=0$, for all $0$-cube chains $c \in C'_0(\mathrm C)$ and $c' \in C'_0(\mathrm C')$ with $\mathrm C, \mathrm C' \in \mathcal M_0$, $\beta_0(c\otimes c')$ is indeed a representative of the homology class $\bar g(c\otimes c')$ in $F_0(\mathrm C, \mathrm C')$. For $i>0$, suppose $\beta_0, \ldots, \beta_{i-1}$ have been constructed during the first $i$ steps of the acyclic model proof. Then for all dimension $i$ shufflings of cube 
chains $c \in C'_j(\mathrm C)$ and $c' \in C'_k(\mathrm C')$ with $\mathrm C \in \mathcal M_j$, $\mathrm C' \in \mathcal M_k$, we indeed have:
$$\partial \beta_i(c \otimes c') = \beta_{i-1}(\partial_j(c) \otimes c' + (-1)^j c \otimes \partial_k(c')) = \beta_{i-1}(\partial(c \otimes c')).$$

For the second statement, consider a $j$-cube chain $c$ in $X$, a $k$-cube chain $c'$ in $Y$ with $i= j+k$, and a generator $a = (u,v) \in R(X \otimes Y)$ with $\dim u + \dim v = 1$. The only case where we do not directly have $\beta_i(a \act (c \otimes c'))= a \act \beta_i(c \otimes c')$ is when $a = (x, v)$ with $x \in X_0$ and $v \in Y_1$. In that case, we have:
$$\alpha_i(\beta_i(a \act (c \otimes c'))) = c \otimes (v \act c') = \alpha_i(a \act \beta_i(c \otimes c')).$$ This concludes the proof.
\end{proof}
We have just proved that for all precubical sets $X,Y \in \Cub$, $\bar \beta$ respects the left action of $R(X\otimes Y)$. Similarly, using the symmetric version of Proposition~\ref{prop:action_iso}, we can show $\bar \beta$ also respects the right action of $R(X\otimes Y)$.
Assembling the previous facts, and using that restriction of scalars commutes with homology, we have shown a weak \EZ{} formula at the level of homology:
\begin{theorem}\label{th:weak_ez}
    For all precubical sets $X,Y \in \Cub$ and $i \geq 0$, there is an isomorphism $H_i(X \otimes Y) \cong h^*H_i(C(X) \otimes C(Y))$ of $R(X \otimes Y)$-bimodules, natural in $X$ and $Y$.
\end{theorem}
From this point onward, suppose $R$ is a field, and let us show that the conditions for applying the Künneth theorem~\ref{th:kunneth} hold. First,
The abelian framed bicategory of modules over $R$-algebras $\Bimod$ is indeed closed, and locally has enough injectives, enough projectives and coproducts. 
Then, for all precubical sets $X,Y \in \Cub$, let $\A$, $\B$ and $\C$ the local categories of $R(X)$, $R(Y)$ and $R(X)$-bimodules, since $\Bimod$ is externally closed in the sense of~\cite[Definition 9.9]{shulman2007framed}, then the external tensor product functor $\otimes: \A \times \B \to \C$ is right exact. 
Finally, it suffices to show the sufficient condition of Proposition~\ref{prop:sufficient_kunneth}.
However, since $R$ is a field, and since external tensor product of bimodules over algebras is the tensor product over $R$ of the bimodules regarded as $R$-modules, equipped with the component-wise action, its first derived functor is trivial. 

Therefore, for all precubical sets $X, Y \in \Cub$, i.e., precubical sets without cycles, by~\ref{th:kunneth} we have  $H_*(X \otimes Y) \cong h^*(H_*(X) \otimes H_*(Y))$. This applies in particular to the second case of~\cite[Example 8.17]{goubault2025directed}.
\begin{example}\label{ex:circle}
Consider the precubical set $\vec{S}^1$: 
\[
\begin{tikzcd}
1 \arrow[r,bend left,"\alpha"] \arrow[r,bend right,swap,"\beta"] & 2
\end{tikzcd}
\]
$\vec{S}^1$ is not a cubical complex, but it is cycle-free. Therefore we can apply our Künneth formula to get:  $$H_1(\vec{S}^1 \otimes \vec{S}^1) \cong H_0(\vec{S}^1) \otimes H_1(\vec{S}^1) \oplus H_1(\vec{S}^1) \otimes H_0(\vec{S}^1) = 0.$$
\end{example}

Theorem~\ref{th:weak_ez} raises the following question: for all functors $F, G: \C \to \Ch(\Bimod)$ such that $F$ and $G$ regarded as functor $\C \to {}_R \mmod$ are free and acyclic on some models $\mathcal{M}$, respectively. Now, when regarding bimodules as $R$-modules, by the acyclic model theorem, for all natural transformations $\overline \tau_0: H_0(F) \rightarrow H_0(G)$ there is a natural transformation $ \tau : F \rightarrow G$ inducing it. If in addition $\overline \tau_0: H_0(F) \rightarrow H_0(G)$ respects the actions of the algebras, under which additional conditions does the natural transformation $\overline \tau$ induced by $\tau$ in homology also respects these actions ?

\section{Embedding theorems}\label{section:embedding_theorems}

The Gabriel and \FM{} theorems state that under certain conditions, an abelian category is equivalent to a concrete category of modules over a fixed ring. Not every abelian category is a concrete category whose objects have underlying sets, but nevertheless, under the conditions specified by these theorems, element-wise reasoning remains possible.

In this section, we investigate the interactions of these embedding theorems with the change of coefficient functors in an abelian framed bicategory, culminating in an embedding theorem for abelian framed bicategories. Like the original embedding theorems, ours permits element-wise reasoning, but is primarily a result about the structure of abelian framed bicategories. The key notion in the proof of these embedding theorems is that of a compact projective generator.

\begin{definition}[{\cite[Definitions 8.4.1, 5.2.1, 6.3.3]{kashiwara2006categories}}]\label{def:cpg}
    An object $P$ in an abelian category $\A$ is called \look{projective} (resp. \look{generator}, \look{compact}) whenever the functor $\phi_P: \Hom(P, \mathunderscore): \A \to \set$ is exact (resp. conservative, preserves filtered colimits).
\end{definition}

Both the Gabriel and the \FM{}  theorems make use of functors of the form $\phi_P = \Hom(P, \mathunderscore) : \A \to \set $ where $\A$ is an abelian category and $P$ an object of $\A$. Since $\A$ is abelian, the homset $\Hom(P,P)$ inherits a ring structure, and $\phi_P$ is promoted to a functor $\A \to {}_{\End(P)}\mmod$ into left $\End(P)$-modules where $\End(P)$ denotes ${\Hom(P,P)}^{\mathrm{op}}$. Explicitly, for an object $X$ in $\A$, the left action of $\End(P)$ on a set $\Hom(P,X)$ is given by precomposition.
However, this functor is not necessarily exact, fully faithful, or an equivalence.

\begin{theorem}[Gabriel~\cite{gabriel1972unzerlegbare}]\label{th:gabriel}
Let $\A$ be a cocomplete abelian category. If $\A$ has a compact projective generator $P$, then $\phi_P$ is an equivalence of categories.
\end{theorem}

\begin{theorem}[\FM{}~\cite{mitchell1964full}]\label{th:freyd-mitchell}
Let $\A$ be a small abelian category. Then $\A$ is equivalent to a fully abelian subcategory of the category ${}_R\mmod$ of left $R$-modules, i.e., there exists a ring $R$ and an exact fully faithful functor $\A \to {}_R\mmod$. 
\end{theorem}

\begin{proof}[Proof sketch~\cite{kashiwara2006categories}.]
If $\A$ is a small abelian category, then there exists an exact fully faithful embedding $\A \to \Pro(\A)$ of $\A$ into the category of pro-objects $\Pro(\A)$. The category $\Pro(\A)$ is a cocomplete abelian category and admits a projective generator $P$ such that every object $X$ of $\A$ regarded as an object of $\Pro(\A)$ is a quotient of $P$, i.e., there is an epimorphism $P \twoheadrightarrow X$ in $\Pro(\A)$.
By~\cite[Lemma 9.6.9]{kashiwara2006categories}, this implies that $\phi_P$ is fully faithful on $\A$.
\end{proof}

The pro-cocompletion $\Pro(\C)$ of a category $\C$ is pseudofunctorial and preserves adjoint functors. To prove it, it suffices to establish these properties for the ind-completion $\Ind(\C)$ of a category $\C$. Recall that for a given category $\C$, the objects of $\Ind(\C)$ are given by `formal colimits' of filtered diagrams $\alpha: I \to \C$ in $C$, denoted $\colim^f(\alpha)$. The embedding $\C \to \Ind(\C)$ sends an object $X$ to the formal colimit $\colim^f(* \mapsto X)$.

\begin{lemma}[{\cite[Prop.~6.1.9-6.1.11]{kashiwara2006categories}}]\label{lemma:ind_pseudofunc}
For all functors $F: \C \to \C'$, there is a unique functor $\Ind(F): \Ind(\C) \to \Ind(\C')$ such that the following diagram commutes:
\[\begin{tikzcd}[row sep=scriptsize]
	\C & \C' \\
	{\Ind(\C)} & {\Ind(\C')}
	\arrow["F", from=1-1, to=1-2]
	\arrow[from=1-1, to=2-1]
	\arrow[from=1-2, to=2-2]
	\arrow["{\Ind(F)}"', from=2-1, to=2-2]
\end{tikzcd}\]
More precisely, $\Ind(F)$ sends the formal colimit $\colim^f \alpha$ to $\colim^f (F\circ \alpha)$.
In addition, for all functors $G: \C' \to \C''$, there is a natural isomorphism $\Ind(G\circ F) \cong \Ind(G) \circ \Ind(F)$, and if $F$ is faithful or fully faithful functors, then so is $\Ind(F)$.
\end{lemma}
In particular, if $\C'$ is a full subcategory of $\C$ then $\Ind(\C')$ is again a full subcategory of $\Ind(\C)$. 

\begin{lemma}
Ind-completion preserves adjoint functors.
\end{lemma}
\begin{proof}
    For all adjoint functors $L \dashv R$ between categories $\C$ and $\DDD$, we denote the functors $\Ind(L)$ and $\Ind(R)$ by $IL$ and $IR$, respectively, and the formal colimits $\colim^f \alpha$ by $\lex f \alpha$. For all filtered diagrams $\alpha: I \to \C$ and $\beta: J \to \DDD$, we have: 
    \begin{align*}
        \Ind(\DDD)(IL(\lex f\alpha),\lex f \beta) & \cong \lim_{i \in I} \colimb_{j \in J} \DDD(L\alpha(i), \beta(j)) \\
        &\cong \lim_{i \in I} \colimb_{j \in J} \C(\alpha(i), R\beta(j)) \\
        &\cong \Ind(\DDD)(\lex f\alpha, IR(\lex f \beta)) \qedhere
    \end{align*}
\end{proof}

We call an object satisfying the property used in the proof sketch of Theorem~\ref{th:freyd-mitchell} a \look{strong generator}. This nomenclature is not standard, but is motivated by the following remark: By definition, a strong generator is a separator in the sense of~\cite{borceux1994handbook}, which, \emph{in a locally small category}, is equivalent to the notion of generator of Definition~\ref{def:cpg}. In a category with coproducts, strong generators coincides with generators, as noted by Qiaochu in~\cite{qchu2015generators}. 

\begin{definition}
An object $P$ in an abelian category $\A$ is called a \look{strong generator} of a set $S$ of objects of $\A$ if for every object $X \in S$, $X$ is a quotient of $P$, i.e., there exists an epimorphism $P \twoheadrightarrow X$.
\end{definition}

We now show that being a strong generator, and the properties of Definition~\ref{def:cpg}, are preserved by the leftmost adjoint of certain adjoint triples.

\begin{theorem}\label{th:preservation_triple}
For any abelian categories $\A,\B$ and adjoint triple $f_! \dashv f^* \dashv f_*$ with functors $f_!,\, f_* \colon \B \to \A$ and $f^* \colon \A \to \B$, $f_!$ preserves projectives. Moreover, if $f^*$ is faithful, then $f_!$ preserves generators and sends a strong generator of a set $S$ to a strong generator of the preimage $(f^*)^{-1}(S)$. If both $\A$ and $\B$ are cocomplete, then $f_!$ preserves compact objects.
\end{theorem}
\begin{proof}
The functor $f^*$ is between abelian categories and has both adjoints so is automatically exact. As a left adjoint to an exact functor, $f_!$ preserves projectives.

If $f^*$ is faithful, since it is exact, it is conservative.
Let us show in that case that $f_!$ preserves (strong) generators. If $G$ is a generator in $\A$, then for all arrows $a: X \to Y$ in $\B$ such that $\B(f_!G, a) : \B(f_!G, X) \to \B(f_!G, Y)$ is an isomorphism, applying the natural isomorphism of homsets of the adjunction $f_! \dashv f^*$ shows that $\B(G, f^*a)$ is an isomorphism. Since $\phi_G$ and $f^*$ are conservative, $a$ is also an isomorphism. 

Then, for all strong generator $P$ of a set $S$, and for all objects $Y$ in $\B$ with $f^*Y \in S$, there exists an epimorphism $\iota: P \twoheadrightarrow f^*Y$. Let us show that there exists an epimorphism $f_! P \twoheadrightarrow Y$.
By properties of adjoint functors, since $f^*$ is faithful, then the counit $\epsilon: {f_!}{f^*} \to \id_{\B}$ is pointwise epic.
In addition $f_!$ preserves epimorphisms as a left adjoint. Therefore, the morphism $\epsilon_Y\circ f_!(\iota): f_! P \to Y$ is epic.

Finally, if $\A$ and $\B$ are cocomplete, then for all compact objects $C$ in $\A$ and filtered diagrams $F: I \to \B$, we have the following natural isomorphism:
\begin{align*}
    \B(f_!C, \colim_I F) &\cong \A(C, f^*(\colim_I F))\\
    & \cong \A(C, \colim_I f^*F)\\
    & \cong \colim_I \A(C, f^*F) \\
    & \cong \colim_I \B(f_!C, F)
\end{align*}
where $f^*$ preserves colimits as a right adjoint. Hence $f_!$ preserves compact objects.
\end{proof}

\begin{example}\label{ex:commutativity1}
    Consider the cocomplete abelian categories of $\R$ and $\CC$-vector spaces, respectively. The inclusion $f: \R \to \CC$ induces restriction and extension of scalars; in particular, extension is called complexification.
    As with any ring morphism, restriction of scalars is faithful.
    $\R$, regarded as a $1$-dimensional $\R$-vector space, is a compact projective generator. Similarly, $\CC$ is a compact projective generator, and indeed the complexification $\R^{\mathrm C} = \R \otimes_{\R} \CC$ of $\R$ is isomorphic to $\CC$.
\end{example}

Theorem~\ref{th:preservation_triple} implies that the Gabriel and \FM{} theorems can be applied locally and coherently to the local horizontal categories of certain abelian framed bicategories. This is the content of Section~\ref{section:gt_pointwise}. The proofs rely on certain notions of framed bicategories, and on a specific choice of cleavage (that can always be made) which simplifies our approach. These preliminary results are presented in Section~\ref{section:preliminaries}.
Finally, in Section~\ref{section:framed_embedding} we show  that these results assemble into a framed embedding.

\subsection{A cleavage for (co)fibrations \texorpdfstring{$(L,R)$}{(L,R)}}\label{section:preliminaries}

In this section, we show more generally that (co)fibrations of the form $(L,R)$ admit convenient cleavages. We treat the case of fibrations, the case of cofibrations is dual. Let $\C$ and $\DDD$ be categories, and let $L,R : \DDD \to \C$ be two functors such that $(L,R): \DDD \to \C \times \C$ is a fibration.

\begin{proposition}\label{prop:LRfib}
    Both $R$ and $L$ are fibrations. Moreover, a cartesian lifting $f$ of $Lf$ to $Y$ can be chosen such that $f$ is also a cartesian lifting of $(Lf, \id_Y)$ to $Y$, and symmetrically for $R$.
\end{proposition}
\begin{proof}
    Let us show that $L$ is a fibration. For all objects $Y$ in $\DDD$ and arrows $f_0: A \to LY$ in $\C$, consider a cartesian lifting $f: X \to Y$ of $(f_0, \id_{RY}): (A,RY) \to (LY, RY)$ to $Y$. In particular, $Lf = f_0$ and $Rf = \id_{RY}$. Let us show that this arrow is cartesian with respect to $L$. For all arrows $g: Z \to Y$ and $w_0: LZ \to A = LX$ such that $Lg = Lf \circ w_0$, we have $Rg =  \id_{RY} \circ Rg$. Since $f$ is cartesian over $(Lf, \id_{RY})$, there exists a unique arrow $w: Z \to X$ satisfying:
    \begin{align*}
        Lw &= w_0\\
        Rw &= Rg\\
        g &= f \circ w.
    \end{align*}
    If $w': Z \to X$ is another arrow satisfying $Lw' = w_0$ and $g = f\circ w'$, then in particular: 
    \begin{align*} 
    Rg &= R(f \circ w')\\
        &= R(f) \circ R(w')\\
        &= \id_{RY} \circ Rw'\\
        &= Rw'.
    \end{align*}
    Thus, $w=w'$ by uniqueness of $w$. 
    The proof that $R$ is a fibration is symmetric. 
\end{proof}

Henceforth, we consider only cartesian liftings with respect to $L$ and $R$ that satisfy the property of Proposition~\ref{prop:LRfib}. 
For all compatible objects $Y$ in $\DDD$ and arrows $f,g$ in $\C$, consider the cartesian liftings $f^LY \to Y$ of $Lf$ to $Y$ and $Yg^R \to Y$ of $Rg$ to $Y$. By general properties of cartesian arrows, both $(f^LY)g^R \to f^LY \to Y$ and $f^L(Yg^R) \to Yg^R \to Y$ are cartesian liftings of $(f,g)$ to $Y$, and there is a unique isomorphism $f^L(Yg^R) \cong (f^LY)g^R$ commuting with the above morphisms. However, $f^L(Yg^R)$ and $(f^LY)g^R$ may not be equal.
Still, we can construct a cleavage of $(L, R)$ where equality holds. 

\begin{proposition}\label{prop:choice}
Assuming the axiom of choice, there is a choice, for all compatible objects $Y$ in $\DDD$ and arrows $f, g$ in $\C$, of 
    objects $f^*Yg^*$, $f^*Y$, and $Yg^*$ in $\DDD$,
    together with cartesian liftings $f^*Yg^*  \xrightarrow{c} Y$ of $(f,g)$, $f^*Yg^* \xrightarrow{c} f^* Y$ of $(\id_{Lf^*Y}, g)$, $f^*Yg^* \xrightarrow{c} Yg^*$ of $(f, \id_{RYg^*})$, $f^*Y \xrightarrow{c} Y$ of $(f,\id_{RY})$, and $Yg^* \xrightarrow{c} Y$ of $(\id_{LY}, g)$,
such that
    the choice is normal in the sense that $\id_{LY}^*Yg^* \to Yg^*$ is $\id_{Yg^*}$, $\id_{LY}^*Y \to Y$ is $\id_Y$, and symmetrically;
    and such that the following diagram commutes:
\begin{equation}\label{diag:triangles}
\begin{tikzcd}
	{f^*Yg^*} & {Yg^*} \\
	{f^*Y} & Y
	\arrow["c"{description}, from=1-1, to=1-2]
	\arrow["c"{description}, from=1-1, to=2-1]
	\arrow["c"{description}, from=1-1, to=2-2]
	\arrow["c"{description}, from=1-2, to=2-2]
	\arrow["c"{description}, from=2-1, to=2-2]
\end{tikzcd}
\end{equation}
\end{proposition}
\begin{proof}
By assumption, both $L$ and $R$ admit a normal cleavage. For all compatible $Y$ and $f,g$, consider the cartesian arrows $f^LY \xrightarrow{l} Y$ and $Yg^R \xrightarrow{r} Y$, and let $f^*Y = f^LY$, $Yg^* = Yg^R$, and $f^*Yg^* = (f^LY)g^R$. Furthermore, let $f^*Y \xrightarrow{c} Y$ be $f^LY \xrightarrow{l} Y$, let $f^*Yg^* \xrightarrow{c} f^*Y$ be $(f^LY)g^R \xrightarrow{r} f^LY$, and since cartesian arrows are stable under composition, let $f^*Yg^* \xrightarrow{c} Y$ be the composition of the two previous arrows. Then, let $Yg^* \xrightarrow{c} Y$ be $Yg^R \xrightarrow{r} Y$. Finally, using the fact that $Yg^* \xrightarrow{c} Y$ is a fibration, let $f^*Yg^* \xrightarrow{c} Yg^*$ be the unique arrow $\alpha$ such that $L(\alpha) = f$, $R(\alpha) = \id$, and such that the following diagram commutes:
\[\begin{tikzcd}
	{f^*Yg^*} & {Yg^*} \\
	{f^*Y} & Y
	\arrow["{\alpha}", from=1-1, to=1-2]
	\arrow["c"{description}, from=1-1, to=2-1]
	\arrow["c"{description}, from=1-2, to=2-2]
	\arrow["c"{description}, from=2-1, to=2-2]
\end{tikzcd}\]
Since postcomposing $\alpha$ by a cartesian arrow yields another cartesian arrow, $\alpha$ is also cartesian. Thus Diagram~\ref{diag:triangles} commutes. Moreover, since we have chosen a normal cleavage satisfying the property of~Proposition~\ref{prop:LRfib}, $\id_{LY}^*Y \xrightarrow{c} Y$ is $\id_Y$ and symmetrically, and $f^*Y{\id_{RY}}^* \xrightarrow{c} f^*Y$ is $\id_{f^*Y}$. To show that $\id_{LY}^*Yg^* \xrightarrow{c} Yg^*$ is $\id_{Yg^*}$, it suffices to apply its uniqueness property to the following commutative diagram:
\[\begin{tikzcd}[baseline=(\tikzcdmatrixname-\the\pgfmatrixcurrentrow-1.base)]
	{ \id_{LY}^*Yg^* = Yg^*} & {Yg^*} \\
	\id_{LY}^*= Y & Y
	\arrow["{\id_{Yg^*}}", from=1-1, to=1-2]
	\arrow["r"', from=1-1, to=2-1]
	\arrow["r", from=1-2, to=2-2]
	\arrow["{\id_Y}"', from=2-1, to=2-2]
\end{tikzcd}\qedhere\]
\end{proof}

Special care is required when using Proposition~\ref{prop:choice}. In particular, for compatible $Y$, $f$, and $g$, we do not necessarily have $f^*(Yg^*) = f^*Yg^*$.
Let us now return to framed bicategories. For all framed bicategories $\D$, we suppose that the bifibration $(L,R)$ is equipped with a cleavage given by Proposition~\ref{prop:choice}. 
Such cleavage significantly simplifies subsequent proofs, particularly that of Proposition \ref{prop:op}.

\subsection{The Gabriel theorem, pointwise}\label{section:gt_pointwise}

In this section, we get to the heart of the matter: we show that Theorem~\ref{th:preservation_triple} can be applied to pairs of local categories in a framed bicategory connected by change of coefficient functors. This implies that, under certain conditions, if the Gabriel theorem applies to one local category in an abelian framed bicategory, then it propagates to other local categories via extension of scalars.

\begin{proposition}\label{prop:pointwise_gabriel}
Let $\AAA$ be a closed abelian framed bicategory, and let $f: A \to C$ and $g: B \to D$ be vertical arrows such that the local categories $\A(A,B)$ and $\A(C,D)$ are cocomplete, and such that the restriction functor $f^*(\mathunderscore)g^*$ is faithful.
If $\A(A,B)$ contains a compact projective generator $P$, then by Theorem~\ref{th:gabriel} it is equivalent to the category of left modules ${}_{\End(P)}\mmod$. Then:
\begin{itemize}[noitemsep]
    \item $\A(C,D)$ is equivalent to the category of left modules ${}_{\End(f_!Pg_!)}\mmod$ via the functor $\A(f_!Pg_!, \mathunderscore)$,
    \item there is a ring homomorphism $ r: \End(P) \to \End(f_!Pg_!)$,
    \item the restriction functor along $r$ is compatible with the restriction functor along $(f,g)$ in $\AAA$, i.e., there is a natural isomorphism $r^*Y \cong f^*Yg^*$ where $1$-cells $Y$ in $\AAA$ are regarded as left modules via the embedding.
\end{itemize}
In addition, this construction is {pseudofunctorial} in the sense that if $f$ and $g$ are identities, then $r$ is naturally isomorphic to identity, and if $f$ and $g$ are compositions, then $r$ is naturally isomorphic to the composition of corresponding ring homomorphisms.
\end{proposition}
\begin{proof}
By Theorem~\ref{th:preservation_triple}, a compact projective generator $P$ in $\A(A,B)$ is sent by extension to a compact projective generator in $\A(C,D)$. Therefore, Gabriel's theorem~\ref{th:gabriel} applies to $\A(C,D)$. The following map $r$ is induced by the extension of scalars functor $f_!(\mathunderscore)g_!$:
\begin{align*}
   r: \End(P) &\to \End(f_!Pg_!)\\ 
   \alpha &\mapsto f_!\alpha g_!.  
\end{align*}
By functoriality and additivity of $f_!(\mathunderscore)g_!$, $r$ is a ring homomorphism.
Now, for all objects $Y \in \A(C,D)$, we denote by $r^*Y$ the restriction on the left of the left $\End(f_!Pg_!)$-module $\Hom(f_!Pg_!, Y)$. Let us show that there is a natural isomorphism of left $\End(P)$-modules $$r^*\Hom(f_!Pg_!, Y) \cong \Hom(P, f^*Yg^*).$$
By adjunction, the abelian group $\Hom(f_!Pg_!, Y)$ is naturally isomorphic to $\Hom(P, f^*Yg^*)$. 
It remains to show that this natural isomorphism respects the action of $\End(P)$.
For all elements $\alpha \in \Hom(f_!Pg_!, Y)$ and $\beta \in \Hom(P,P)^{\mathrm{op}}$, by naturality of the unit $\eta$ of the adjunction, we have:
\begin{align*}
    f^*(\beta \act \alpha)g^* \circ \eta_P  &= f^*(\alpha \circ f_!\beta g_!)g^* \circ \eta_P\\ 
    &= f^*\alpha g^* \circ f^*f_!\beta g_!g^* \circ \eta_P \\
    &= f^*\alpha g^* \circ \eta_P \circ \beta\\
    &= \beta \act (f^*\alpha g^* \circ \eta_P). 
\end{align*}
Finally, pseudofunctoriality follows from the pseudofunctoriality of changing coefficients, and from the stability of faithful functors under composition.
\end{proof}
\begin{example}
Consider the abelian framed bicategory $\BimodZ$ of bimodules over rings, and consider the local categories of $\R$-vector spaces and $\CC$-vector spaces (since $\R$ and $\CC$ are commutative a bimodule is the same thing as vector space), and the compact projective generator $\R$. $\BimodZ$ satisfies the hypothesis of Proposition~\ref{prop:pointwise_gabriel}, therefore Gabriel's theorem can indeed be applied coherently on $\R$-vector spaces and $\CC$-vector spaces.   
\end{example}
\begin{remark}\label{remark:alg1}
Proposition~\ref{prop:pointwise_gabriel} implies that for any such abelian framed bicategories $\AAA$, because of the ring homomorphism $r: \End(P) \to \End(f_!Pg_!)$, \emph{if the ring $\End(P)$ is commutative and the compatibility condition $r(p) \circ \alpha = \alpha \circ r(p)$ holds for all elements $p \in \End(P)$ and $\alpha \in \End(f_!Pg_!)$}, then the ring $\End(f_!Pg_!)$ carries the structure of an $\End(P)$-algebra (see Proposition~\ref{prop:algebra}), and for all $1$-cells $M: C \tobar D$, the image $\A(f_!Pg_!, M)$ of $M$ under the embedding automatically acquires the structure of a left module over an $\End(P)$-algebra. 

Determining whether the ring $\End(P)$ is commutative is in general a difficult problem~\cite{cualuguareanu2010modules}. 
Nevertheless, under the conditions of Section~\ref{section:framed_embedding}, an Eckmann-Hilton argument applies, see Proposition~\ref{prop:endui_commut}, and also sheds light on the second condition.
\end{remark}
\begin{remark}\label{remark:faithful_necessary}
Since natural isomorphisms preserve faithful functors, and since restriction of scalars along ring homomorphisms is always faithful in concrete categories of modules, the condition in Proposition~\ref{prop:pointwise_gabriel} that the restriction functor be faithful is necessary.
\end{remark}
Proposition~\ref{prop:pointwise_gabriel} implies that in a closed abelian framed bicategory, local equivalence to a concrete category of modules propagates via extension of scalars, albeit only pseudofunctorially.
Proposition~\ref{prop:initial_faithful} demonstrates how to apply Proposition~\ref{prop:pointwise_gabriel} without losing functoriality by using a cone over a diagram of pairs of vertical objects whose restriction of scalars functors along the cone components are faithful. This occurs, for instance, in an abelian framed bicategory $\AAA$ with an initial coefficient $I$ such that all restriction functors (equivalently, all restrictions along vertical maps $I \to A$ for $A \in \Ob(\AAA_0)$) are faithful.
\begin{definition}
An \look{initial coefficient} $I$ in a double category $\D$ is an initial object in the vertical category $\D_0$. We write $\iota_X: I \to X$ for the unique vertical arrow from $I$ to a vertical object $X$. 
\end{definition}
\begin{example}
The ring $R$ regarded as an $R$-algebra is initial in the vertical category of $R$-algebras,  the empty set is initial in the vertical category $\set$, and the absorption monoid $\{0,1\}$ is initial in the vertical category of absorption monoids.
\end{example}
\begin{remark}
The existence of an initial coefficient $I$ in a framed bicategory $\D$ allows one to define additional concepts: for example, local categories of left (resp. right) parameterized objects $\DD(A, I)$ (resp. $\DD(I, A)$) for all vertical objects $A$. In a framed bicategory of modules, this corresponds to the usual notion of left (resp. right) modules. 
\end{remark}
\begin{proposition}
    In an abelian framed bicategory $\AAA$ with initial coefficient $I$, if the restriction functors along all vertical maps $I \to A$, $A \in \Ob(\AAA_0)$ are faithful, then all the restriction functors are faithful.
\end{proposition}
\begin{proof}
For all vertical arrows $f: A \to B$, pseudofunctoriality of change of coefficient functors yields the natural isomorphism $\iota_A^* \circ f^* \cong (f \circ \iota_A)^* = \iota_B^*$. Since $\iota_B^*$ is faithful, so is $f^*$.
\end{proof}
In an abelian framed bicategory with initial coefficient $I$, if Gabriel's theorem applies to the local category $\AAA(I,I)$, then it applies functorially to all local categories. To see this, we apply Proposition~\ref{prop:pointwise_gabriel} to the extension of scalars functors ${\iota_A}!(\mathunderscore){\iota_B}!: \AAA(I,I) \to \AAA(A,B)$ for all vertical objects $A$ and $B$, and construct the remaining required ring homomorphisms explicitly, using the following lemma. 

\begin{lemma}[Factorization is pseudofunctorial]\label{lemma:warmup}
Let $\D$ be a framed bicategory. For all $2$-cells $\alpha: M \ccs{f}{g} N$, let $\fact \alpha$ denote the factorization of $\alpha$ through the cartesian cell $f^*Ng^* \xrightarrow{\cart} N$. Then, for all $2$-cells $\alpha: M \ccs{f}{g} N$ and $\beta: N \cc{h}{k} Q$, consider the unique isomorphism $\phi: f^*h^*Qk^*g^* \to (hf)^*Q(kg)^*$ given by the uniqueness property of cartesian $2$-cells.
Then, we have $\fact{\beta \circ \alpha} = \phi \circ f^*\fact \beta g^* \circ \fact \alpha$. 
\end{lemma}
\begin{proof}
By uniqueness of factorization through the cartesian cell $(hf)^*Q(kg)^* \xrightarrow{\ct} Q$, the commutativity of the following diagram suffices:
\[\begin{tikzcd}[baseline=(\tikzcdmatrixname-\the\pgfmatrixcurrentrow-1.base)] 
	M && {(hf)^*Q(kg)^*} \\
	{f^*Ng^*} & N & Q \\
	{f^*h^*Qk^*g^*} && {h^*Qk^*}
	\arrow["{\fact{\beta \alpha}}"{description}, from=1-1, to=1-3]
	\arrow["{\fact \alpha}"', from=1-1, to=2-1]
	\arrow["\alpha"{description}, from=1-1, to=2-2]
	\arrow["{\beta\alpha}"{description}, from=1-1, to=2-3]
	\arrow["{ \cart}", from=1-3, to=2-3]
	\arrow["\cart"{description}, from=2-1, to=2-2]
	\arrow["{f^*\fact \beta g^*}"', from=2-1, to=3-1]
	\arrow["\beta"{description}, from=2-2, to=2-3]
	\arrow["{\fact \beta}"{description}, from=2-2, to=3-3]
	\arrow[bend right = 25.5, from=3-1, to=1-3]
	\arrow["\cart"{description}, from=3-1, to=3-3]
	\arrow["\cart"', from=3-3, to=2-3]
\end{tikzcd}\qedhere
\]
\end{proof}

In the following, we will use the dual of Lemma~\ref{lemma:warmup} for cocartesian $2$-cells.

\begin{proposition}\label{prop:general_functoriality}
 Let $\AAA$ be a locally cocomplete closed abelian framed bicategory with an initial coefficient $I$ such that all restriction functors are faithful. For all compact projective generators $P$ in $\A(I,I)$, there exists a functor:
 \begin{align*}
 G:\AAA_1 &\to {}_{\Z}\mmod \\
 M: C \tobar D &\mapsto \A({\iota_C}_! P {\iota_D}_!, M)
 \end{align*}
 where  ${}_{\Z}\mmod$ denotes the category of left modules over rings and left linear maps.
\end{proposition}
\begin{proof}
For clarity, we prove the statement for $1$-cells $M$ and $2$-cells $\alpha$ with $R(M) = I$ and $R(\alpha) = \id_I$. However, the proof only relies on the fact that $L$ is a bifibration (by Proposition~\ref{prop:LRfib}), and the same construction can be carried out with the bifibration $(L,R)$, yielding the above statement.
It remains to define $G$ on $2$-cells, and to show its functoriality. For all vertical arrows $f: A \to B$, $1$-cells $M: A \tobar I$, $N: B \tobar I$ and $2$-cells $\alpha: M \cc{\id_I}{f} N$, there is a unique globular, isomorphic $2$-cell $\varphi$ such that the following diagram commutes:
\begin{equation}\label{diag:uniquecommuting}
    \begin{tikzcd}
	& {{\iota_A}_!P} \\
	P & {f_!{\iota_A}_!P} \\
	& {{\iota_B}_!P}
	\arrow["\cct", from=1-2, to=2-2]
	\arrow["\cct", from=2-1, to=1-2]
	\arrow["\cct"', from=2-1, to=3-2]
	\arrow["\varphi"', from=3-2, to=2-2]
\end{tikzcd}.
\end{equation}
We denote $\varphi^{-1} \circ \cct$ by $\hat f$. $\hat f$ is the unique $2$-cell making the following diagram commutes. Moreover, by fact~\ref{fact2}, $\hat f$ is also cocartesian.
\[\begin{tikzcd}[row sep=tiny]
	& {{\iota_A}_!P} \\
	P \\
	& {{\iota_B}_!P}
	\arrow["{\hat f}", from=1-2, to=3-2]
	\arrow["\cct", from=2-1, to=1-2]
	\arrow["\cct"', from=2-1, to=3-2]
\end{tikzcd}\]
Then, for all elements $m \in G(M)$, let $G(\alpha)(m)$ denotes $\fact \alpha \circ f_! m \circ \varphi$, where $\fact \alpha : f_!M \to N$ is the factorization of $\alpha$ through $M \xrightarrow{cc} f_!M$. By construction, the following diagram commutes. Moreover, since $\hat f$ is cocartesian, $G(\alpha)(m)$ is the unique factorization of $\alpha \circ m$ trough $\hat f$. 
\[\begin{tikzcd}
	& {{\iota_A}_!P} & M \\
	P & {f_!{\iota_A}_!P} & {f!M} \\
	& {{\iota_B}_!P} & N
	\arrow["m", from=1-2, to=1-3]
	\arrow["\cct"{description}, from=1-2, to=2-2]
	\arrow["\cct"{description}, from=1-3, to=2-3]
	\arrow["\alpha", curve={height=-18pt}, from=1-3, to=3-3]
	\arrow["\cct", from=2-1, to=1-2]
	\arrow["\cct"', from=2-1, to=3-2]
	\arrow["{f_!m}", from=2-2, to=2-3]
	\arrow["{\fact \alpha}"', from=2-3, to=3-3]
	\arrow["\varphi"', from=3-2, to=2-2]
	\arrow["{G(\alpha)(m)}"', from=3-2, to=3-3]
\end{tikzcd}\]
Since extension of scalars and composition in an abelian category are additive, $G(\alpha)$ is a group homomorphism. We show it is in fact left $r_f$-linear, where $r_f$ denotes the group homomorphism $G(\hat f): \End({\iota_A}_!P) \to \End({\iota_B}_!P)$. We have to show that $r_f$ is a ring homomorphism, i.e., that it preserves composition, and that $G(\alpha)$ preserves the action of $\End({\iota_A}_!P)$ through $r_f$. For all $1$-cells $X: A \tobar I$, $Y: B \tobar I$, $2$-cells $\gamma: X \ccs{f}{\id_I} Y$ 
and elements $x \in G(X)$ and $a \in \End({\iota_A}_!P)$, consider the following commutative diagram:
\begin{equation}
\begin{tikzcd}\label{diag:threeinone}
	& {{\iota_A}_!P} & {{\iota_A}_!P} & X \\
	P & {f_!{\iota_A}_!P} & {f_!{\iota_A}_!P} & {f!X} \\
	& {{\iota_B}_!P} & {{\iota_B}_!P} & Y
	\arrow["a", from=1-2, to=1-3]
	\arrow["\cct"{description}, from=1-2, to=2-2]
	\arrow["x", from=1-3, to=1-4]
	\arrow["\cct"{description}, from=1-3, to=2-3]
	\arrow["\cct"{description}, from=1-4, to=2-4]
	\arrow["\gamma", curve={height=-18pt}, from=1-4, to=3-4]
	\arrow["\cct", from=2-1, to=1-2]
	\arrow["\cct"', from=2-1, to=3-2]
	\arrow["{f_!a}", from=2-2, to=2-3]
	\arrow["{f_!x}", from=2-3, to=2-4]
	\arrow["{\fact \gamma}"', from=2-4, to=3-4]
	\arrow["\varphi"', from=3-2, to=2-2]
	\arrow["{r_f(a)}"', from=3-2, to=3-3]
	\arrow["\varphi"', from=3-3, to=2-3]
	\arrow["{G(\gamma)(x)}"', from=3-3, to=3-4]
\end{tikzcd}
\end{equation}
Both points follow directly from functoriality of changing coefficients and diagram~\ref{diag:threeinone} with:
\begin{itemize}[noitemsep]
    \item for the first point, $X={\iota_A}_!P$, $Y= {\iota_B}_!P$, $\gamma = \varphi^{-1} \circ \cct$ for all elements $a,x \in \End({\iota_A}_!P)$;
    \item for the second point, $X=M$, $Y=N$, $\gamma = \alpha$ and $x = m$ for all elements $a \in \End({\iota_A}_!P)$.
\end{itemize}
Finally, let us show that $G$ is functorial.
For all vertical arrows $f: A \to B$, $g: B \to C$, $1$-cells $M: A \tobar I$, $N: B \tobar I$, $Q: C \tobar I$, $2$-cells $\alpha: M \ccs{\id_I}{f} N$ and $\beta: N \cc{\id_I}{g} Q$, let $\varphi$ and $\varphi'$ be the unique isomorphisms ${\iota_B}_! \to f_!{\iota_A}_!P$ and ${\iota_C}_! \to g_!{\iota_B}_!P$ such that diagram~(\ref{diag:uniquecommuting}) and its adaptation for $\varphi'$ commute. By uniqueness, $\widehat {g \circ f} = \hat g \circ \hat f$. We must show that for any $m \in G(M)$, $G(\beta\alpha)(m) = G(\beta)(G(\alpha)(m))$. The commutativity of the following diagram suffices:
\[\begin{tikzcd}
	{{\iota_A}_!P} & M \\
	{f_!{\iota_A}_!P} & {f!M} \\
	{{\iota_B}_!P} & N & {g_!f_!M} & {(gf)_!M} \\
	{g_!{\iota_B}_!P} & {g_!N} \\
	{{\iota_C}_!P} & Q
	\arrow["m", from=1-1, to=1-2]
	\arrow["\cct"{description}, from=1-1, to=2-1]
	\arrow["{\hat f}"', curve={height=20pt}, from=1-1, to=3-1]
	\arrow["{\widehat{gf}}"', curve={height=50pt}, from=1-1, to=5-1]
	\arrow["\cct"{description}, from=1-2, to=2-2]
	\arrow["\cct", from=1-2, to=3-4]
	\arrow["{f_!m}", from=2-1, to=2-2]
	\arrow["{\varphi^{-1}}", from=2-1, to=3-1]
	\arrow["{\fact \alpha}", from=2-2, to=3-2]
	\arrow["\cct", from=2-2, to=3-3]
	\arrow["{G(\alpha)(m)}"', from=3-1, to=3-2]
	\arrow["\cct"{description}, from=3-1, to=4-1]
	\arrow["{\hat g}"', curve={height=20pt}, from=3-1, to=5-1]
	\arrow["\cct"{description}, from=3-2, to=4-2]
	\arrow["\phi", from=3-3, to=3-4]
	\arrow[""{name=0, anchor=center, inner sep=0}, "{g_!\fact \alpha}"', from=3-3, to=4-2]
	\arrow[""{name=1, anchor=center, inner sep=0}, "{\fact{\beta\circ \alpha}}", from=3-4, to=5-2]
	\arrow[from=4-1, to=4-2]
	\arrow["{\varphi'^{-1}}", from=4-1, to=5-1]
	\arrow["{\fact \beta}", from=4-2, to=5-2]
	\arrow["{G(\beta)(G(\alpha)(m))}"', shift right=2, draw=none, from=5-1, to=5-2]
	\arrow[from=5-1, to=5-2]
	\arrow["A"{description}, draw=none, from=0, to=1]
\end{tikzcd}\]
where cell $A$ commutes by the dual of Lemma~\ref{lemma:warmup}. This completes the (partial) proof. In the full proof, one would define $2$-cells $\widehat{f,g}$ and ring homomorphisms $r_{f,g}$ for all pairs of vertical arrows $f,g$, but otherwise the notations would have been similar.
\end{proof}
\begin{remark}\label{remark:alg2}
Following Remark~\ref{remark:alg1}, if the rings $\End({\iota_A}_!P)$ and $\End({\iota_B}_!P)$ are actually $\End(P)$-algebras, then the ring homomorphism $r: \End({\iota_A}_!P) \to \End({\iota_B}_!P)$ defined in Proposition~\ref{prop:general_functoriality} is an $\End(P)$-algebra homomorphism: it suffices to verify that $r$ preserves the action of $\End(P)$, which follows from the commutativity of diagram~\ref{diag:threeinone} with $X={\iota_A}_!P$, $Y= {\iota_B}_!P$, $\gamma = \varphi^{-1} \circ \cct$ and $a={\iota_A}_!(p)$ for all elements $p \in \End(P)$, $x \in \End({\iota_A}_!P)$, together with the commutativity of diagram~(\ref{diag:riotaAp}) showing that $r({\iota_A}_!(p)) = {\iota_B}_!(p)$ for all $p \in \End(P)$:
\begin{equation}\label{diag:riotaAp}
\begin{tikzcd}
	{{\iota_A}_!P} &&& {{\iota_A}_!P} \\
	& P & P \\
	{{\iota_B}_!P} &&& {{\iota_B}_!P}
	\arrow["{{\iota_A}_!p}", from=1-1, to=1-4]
	\arrow["{\hat f}"', from=1-1, to=3-1]
	\arrow["{\hat f}", from=1-4, to=3-4]
	\arrow["\cct"{description}, from=2-2, to=1-1]
	\arrow["p", from=2-2, to=2-3]
	\arrow["\cct"{description}, from=2-2, to=3-1]
	\arrow["\cct"{description}, from=2-3, to=1-4]
	\arrow["\cct"{description}, from=2-3, to=3-4]
	\arrow["{{\iota_B}_!p}"', from=3-1, to=3-4]
\end{tikzcd}
\end{equation}
\end{remark}
\begin{proposition}\label{prop:initial_faithful}
Let $\AAA$ be a locally cocomplete closed abelian framed bicategory with an initial coefficient $I$ such that all restriction functors are faithful. For all compact projective generators $P$ in $\A(I,I)$, let $R(\mathunderscore,\mathunderscore): \AAA_0 \times \AAA_0 \to \ring$ be the functor 
\begin{align*}
    (A,B) & \mapsto \End({\iota_A}_!P{\iota_B}_!) = G({\iota_A}_!P{\iota_B}_!)^{\mathrm{op}}\\
    (f: A \to C,g: B \to D) & \mapsto G(\widehat{f,g}) : \End({\iota_A}_!P{\iota_B}_!) \to \End({\iota_C}_!P{\iota_D}_!).  
\end{align*}
For all pairs of vertical arrows $(f,g)$, the change of coefficient functors along $R(f,g)$ are compatible with the change of coefficients functor along $(f,g)$ in the sense of Proposition~\ref{prop:pointwise_gabriel}.
\end{proposition}
\begin{proof}
The proof follows the pattern of Proposition~\ref{prop:pointwise_gabriel}. As in Proposition~\ref{prop:general_functoriality}, we present only the case of left restriction. Using the notation of Proposition~\ref{prop:general_functoriality}, for all vertical arrows $f: A \to B$, we must show that the following group isomorphism, natural in $1$-cells $N: B \tobar I$, preserves the action of $\End({\iota_A}_!P)$: $$R(f,\id_I)^*\A({\iota_B}_!P, N) \xrightarrow{\mathunderscore \circ \varphi^{-1} }{\cong} \A(f_!{\iota_A}_!P, N) \cong \A({\iota_A}_!P, f^*N)$$
For all $1$-cells $N: B \tobar I$, and elements $n \in \A({\iota_B}_!P, N)$ and $a \in \End({\iota_A}_!P)$, we have:
\begin{align*}
    f^*((a \act n) \circ \varphi^{-1}) \circ \eta_{{\iota_A}_!P} &= f^*(n \circ R(f,\id_I)(a) \circ \varphi^{-1}) \circ \eta_{{\iota_A}_!P}\\
    &= f^*(n \circ \varphi^{-1} \circ f_!a) \circ \eta_{{\iota_A}_!P}\\
    &= f^*n \circ f^*\varphi^{-1} \circ f^*f_!a  \circ \eta_{{\iota_A}_!P}\\
    &= f^*n \circ f^*\varphi^{-1} \circ \eta_{{\iota_A}_!P} \circ a\\
    &= a \act (f^* (n \circ \varphi^{-1}) \circ \eta_{{\iota_A}_!P}) \qedhere
\end{align*}
\end{proof}

\begin{remark}\label{remark:all_faithful}
    The condition that all restriction functors be faithful is very restrictive. Restriction of scalars along ring morphisms is always faithful, but this rarely holds for other examples of framed bicategories. For distributors and spans (see~\cite[Example 2.4]{shulman2007framed}), restriction is faithful only along epimorphisms.
    Still, this condition is necessary; see Remark~\ref{remark:faithful_necessary}.
\end{remark}

\begin{example}\label{ex:bimodules_unsmash}
Consider the abelian framed bicategory $\BimodZ$ of modules over rings. $\BimodZ$ is closed, locally cocomplete, and has an initial coefficient: the ring $\Z$. $U_{\Z}$ is $\Z$ regarded as a $\Z$-module, i.e, as an abelian group. $U_{\Z}$ is a compact projective generator in the local category of abelian groups of $\BimodZ$. Moreover, restriction of scalars functors are faithful. 
By definition, for all rings $A$ and $B$, we have the $A$-$B$-bimodule isomorphism ${\iota_A}_!\Z{\iota_B}_! \cong U_A\otimes_{\Z}U_B$, 
hence the ring isomorphism $\End({\iota_A}_!\Z{\iota_B}_!) \cong A \otimes_{\Z} B^{\mathrm{op}}$, and for all $A$-$B$-bimodule the left $A \otimes_{\Z} B^{\mathrm{op}}$ module isomorphism $\Hom({\iota_A}_!\Z{\iota_B}_!, M) \cong M$.
Therefore, Propositions~\ref{prop:general_functoriality} applies and yields a functor $G: \BimodZ_1 \to {}_{\Z}\mmod$ sending an $A$-$B$-bimodule
$M$ to $M$ regarded as a left $A \otimes_{\Z} B^{\mathrm{op}}$ module, and similarly for morphisms.
Intuitively, $G$ smashes bimodules into left modules.
\end{example}

Putting together Propositions~\ref{prop:pointwise_gabriel},~\ref{prop:general_functoriality}, and~\ref{prop:initial_faithful}, we see that for any abelian framed bicategory $\AAA$ satisfying the hypotheses of Proposition~\ref{prop:initial_faithful}, all local abelian categories are equivalent to concrete categories of left modules, coherently with respect to change of coefficient functors, and functorially. One might ask whether this data assemble into a framed functor~\cite[Definition 6.1]{shulman2007framed}, the correct notion of morphism between framed bicategories, whose target is the concrete abelian framed bicategory of bimodules over varying rings $\BimodZ$. This is the content of the following section.   

\subsection{A framed embedding}\label{section:framed_embedding}

In this section, we construct for all abelian framed bicategories $\AAA$ satisfying the conditions of Proposition~\ref{prop:initial_faithful} and such that $U_I$ is a compact projective generator, a lax framed functor into 
$\BimodZ$. This will actually yield a lax framed functor into the concrete abelian framed bicategory of bimodules over $\End(P)$-algebras $\BimodP$.
There are several variants of framed functors, but the lax ones are often of greatest interest:

\begin{definition}[{\cite[Definition 6.14]{shulman2007framed}}]
A \look{lax framed functor} $F: \AAA \to \BBB$ between framed bicategories is the data of a \look{vertical functor} $F_0 : \AAA_0 \to \BBB_0$ and an \look{horizontal functor} $F_1: \AAA_1 \to \BBB_1$ satisfying $L \circ F_1 = F_0 \circ L$ and $R \circ F_1 = F_0 \circ R$, along with natural transformations $F_\odot : F_1(M) \odot F_1(N) \to F_1(M \odot N)$ and $F_U : U_{F_0(A)} \to F_1(U_A)$ satisfying the usual axioms for $2$-functors. If the natural transformations $F_\odot$ and $F_U$ go in the reverse direction (resp. are isomorphism), then $F$ is said to be \look{oplax} (resp. \look{strong}) instead.
\end{definition}

For any abelian framed bicategories $\AAA$ satisfying the conditions of Proposition~\ref{prop:initial_faithful} with a compact projective generator $P$ in $\A(I,I)$, consider the functors $G$ and $R$ defined in Propositions~\ref{prop:general_functoriality} and~\ref{prop:initial_faithful}.
For all $1$-cells $M: A \tobar B$, $G(M)$ is a left $R(A,B)$-module, whereas we expect $F_1(M)$ to be an $F_0(A)$-$F_0(B)$-bimodule. Thus we need a way to 'unsmash` the ring $R(A,B)$, i.e., we need to show that $F_1(M)$ is a left-$F_0(A) \otimes_{\Z} F_0(B)$ module for some functor $F_0$. Since the ring-valued functor $R(\mathunderscore, I)$ is a simple candidate for $F_0$, a two-step approach would be:
\begin{enumerate}
  \item first, exhibit a natural ring homomorphism: $$\psi_{A,B}: R(A,P) \otimes_{\Z} R(B,P)^{\mathrm{op}} \to R(A,B);\label{step:1}$$
  \item second, show that restriction of scalars along $\psi$ can be applied on $G$, preserving its functoriality.\label{step:2}. 
\end{enumerate}
However, the first step proves challenging at this level of generality, not to mention the natural transformations $F_\odot$ and $F_U$.
Therefore, we restrict our attention to a well-behaved family of abelian framed bicategory,  which we call \look{module-like}.

\begin{definition}\label{def:modulelike}
A \look{module-like} abelian framed bicategory is a locally cocomplete, closed, abelian framed bicategory $\AAA$ with an initial coefficient $I$ such that all restriction functors are faithful and $U_I$ is a compact projective generator of $\A(I,I)$.
\end{definition}

\begin{example}
    The prototypical examples of module-like abelian framed bicategories are framed bicategories of modules as $\Bimod$ ($I$ is $R$ regarded as an $R$-algebra), $\BimodZ$ ($I = \Z$). This also includes trivial framed bicategories where the vertical category is a point category and the horizontal category is a category of modules.
\end{example}

From this point onward, let $\AAA$ denote a module-like abelian framed bicategory with initial coefficient $I$. A consequence of Definition~\ref{def:modulelike} is that certain endomorphism rings automatically become $\End(U_I)$-algebras. In particular, the ring $\End(U_I)$ is now commutative.

\begin{proposition}\label{prop:endui_commut} In $\End(U_I)$, the two operations $ \mathfrak l_{U_I}\circ (\mathunderscore \odot \mathunderscore) \circ \mathfrak l_{U_I}^{-1}$ and $\circ$ coincide and are commutative. In particular, $\End(U_I)$ is commutative ring. 
\end{proposition}
\begin{proof}
The proof is similar to the argument used in~\cite[\S I.1.3.3]{rivano1972categories}.
In details, consider the following binary operations: multiplication in $\End(U_I)$ and $\times:  (\alpha, \beta) \mapsto \mathfrak l_{U_I} \circ (\alpha \odot \beta) \circ \mathfrak l_{U_I}^{-1}$. By naturality of $\mathfrak l$ the identity $2$-cell $1 \in \End(U_I)$ is also a unit for $\times$. For all elements $\alpha, \beta, \gamma, \delta \in \End(U_I)$, we have:
\begin{align*}
    (\alpha \times \beta) \circ (\gamma \times \delta) &= \mathfrak l_{U_I} \circ (\alpha \odot \beta) \circ (\gamma \odot \delta) \circ \mathfrak l_{U_I}^{-1}  \\
    &=  \mathfrak l_{U_I} \circ ((\alpha \circ \gamma) \odot (\beta \circ \delta)) \circ \mathfrak l_{U_I}^{-1} \\
    &= (\alpha \circ \gamma) \times (\beta \circ \delta).
\end{align*}
Hence the Eckmann-Hilton argument applies. 
\end{proof}

\begin{lemma}\label{lemma:simple_natiso}
    There is a natural isomorphism ${\iota_A}_! U_I \odot U_I {\iota_C}_! \cong {\iota_A}_! (U_I \odot U_I) {\iota_C}_! $ commuting with the $2$-cells $U_I \odot U_I \xrightarrow{\cct \odot \cct } {\iota_A}_! U_I \odot U_I {\iota_C}_!$ and $U_I \odot U_I \xrightarrow{\cct} {\iota_A}_! (U_I \odot U_I) {\iota_C}_!$. In particular, the $2$-cell $U_I \odot U_I \xrightarrow{\cct \odot \cct } {\iota_A}_! U_I \odot U_I {\iota_C}_!$ is cocartesian. 
\end{lemma}
\begin{proof}
Consider the following chain of natural isomorphisms:
\begin{align*}
{\iota_A}_! U_I \odot U_I {\iota_C}_! &\cong (A_{\iota_A} \odot U_I) \odot (U_I \odot {}_{\iota_C}C)\\
&\cong  A_{\iota_A} \odot {}_{\iota_C}C\\
&\cong A_{\iota_A} \odot ((U_I \odot U_I) \odot {}_{\iota_C}C)\\
&\cong {\iota_A}_! (U_I \odot U_I) {\iota_C}_!.
\end{align*}
It is then elementary to build a commutative diagram involving the above $2$-cells and the cocartesian arrows $\lambda_!$ and $\rho_!$ defined in Section~\ref{subsection:fb}.
\end{proof}

\begin{proposition}\label{prop:algebra}
For all objects $A$, the ring homomorphism $$r : \End(U_I) \to \End({\iota_A}_!U_I)$$ from Proposition~\ref{prop:pointwise_gabriel} induces an $\End(U_I)$-algebra structure on $\End({\iota_A}_!U_I)$.
\end{proposition}
\begin{proof}
Recall that $r$ is defined by $p \mapsto {\iota_A}_! p$. The action induced by $r$ is defined as:
\begin{align*}
   \End(U_I) \times \End({\iota_A}_!U_I) &\to \End({\iota_A}_!U_I)\\
   (p, \alpha) &\mapsto \alpha \circ {\iota_A}_!p.
\end{align*}
By Proposition~\ref{prop:endui_commut}, $\End(U_I)$ is commutative. It remains to check that the image of $r$ is contained in the center of $\End({\iota_A}_!U_I)$, i.e, for all elements $p \in \End(U_I)$ and $\alpha \in \End({\iota_A}_!U_I)$, we have to show $\alpha \circ {\iota_A}_! p = {\iota_A}_! p \circ \alpha$.
Conjugating by the natural isomorphism $\mathfrak l$, it suffices to show that  
$(\alpha \odot \id_{U_I}) \circ ({\iota_A}_!p \odot \id_{U_I}) = ({\iota_A}_!p \odot \id_{U_I}) \circ (\alpha \odot \id_{U_I})$.
Let us first show the equality ${\iota_A}_!p \odot \id_{U_I} = {\iota_A}_!\id_{U_I} \odot p$. By Lemma~\ref{lemma:simple_natiso}, it suffices to show ${\iota_A}_!(p \odot \id_{U_I}) = {\iota_A}_!(\id_{U_I} \odot p)$, which holds by Proposition~\ref{prop:endui_commut}.
Finally, consider the following chain of equalities:
\begin{align*}
    (\alpha \odot \id_{U_I}) \circ ({\iota_A}_!p \odot \id_{U_I}) &= (\alpha \odot \id_{U_I}) \circ ({\iota_A}_!\id_{U_I} \odot p)\\
    &= (\alpha \odot \id_{U_I}) \circ (\id_{{\iota_A}_!U_I} \odot p)\\
    &= (\id_{{\iota_A}_!U_I} \odot p)  \circ (\alpha \odot \id_{U_I})\\
    &= ({\iota_A}_!\id_{U_I} \odot p) \circ (\alpha \odot \id_{U_I})\\
    &= ({\iota_A}_!p \odot \id_{U_I}) \circ (\alpha \odot \id_{U_I}). \qedhere
\end{align*}
\end{proof}

\begin{remark}
The same reasoning does not adapt to show that for any objects $A$ and $B$, the ring $\End({\iota_A}!U_I{\iota_B}!)$ is an $\End(U_I)$-algebra. Nevertheless, Proposition~\ref{prop:algebra} suffices for our purposes.
\end{remark}

At this stage, Proposition~\ref{prop:algebra} and Remarks~\ref{remark:alg1} and~\ref{remark:alg2} apply, allowing us to refine our construction:  The functor $G: \AAA_1 \to {}_{\Z} \mmod$ from Proposition~\ref{prop:general_functoriality} is promoted to a functor into modules over $\End(U_I)$-algebras, and the functor $R: \AAA_0 \times \AAA_0 \to \ring$ from Proposition~\ref{prop:initial_faithful} is promoted to a functor into $\End(U_I)$-algebras.
Therefore, from now one we construct a lax framed functor into the abelian framed bicategory $\BimodUI$ rather than into $\BimodZ$. In particular, we define the following vertical functor $F_0$.

\begin{definition}
Let $F_0: \AAA_0 \to \BimodUI_0$ be the functor $R(\mathunderscore, I)$, sending an object $A$ in $\AAA$ to the $\End(U_I)$-algebra $\End({\iota_A}_!U_I)$.
\end{definition}

\begin{example}
Consider the module-like abelian framed bicategory $\Bimod$. The ring $\End(U_I)$ is isomorphic to $R$, and the functor $F_0: {}_R\alg \to {}_{\End(U_I)}\alg$ sends an $R$-algebra $A$ to the $\End(U_I)$-algebra $\End(A\otimes_{R} U_R) \cong A$, and similarly for morphisms.
\end{example}

\subsubsection{The horizontal functor \texorpdfstring{$F_1$}{F1}}

Let us now define the horizontal functor $F_1: \AAA_1 \to \BimodUI_1$ by applying the two steps~\ref{step:1} and~\ref{step:2}. We start by showing that for all objects $A$, $F_0(A)^{\mathrm{op}}$ can be expressed as $R(I, A)$. Note that this is \emph{not} the same as showing $U_I {\iota_A}_! \cong {\iota_A}_!U_I$ for all objects $A$. In particular, module-like categories are not automatically involutive in the sense of~\cite[Def. 10.1]{shulman2007framed}.

\begin{proposition}\label{prop:op}
For all vertical arrows $f: A \to B$, there is a ring isomorphism $\swap_f: \End(f_!U_A) \to \End(U_Af_!)^{\mathrm{op}} $. When $A = I$, these assemble into a natural isomorphism of $\End(U_I)$-algebras $\swap_M: R(M, I)^{\mathrm{op}} \Rightarrow R(I, M)$.
\end{proposition}
\begin{proof}
Recall the notation from the end of Section~\ref{subsection:fb}. For any vertical arrow $f: A \to B$,
the isomorphisms $\lambda^*, \lambda_!, \rho^*$ and $\rho_!$ induce two isomorphisms of additive group $\sigma: f^*f_!U_A \to f^* U_Bf^*$ and $\sigma': U_Af_! f^* \to f^*U_Bf^*$. The former is defined as $\varphi \circ f^*\mathfrak r_A \circ f^*\lambda_!^{-1}$ where $\varphi$ is the unique isomorphism $f^*B_f \to f^*U_Bf^*$ commuting with the cartesian cells $f^*B_f \to B_f = U_bf^* \leftarrow f^*U_Bf^*$ (see Diagram~\ref{diag:identity}). The latter is defined symmetrically. 
Then, consider the group isomorphism $\swap_f: \A(f_!U_A, f_! U_A) \to \A(U_Af_!, U_Af_!)$ induced by the adjunction $f_! \dashv f^*$ and by the isomorphisms $\sigma$ and $\sigma'$:
\begin{align}
    \A(f_!U_A, f_! U_A) &\cong  \A(U_A, f^*f_! U_A) \label{eq:comp1}\\
    &\cong \A(U_A, f^*U_Bf^*)\label{eq:comp2} \\
    & \cong \A(U_A, U_Af_!f^*) \cong \A(U_Af_!, U_Af_!).\label{eq:comp4}
\end{align}
The abelian groups $\A(f_!U_A, f_!U_A)$ and $\A(U_Af_!, U_Af_!)$ carry ring structures with multiplication given by composition and units the identity morphisms. We show that $\swap_f$ is actually a ring isomorphism $\End(f_!U_A) \to \End(U_Af_!)^{\mathrm{\mathrm{op}}}$; i.e., it preserves identities and reverses composition.

First note that each group in the chain of isomorphisms is a ring. In $\A(U_A, f^*f_! U_A)$ (resp. $\A(U_A, U_Af_!f^*)$), multiplication is composition in the Kleisli category of the monad $(T=f^*f_!(\mathunderscore), \eta, \mu)$ (resp. $(M, \eta', \mu')$) of the adjunction $f_!(\mathunderscore) \dashv f^*(\mathunderscore)$ (resp. $(\mathunderscore)f_! \dashv (\mathunderscore)f^*$), and the unit is the unit of the adjunction. 
Recall that $\eta$ is the unit of the adjunction, and $\mu_M$ is $f^*\epsilon_{f^*M}$ with $\epsilon$ the counit of the adjunction. 
Finally, the group $\A(U_A, f^*U_Bf^*)$ inherits two ring structures from the two previous Kleisli compositions. 
 
Then, we show that multiplication is preserved by steps~\ref{eq:comp1} and~\ref{eq:comp4}. This boils down to the following commutative diagram for all $\alpha, \beta \in \A(f_!U_A, f_! U_A)$ (the case of Equation~\eqref{eq:comp4} is similar):
\[\begin{tikzcd}[column sep=3em, row sep=normal]
	&&&& {TU_A} \\
	{U_A} & {TU_A} & {TU_A} &[6pt] {TU_A} &&[-6pt] {T^2U_A} \\
	&&&& {T^2U_A}
	\arrow["{\eta_{U_A}}", from=2-1, to=2-2]
	\arrow["{f^*\alpha}", from=2-2, to=2-3]
	\arrow["{f^*\beta}", curve={height=-10pt}, from=2-3, to=1-5]
	\arrow["{\id_{U_A}}"{description}, from=2-3, to=2-4]
	\arrow["{T\eta_{U_A}}"', curve={height=10pt}, from=2-3, to=3-5]
	\arrow["{f^*\beta}"{description}, from=2-4, to=1-5]
	\arrow["{\mu_{U_A}}"{description}, from=2-6, to=1-5]
	\arrow["{\mu_{U_A}}"{description}, from=3-5, to=2-4]
	\arrow["{Tf^*\beta}"{description}, from=3-5, to=2-6]
\end{tikzcd}\]
where the rightmost square commutes by functoriality of $f^*(\mathunderscore)$ applied to the naturality of the counit $\epsilon: f_!f^*(\mathunderscore) \to \Id$. 
Moreover, steps~\ref{eq:comp1} and~\ref{eq:comp4} each send the identity morphism to the unit of the corresponding adjunction.

Then, for all elements $\alpha,\beta \in \A(U_A, f^*U_Bf^*)$, the composition $\beta \circ_T \alpha$ induced by $T$ is $\sigma \circ (\mu_{U_A} \circ T (\sigma^{-1} \circ \beta) \circ \sigma^{-1} \circ \alpha)$, and the induced unit $1_T$ is $\sigma \circ \eta_{U_A}$. Similarly, $M$ yields a composition $\beta \circ_M \alpha$ and a unit $1_M$. Let us show that $1_T=1_M$, by showing that $1_T$ and $1_M$ both are the factorization if $Uf$ through the cartesian cell $f^*U_Bf^* \to U_B$. For $1_T$,  it suffices that the following diagram commutes:
\begin{equation}
    \label{diag:identity}
    \begin{tikzcd}[column sep=scriptsize]
	&& {f^*f_!U_A} & {f^*(B_f \odot U_A)} \\
	&& {f_!U_A} & {B_f \odot U_A} & {f^*B_f} \\
	&& {U_A\odot U_A} && {B_f = U_Bf^*} & {f^*U_Bf^*} \\
	{U_A} &&&&& {U_B}
	\arrow["{f^*\lambda_!^{-1}}", from=1-3, to=1-4]
	\arrow["\ct"{description}, from=1-3, to=2-3]
	\arrow["\ct"{description}, from=1-4, to=2-4]
	\arrow["{f^* \mathfrak r_{B_f}}", from=1-4, to=2-5]
	\arrow["{\lambda_!^{-1}}", from=2-3, to=2-4]
	\arrow["{\mathfrak r_{B_f}}", from=2-4, to=3-5]
	\arrow["\ct"{description}, from=2-5, to=3-5]
	\arrow["{\exists ! \varphi}", from=2-5, to=3-6]
	\arrow["{\chi_f \odot \id_{U_A}}"{description}, from=3-3, to=2-4]
	\arrow["\ct"{description}, from=3-5, to=4-6]
	\arrow["\ct"{description}, from=3-6, to=3-5]
	\arrow["\ct"{description}, from=3-6, to=4-6]
	\arrow["{\eta_{U_A}}", from=4-1, to=1-3]
	\arrow["\cct"{description}, from=4-1, to=2-3]
	\arrow["{\mathfrak r^{-1}_{U_A}}"{description}, from=4-1, to=3-3]
	\arrow["{\chi_f}"{description}, from=4-1, to=3-5]
	\arrow["Uf"', from=4-1, to=4-6]
\end{tikzcd}
\end{equation}
Symmetrically, one shows that $1_M$ is also the factorization of $Uf$. 

Using similar ideas, we show $\beta \circ_T \alpha = m \circ (\alpha \odot \beta) \circ \mathfrak r^{-1}_{U_A}$ and $\beta \circ_M \alpha = m \circ (\beta \odot \alpha) \circ \mathfrak l^{-1}_{A}$, where $m$ is the factorization depicted in the following diagram. Lemma~\ref{lemma:reql} will then complete the argument.
\[\begin{tikzcd}
	{f^*U_Bf^* \odot f^*U_Bf^*} & {f^*U_Bf^*} \\
	{U_B \odot U_B} & {U_B}
	\arrow["{\exists! m}", from=1-1, to=1-2]
	\arrow["{\ct \odot \ct}"{description}, from=1-1, to=2-1]
	\arrow["\ct"{description}, from=1-2, to=2-2]
	\arrow["{\mathfrak l_{U_B}^{-1}}", from=2-1, to=2-2]
\end{tikzcd}\]
For $\beta \circ_T \alpha$ ((The proof for $\beta \circ_M \alpha$ is similar), consider the commutative diagram in Figure~\ref{fig:diag1}, in which cell $A$ commutes as shown in the following diagram:
\[\begin{tikzcd}
	{U_Bf^* \odot f^*f_!U_A} && {U_B\odot f_!U_A} \\
	\\
	& {U_A \odot f^*f_!U_A} \\
	& {f^*f_!U_A} \\
	{f_!(f^*f_!U_A)} && {f_!U_A}
	\arrow["{\ct \odot \ct}", from=1-1, to=1-3]
	\arrow["{{\lambda_!}_{f^*f_!U_A}}"', from=1-1, to=5-1]
	\arrow["{\mathfrak l_{f_!U_A}}", from=1-3, to=5-3]
	\arrow["{\chi_f\odot \id_{f^*f_!U_A}}"{description}, from=3-2, to=1-1]
	\arrow["{Uf \odot \ct}"{description}, from=3-2, to=1-3]
	\arrow["{\mathfrak l_{f^*f_!U_A}}", from=3-2, to=4-2]
	\arrow["\cct", from=4-2, to=5-1]
	\arrow["\ct", from=4-2, to=5-3]
	\arrow["{\epsilon_{f_!U_A}}", from=5-1, to=5-3]
\end{tikzcd}\]
\begin{figure}
    \centering
\[
\rotatebox{90}{%
\begin{tikzcd}[ampersand replacement=\&, column sep=.55em, row sep=3.4em]
	{U_A} \& {f^*U_Bf^*} \&\& {f^*f_!U_A} \& {f^*f_!(f^*U_Bf^*)} \& {f^*f_!(f^*f_!U_A)} \\
	{U_A \odot U_A} \&\& {U_Bf^*} \& {f_!U_A} \\
	\& {f^*U_Bf^*\odot U_A} \& {U_Bf^*\odot U_A} \& {f_!(f^*U_Bf^*)} \& {f_!(f^*f_!U_A)} \\
	{f^*U_Bf^* \odot f^*U_Bf^*} \&\& {U_Bf^* \odot f^*U_Bf^*} \&\& {U_Bf^* \odot f^*f_!U_A} \\
	\& {U_B \odot U_B} \& {U_B\odot U_Bf^*} \& {U_B\odot (U_Bf^* \odot U_A)} \& {U_B\odot f_!U_A} \\
	\& {U_B} \& {U_Bf^*} \& {U_Bf^* \odot U_A} \& {f_!U_A} \\
	{f^*U_Bf^*} \&\&\&\&\& {f^*f_!U_A}
	\arrow["\alpha", from=1-1, to=1-2]
	\arrow["{\mathfrak r_{U_A}^{-1}}"', from=1-1, to=2-1]
	\arrow["{\sigma^{-1}}", from=1-2, to=1-4]
	\arrow["\ct"{description}, from=1-2, to=2-3]
	\arrow["{\mathfrak r_{f^*U_Bf^*}^{-1}}"{description}, from=1-2, to=3-2]
	\arrow["{f^*f_!\beta}", from=1-4, to=1-5]
	\arrow["\ct"{description}, from=1-4, to=2-4]
	\arrow["{f^*f_!\sigma^{-1}}", from=1-5, to=1-6]
	\arrow["\ct"{description}, from=1-5, to=3-4]
	\arrow["\ct"{description}, from=1-6, to=3-5]
	\arrow["{\mu_{U_A}}", from=1-6, to=7-6]
	\arrow["{\alpha \odot \id_{U_A}}"{description}, from=2-1, to=3-2]
	\arrow["{\alpha \odot \beta}"', from=2-1, to=4-1]
	\arrow["{\mathfrak r_{U_Bf^*}^{-1}}"', from=2-3, to=3-3]
	\arrow["{\lambda_!^{-1}}"', from=2-4, to=3-3]
	\arrow["{f_!\beta}"', from=2-4, to=3-4]
	\arrow["{\ct\odot \id_{U_A}}", from=3-2, to=3-3]
	\arrow["{\id_{f^*U_Bf^*} \odot \beta}"{description}, from=3-2, to=4-1]
	\arrow["{\id_{U_Bf^*} \odot \beta}"', from=3-3, to=4-3]
	\arrow["{f_!\sigma^{-1}}", from=3-4, to=3-5]
	\arrow["{{\lambda_!^{-1}}_{f^*U_Bf^*}}"{description}, from=3-4, to=4-3]
	\arrow["{{\lambda_!^{-1}}_{f^*f_!U_A}}"', from=3-5, to=4-5]
	\arrow[""{name=0, anchor=center, inner sep=0}, "{\epsilon_{f_!U_A}}", curve={height=-60pt}, from=3-5, to=6-5]
	\arrow["{\ct\odot \id_{f^*U_Bf^*}}"{description}, from=4-1, to=4-3]
	\arrow["{\ct \odot \ct}"{description}, from=4-1, to=5-2]
	\arrow["m"', from=4-1, to=7-1]
	\arrow["{\ct \odot \ct}"{description}, from=4-3, to=5-2]
	\arrow["{U_Bf^* \odot \sigma}"', from=4-5, to=4-3]
	\arrow["{\ct \odot \ct}"', from=4-5, to=5-5]
	\arrow["{ \mathfrak l_{U_B}}"', from=5-2, to=6-2]
	\arrow["{\id_{U_B} \odot  \ct}", from=5-3, to=5-2]
	\arrow["{ \mathfrak l_{U_Bf^*}}"', from=5-3, to=6-3]
	\arrow["{\id_{U_B} \odot \mathfrak r_{U_Bf^*}}", from=5-4, to=5-3]
	\arrow["{ \mathfrak l_{(U_Bf^* \odot U_A)}}", from=5-4, to=6-4]
	\arrow["{\id_{U_B} \odot \lambda_!^{-1}}", from=5-5, to=5-4]
	\arrow["{ \mathfrak l_{f_!U_A}}"', from=5-5, to=6-5]
	\arrow["\ct", from=6-3, to=6-2]
	\arrow["{\mathfrak r_{U_Bf^*}}", from=6-4, to=6-3]
	\arrow["{\lambda_!^{-1}}", from=6-5, to=6-4]
	\arrow["\ct"{description}, from=7-1, to=6-2]
	\arrow["\ct"{description}, from=7-6, to=6-5]
	\arrow["\sigma", from=7-6, to=7-1]
	\arrow["A", shift right=5, draw=none, from=0, to=5-5]
\end{tikzcd}
}
\]
    \caption{Proof of $\beta \circ_T \alpha = m \circ (\alpha \odot \beta) \circ \mathfrak r^{-1}_{U_A}$.}\label{fig:diag1}
\end{figure}
We have just shown that $\swap_f$ is a ring isomorphism. 
Now, if $A=I$, $f$ must be of the form $\iota_B : I \to B$. We denote $\sigma$ by $\sigma_B$, $\sigma'$ by $\sigma'_B$, $\swap_{\iota_B}$ by $\swap_B$, and ${\sigma'_B}^{-1} \circ \sigma_B$ by $\kappa_B$.
Let us start by showing that $\swap_B$ is an $\End(U_I)$-algebra isomorphism, i.e., that the above chain of isomorphisms preserves the action of $\End(U_I)$. For all elements $r \in \End(U_I)$ and $\beta \in \A({\iota_B}_!, {\iota_B}_!)$, the following equation (and its symmetric) holds:
\begin{align*}
    \iota_B^*(r \act \beta) \circ \eta &=  \iota_B^*\beta \circ \iota_B^*{\iota_B}_! r \circ\eta\\
    &= \iota_B^*r \circ \eta \circ r\\
    &= r \act (\iota_B^*r \circ \eta).
\end{align*}
Finally, let us show that the isomorphisms $\swap_B$ for all objects $B$ assemble into a natural isomorphism. Recall the notations of Proposition~\ref{prop:general_functoriality}. For simplicity, we write $(\mathunderscore)^{C,D}$ for the $\End(U_I)$-algebra homomorphism $r_{f,g}$ for all vertical arrows $f: A \to C$ and $g: B \to D$ when there in no ambiguity. We also denote $(\mathunderscore)^{C,I}$ by $(\mathunderscore)^{C}$ for all vertical arrows $f: A \to C$.
We now have to show that for all vertical arrows $f: A \to B$ and elements $\alpha \in R(A,I)$, $\sigma_A(\alpha)^{I,B} = \sigma_B(\alpha^B)$ where for all vertical objects $C$, $\sigma_C$ denote the isomorphism $R(C,I) \to R(I,C)^{\mathrm{op}}$ we have just constructed.
By the uniqueness property of $\sigma_A(\alpha)^{I,B}$, it suffices to show that the following diagram commutes:
\[\begin{tikzcd}
	{U_I{\iota_A}_!} & {U_I{\iota_A}_!} \\
	{U_I{\iota_B}_!} & {U_I{\iota_B}_!}
	\arrow["{\sigma_A(\alpha)}", from=1-1, to=1-2]
	\arrow["{\widehat {I,f}}"', from=1-1, to=2-1]
	\arrow["{\widehat {I,f}}", from=1-2, to=2-2]
	\arrow["{\sigma_B(\alpha^B)}", from=2-1, to=2-2]
\end{tikzcd}\]
Unfolding the definition of $\sigma_A(\alpha)$ and $\sigma_B(\alpha^B)$ and the fact that cocartesian cells are epic~\ref{fact1}, it suffices to show the commutativity of the following diagram:
\[\begin{tikzcd}
	{U_I\iota_{A_!}} & {U_I} & {\iota_A^*{\iota_A}_!U_I} & { \iota_A^*{\iota_A}_!U_I} & {U_I{\iota_A}_!\iota_A^*} & {U_I\iota_{A_!}} \\
	& {U_I} \\
	{U_I\iota_{B_!}} & {U_I} & {\iota_B^*{\iota_B}_!U_I} & {\iota_B^*{\iota_B}_!U_I} & {U_I\iota_!\iota^*} & {U_I\iota_{B_!}}
	\arrow["{\sigma_A(\alpha)}"{description}, curve={height=-20pt}, from=1-1, to=1-6]
	\arrow["{\widehat {I,f}}"', from=1-1, to=3-1]
	\arrow[from=1-2, to=1-1]
	\arrow["{\eta_{U_A}}"', from=1-2, to=1-3]
	\arrow["{\iota^*\alpha}"', from=1-3, to=1-4]
	\arrow["{\kappa_A}"', from=1-4, to=1-5]
	\arrow["\ct"', from=1-5, to=1-6]
	\arrow["{\widehat {I,f}}", from=1-6, to=3-6]
	\arrow["{\id_{U_I}}", from=2-2, to=1-2]
	\arrow["{\id_{U_I}}"', from=2-2, to=3-2]
	\arrow["{\sigma_B(\alpha^B)}"{description}, curve={height=20pt}, from=3-1, to=3-6]
	\arrow[from=3-2, to=3-1]
	\arrow["{\eta_{U_B}}", from=3-2, to=3-3]
	\arrow["{\iota_B^* (\alpha^B)}", from=3-3, to=3-4]
	\arrow["{\kappa_B}", from=3-4, to=3-5]
	\arrow["\ct", from=3-5, to=3-6]
\end{tikzcd}\]
where the leftmost diagram commutes by definition of $\widehat {I,f}$, and the top and bottom diagrams by definition of $\sigma_A$ and $\sigma_B$. The commutativity of the middle diagram follows from that of the following diagram:
\[\begin{tikzcd}[column sep=small]
	& {\iota_A^*{\iota_A}_!U_I} & {\iota_A^*{\iota_A}_!U_I} && {{\iota_A}^*U_A{\iota_A}^* } && {U_I{\iota_A}_!\iota_A^*} \\
	& {{\iota_A}_!U_I} & {{\iota_A}_!U_I} & {U_A\iota_A^*} & {U_A} & {\iota_A^*U_A} & {U_I\iota_{A_!}} \\
	{U_I} &&& A && B \\
	& {{\iota_B}_!U_I} & {{\iota_B}_!U_I} & {U_B\iota_B^*} & {U_B} & {\iota_B^*U_B} & {U_I\iota_{B_!}} \\
	& {\iota_B^*{\iota_B}_!U_I} & {\iota_B^*{\iota_B}_!U_I} && {{\iota_B}^*U_B{\iota_B}^* } && {U_I{\iota_B}_!\iota_B^*}
	\arrow["{\iota^*\alpha}", from=1-2, to=1-3]
	\arrow["\ct", from=1-2, to=2-2]
	\arrow["{\sigma_A}", from=1-3, to=1-5]
	\arrow["\ct", from=1-3, to=2-3]
	\arrow["{\sigma'_A}", from=1-5, to=1-7]
	\arrow["\ct"', from=1-5, to=2-4]
	\arrow["\ct", from=1-5, to=2-6]
	\arrow["\ct", from=1-7, to=2-7]
	\arrow["\alpha", from=2-2, to=2-3]
	\arrow["{\widehat {f,I}}", from=2-2, to=4-2]
	\arrow["\mathfrak r_{U_A} \circ \lambda_!^{-1}"', from=2-3, to=2-4]
	\arrow["{\widehat {f,I}}", from=2-3, to=4-3]
	\arrow["\ct"', from=2-4, to=2-5]
	\arrow["Uf"{description}, from=2-5, to=4-5]
	\arrow["\ct", from=2-6, to=2-5]
	\arrow["\mathfrak l_{U_A} \circ \rho_!^{-1}"', tail reversed, no head, from=2-6, to=2-7]
	\arrow["{\widehat {I,f}}", from=2-7, to=4-7]
	\arrow["{\eta_{U_A}}", from=3-1, to=1-2]
	\arrow["\cct"', from=3-1, to=2-2]
	\arrow["\cct", from=3-1, to=4-2]
	\arrow["{\eta_{U_B}}"', from=3-1, to=5-2]
	\arrow["{\alpha^B}", from=4-2, to=4-3]
	\arrow["\mathfrak r_{U_B} \circ \lambda_!^{-1}", from=4-3, to=4-4]
	\arrow["\ct", from=4-4, to=4-5]
	\arrow["\ct"', from=4-6, to=4-5]
	\arrow["\mathfrak l_{U_B} \circ \rho_!^{-1}", tail reversed, no head, from=4-6, to=4-7]
	\arrow["\ct"', from=5-2, to=4-2]
	\arrow["{\iota_B^*( \alpha^B)}"', from=5-2, to=5-3]
	\arrow["\ct"', from=5-3, to=4-3]
	\arrow["{\sigma_B}"', from=5-3, to=5-5]
	\arrow["\ct", from=5-5, to=4-4]
	\arrow["\ct"', from=5-5, to=4-6]
	\arrow["{\sigma'_B}"', from=5-5, to=5-7]
	\arrow["\ct"', from=5-7, to=4-7]
\end{tikzcd}\]
where the diagrams $A$ and $B$ commute: for $A$, by definition of $\widehat {f,I}$, it suffices to check that the following diagram commutes (the proof that $B$ commutes is symmetric): 
\[\begin{tikzcd}
	{{\iota_A}_!U_I} & {U_A{\iota_A}^*} & {U_A} \\
	{U_I} \\
	{{\iota_B}_!U_I} & {U_B{\iota_B}^*} & {U_B}
	\arrow["{\mathfrak r_{U_A} \circ \lambda^{-1}_!}", from=1-1, to=1-2]
	\arrow["\ct", from=1-2, to=1-3]
	\arrow["Uf", from=1-3, to=3-3]
	\arrow["\cct", from=2-1, to=1-1]
	\arrow["{U(\iota_A)}"{description}, from=2-1, to=1-3]
	\arrow["\cct"', from=2-1, to=3-1]
	\arrow["{U(\iota_B)}"{description}, from=2-1, to=3-3]
	\arrow["{\mathfrak r_{U_B} \circ \lambda^{-1}_!}"', from=3-1, to=3-2]
	\arrow["\ct"', from=3-2, to=3-3]
\end{tikzcd}\]
where the top and bottom triangles commute because of the commutativity of the following diagram for all objects $C$:
\[\begin{tikzcd}[baseline=(\tikzcdmatrixname-\the\pgfmatrixcurrentrow-1.base)]
	{U_I} & {U_I \odot U_I} & {U_I} \\
	{ {\iota_C}_!U_I} & {U_C \iota_C^* \odot U_I} & {U_C\iota_C^*} & {U_C}
	\arrow["{\mathfrak l_{U_I}^{-1} = \mathfrak r_{U_I}^{-1}}"', from=1-1, to=1-2]
	\arrow["{\id_{U_I} }", curve={height=-18pt}, from=1-1, to=1-3]
	\arrow["\cct"', from=1-1, to=2-1]
	\arrow["{\mathfrak r_{U_I}}"', from=1-2, to=1-3]
	\arrow["{\chi_{\iota_C}\odot\id_{U_I}}"{description}, from=1-2, to=2-2]
	\arrow["{\chi_{\iota_C}}"{description}, from=1-3, to=2-3]
	\arrow["{U(\iota_C)}", from=1-3, to=2-4]
	\arrow["{{\lambda^{-1}_!}_{U_C}}"', from=2-1, to=2-2]
	\arrow["{\mathfrak r_{U_C\iota_C^*}}"', from=2-2, to=2-3]
	\arrow["\ct"', from=2-3, to=2-4]
\end{tikzcd}\qedhere\]
\end{proof}
\begin{example}
    In $\BimodZ$, for all ring homomorphisms $f: A \to B$, we immediately have $\End(f_!U_A) \cong \End(U_Bf^*) \cong B$ and $\End(U_Af_!) \cong \End(f^*U_B) \cong B^{\mathrm{op}}$. 
\end{example}
Before using Proposition~\ref{prop:op} to define a natural transformation $$F_0(A) \otimes_{\End(U_I)} F_0(B)^{\mathrm{op}} \to R(A,B),$$ we introduce the following useful $2$-cells, also involved in Section~\ref{section:compatibility}. As in Proposition~\ref{prop:op}, the following results could be stated more generally, but without the naturality and $\End(U_I)$-algebra properties.
\begin{definition}\label{def:ii}
For all objects $A,B, C$, consider the cocartesian $2$-cells $U_I \xrightarrow{\cct} {\iota_A}_!U_I{\iota_B}_!$ and $U_I \xrightarrow{\mathfrak l^{-1}} U_I \odot U_I  \xrightarrow{\cct \odot \cct} {\iota_A}_! U_I \odot U_I{\iota_B}_!$. 
By the uniqueness property of cocartesian cells, there exists a unique globular isomorphic $2$-cell ${i_1}_{A,B}$ making the following diagram commute:
\[\begin{tikzcd}
	{U_I} & {U_I \odot U_I} \\
	{{\iota_A}_!U_I{\iota_C}_!} & {{\iota_A}_!U_I \odot U_I{\iota_C}_!} & {{\iota_A}_!U_I{\iota_B}_! \odot {\iota_B}_!U_I{\iota_C}_!}
	\arrow["{\mathfrak l_{U_I}^{-1}}", from=1-1, to=1-2]
	\arrow["\cct"', from=1-1, to=2-1]
	\arrow["{\cct \odot \cct}"{description}, from=1-2, to=2-2]
	\arrow["{\cct \odot \cct}", from=1-2, to=2-3]
	\arrow["{\exists ! {i_1}_{A,C}}"', from=2-1, to=2-2]
	\arrow["{\cct \odot \cct}"', from=2-2, to=2-3]
\end{tikzcd}\]
Then define the cocartesian $2$-cell ${i_2}_{A,B,C} : {\iota_A}_!U_I{\iota_C}_! \to {\iota_A}_!U_I{\iota_B}_! \odot {\iota_B}_!U_I{\iota_C}_!$ as  $(\cct \odot \cct) \circ {i_1}_{A,C}$.
\end{definition}
\begin{proposition}\label{prop:tensorop}
There is a natural transformation of rings: $$\psi: R(\mathunderscore,I) \otimes_{\End(P)} R(\mathunderscore, I)^{\mathrm{op}} \Rightarrow R(\mathunderscore, \mathunderscore).$$ 
\end{proposition}
\begin{proof}
    By Definition~\ref{def:afb}, horizontal composition $\odot$ induces additive functors between local categories, hence for any objects $A$ and $B$, the map: $$\odot: \End({\iota_A}_!U_I) \times \End(U_I{\iota_B}_!) \xrightarrow{} \End({\iota_A}_!U_I{\iota_B}_!)$$ is bilinear and preserves composition.
     We show it also preserves the action of $\End(P)$. For all elements $p \in \End(P)$, $\alpha \in \End({\iota_A}_!U_I) $ and $\beta \in \End(U_I{\iota_B}_!)$, we have  $$(p \act \alpha) \odot \beta = (\alpha \odot \beta) \circ ({\iota_A}_!p \odot \id_{U_I{\iota_B}_!}).$$ However, by functoriality of extension of scalars, Lemma~\ref{lemma:simple_natiso} and Proposition~\ref{prop:endui_commut}, we have:
    \begin{align*}
        {\iota_A}_!p \odot \id_{U_I{\iota_B}_!} &= {\iota_A}_!p \odot \id_{U_I}{\iota_B}_!\\ &= {\iota_A}_! \id_{U_I} \odot p{\iota_B}_!\\
        &=  \id_{{\iota_A}_!U_I} \odot p{\iota_B}_!.
    \end{align*}
    Thus horizontal composition yields an $\End(U_I)$-algebra homomorphism: $$\End({\iota_A}_!U_I) \otimes_{\End(P)} \End(U_I{\iota_B}_!) \xrightarrow{\odot} \End({\iota_A}_!U_I{\iota_B}_!).$$
    Let us now define $\psi_{A,B}$ as the composite: 
    \begin{align*}
    &\End({\iota_A}_!U_I) \otimes_{\End(P)} \End({\iota_B}_!U_I)^{\mathrm{op}}\\ &\downarrow{\cong}\\
    &\End({\iota_A}_!U_I) \otimes_{\End(P)} \End(U_I{\iota_B}_!)\\ 
    &\downarrow{\odot}\\ 
    &\End({\iota_A}_!U_I\odot U_I{\iota_B}_!)\\
    &\downarrow {\cong}\\ 
    &\End({\iota_A}_!U_I{\iota_B}_!) 
    \end{align*}
    where the last isomorphism is ${i_1}_{A,B} \circ (\mathunderscore) \circ {i_1}_{A,B}^{-1}$. 
    The transformation $\psi_{A,B}$ is the vertical compositions of two natural transformations: $R(\mathunderscore,I) \otimes_{\End(P)} R(\mathunderscore, I)^{\mathrm{op}} \Rightarrow R(\mathunderscore, I) \otimes_{\End(P)} R(I, \mathunderscore)$ and $R(\mathunderscore, I) \otimes_{\End(P)} R(I, \mathunderscore) \Rightarrow R(\mathunderscore, \mathunderscore)$. The first is natural by Proposition~\ref{prop:op}. For the second, with the notations of the proof of Proposition~\ref{prop:op}, we have to show $(\alpha \odot\beta)^{A,B} = \alpha^{A,I} \odot \beta^{I,B}$ for all compatible vertical objects $A,B$ and globular $2$-cells $\alpha$ and $\beta$. This follows from the identities $\widehat {f,g} = \widehat {f,\id_I} \odot \widehat {\id_I, g}$ for all compatible vertical arrows $f$ and $g$, which hold by the uniqueness property of the $\widehat {f,g}$ along with the functoriality of $\odot$.
\end{proof}
\begin{example}
    In $\BimodZ$, $\psi$ is a natural isomorphism, which amounts to the identity transformation $A \otimes_{\Z} B \Rightarrow A \otimes_{\Z} B$.
\end{example}
\begin{proposition}
For all categories $\C$, natural isomorphism $\eta: R' \Rightarrow R$ of functors $R',R: \C \to \ring$, morphism $f: A \to B$ in $\C$, left $RA$-modules $M$ and $RB$-modules $N$, and left $Rf$-linear module homomorphism $\alpha: M \to N$, restriction of scalars along $\eta$ makes $M$ a left $R'A$-module, $N$ a left $R'B$-modules, and $\alpha$ left $R'f$-linear. Moreover, this construction is functorial. 
\end{proposition}
\begin{proof}
    Restriction of scalars along $\eta_A$ makes $M$ a left $R'A$-module with the action of $R'A$ defined for all elements $a\in R'A$ and $m \in M$, by $a \act m = \eta_A(a) \act_M m$. 
     Similarly,  $N$ becomes a left $R'B$-module via restriction along $\eta_B$.
    Then, for all elements $a$ in $R'A$ and $m$ in $M$:
    \begin{align*}
        \alpha(a \act m) &= \alpha(\eta_A(A) \act_M m)\\
        &= Rf(\eta_A(a)) \act_N \alpha(m)\\
        &= \eta_B(R'f(a)) \act_N \alpha(m)\\
        &= R'f(a) \act \alpha(m).
    \end{align*}
     Functoriality is straightforward to verify.
\end{proof}
We can now define the horizontal functor $F_1$, valued in bimodules over $\End(U_I)$-algebras.
\begin{definition}
Let $F_1$ be the restriction of scalars along $\psi$ of the functor $G$, i.e., $F_1$ sends a $1$-cells $M: A \tobar B$ to the left $F_0(A) \otimes_{\End(Z)} F_0(B)^{\mathrm{op}}$-module over rings $G(M)$. By Proposition~\ref{prop:algebra}, the rings $F_0(A)$ and $F_0(B)$ are $\End(U_I)$-algebras, and $F_1(M)$ inherits the structure of a left $F_0(A) \otimes_{\End(Z)} F_0(B)^{\mathrm{op}}$ module over $\End(U_I)$-algebras, or equivalently, of an $F_0(A)$-$F_0(B)$-bimodule over $\End(U_I)$-algebras. 
\end{definition}
By definition of $F_0$ and $F_1$, we have $L \circ F_1 = F_0 \circ L$ and $R \circ F_1 = F_0 \circ R$. To show that $F_0$ and $F_1$ assemble into a framed functor $\AAA \to \BimodZ$, it remains to exhibit natural transformations $F_\odot : F_1(M) \odot F_1(N) \to F_1(M \odot N)$ and $F_U : U_{F_0(A)} \to F_1(U_A)$, satisfying the usual axioms for $2$-functors.
\subsubsection{Compatibility conditions}\label{section:compatibility}
We define the natural transformations $F_\odot$ and $F_U$. Since the projective compact generator used to define the functors $F_0$ and $F_1$ is $U_I$, we can use the natural isomorphism $\mathfrak l_{U_I}: U_I \odot U_I \to U_I$ and the $2$-cell $U(\iota_A): U_I \to U_A$ for all objects $A$. 
\begin{definition}
    For all $1$-cells $M: A \tobar B$ and $N: B \tobar C$,
    let ${F_\odot} _{M,N}: F_1(M) \otimes_{\Z} F_1(N) \to F_1(M \odot N)$ be the abelian group homomorphism defined on pure tensors $(m,n) \in F_1(M) \otimes_{\Z} F_1(N) $ by $(m\odot n )\circ {i_2}_{A,B,C}$, extended linearly using additivity of $\odot$ and composition in an abelian category.
\end{definition}
\begin{proposition}
The abelian group homomorphism ${F_\odot} _{M,N}$ assemble into a natural transformation.
\end{proposition} 
\begin{proof} 
    It suffices for all $1$-cells $M': A' \tobar B'$, $N': B' \tobar C' $, $2$-cells $\alpha: M \ccs{f}{g} M'$ and $\beta: N \ccs{g}{h} N'$, and pure tensor $(m,n) \in F_1(M) \otimes_{\Z} F_1(N)$ to show $G(\alpha \odot \beta)({F_\odot}_{M,N}(m,n)) = (G(\alpha)(m) \odot G(\beta)(n)) \circ {i_2}_{A',B',C'}$. This follows from commutativity of the diagram:
\[\begin{tikzcd}
	{{\iota_A}_!U_I{\iota_C}_!} & {{\iota_A}_!U_I{\iota_B}_! \odot {\iota_B}_!U_I{\iota_C}_!} & {M \odot N} \\
	{{\iota_A'}_!U_I{\iota_C'}_!} & {{\iota_A'}_!U_I{\iota_B'}_! \odot {\iota_B'}_!U_I{\iota_C'}_!} & {M' \odot N'}
	\arrow["{{i_2}_{A,B,C}}", from=1-1, to=1-2]
	\arrow["{\widehat{f,h}}"', from=1-1, to=2-1]
	\arrow["{m \odot n}", from=1-2, to=1-3]
	\arrow["{\widehat{f,g} \odot \widehat{g,h}}"{description}, from=1-2, to=2-2]
	\arrow["{\alpha \odot \beta}", from=1-3, to=2-3]
	\arrow["{{i_2}_{A',B',C'}}"', shift right=2, draw=none, from=2-1, to=2-2]
	\arrow[from=2-1, to=2-2]
	\arrow["{G(\alpha)(m) \odot G(\beta)(n)}"', shift right=2, draw=none, from=2-2, to=2-3]
	\arrow[from=2-2, to=2-3]
\end{tikzcd}\]
where the rightmost cell commutes by definition of $G(\alpha)(m)$ and $G(\beta)(n)$. To see that the left square commutes, precompose by the epic cocartesian $2$-cell $U_I \xrightarrow{\cct} {\iota_A}_!U_I{\iota_C}_!$. By Definition~\ref{def:ii} of $i_2$, and by definition of $\widehat{f,h}$, the commutativity of the following diagram suffices:
\[\begin{tikzcd}[row sep=1em]
	&& {{{\iota_A}_!U_I{\iota_B}_! \odot {\iota_B}_!U_I{\iota_C}_!}} \\
	{U_I} & {U_I \odot U_I} \\
	&& {{{\iota_A'}_!U_I{\iota_B'}_! \odot {\iota_B'}_!U_I{\iota_C'}_!}}
	\arrow["{\widehat{f,g} \odot \widehat{g,h}}", from=1-3, to=3-3]
	\arrow["{\mathfrak l_{U_I}^{-1}}", from=2-1, to=2-2]
	\arrow["\cct",from=2-2, to=1-3]
	\arrow["\cct", from=2-2, to=3-3]
\end{tikzcd}\]
where the triangle commutes by definition of $\widehat{f,g}$ and $\widehat{g,h}$. 
\end{proof}
\begin{proposition}
    For all $1$-cells $M: A \tobar B$ and $N: B \tobar C$, ${F_\odot}_{M,N}$ is actually a $F_0(A)$-$F_0(C)$-bimodule homomorphism $F_1(M) \otimes_{F_0(B)} F_1(N) \to F_1(M \odot N)$.
\end{proposition}
\begin{proof}
We have to show that ${F_\odot}_{M,N}$ respects the actions of $F_0(A)$ and $F_0(C)$, and the actions of $F_0(B)$. For the latter, let us show that for all elements $m\in F_1(M), n \in F_1(N)$ and $\beta \in F_0(B)$, we have ${F_\odot}_{M,N}((m, \beta \act n)) = {F_\odot}_{M,N}((m \act \beta, n))$. Unfolding the definitions, we have to show that the following diagram commutes: 
\[\begin{tikzcd}[column sep=scriptsize]
	{{\iota_A}_! U_I{\iota_C}_! } & {{\iota_A}_! U_I {\iota_B}_! \odot {\iota_B}_!U_I {\iota_C}_! } & {{\iota_A}_! U_I {\iota_B}_! \odot {\iota_B}_!U_I {\iota_C}_! } & {M \odot N}
	\arrow["{i_2}_{A,B,C}", from=1-1, to=1-2]
	\arrow["{\id \odot \psi_{B,C}(\beta, \id_{{\iota_C}_!U_I})}", curve={height=-12pt}, from=1-2, to=1-3]
	\arrow["{\psi_{A,B}(\id_{{\iota_A}_!U_I}, \beta^{\mathrm{op}}) \odot \id}"', curve={height=12pt}, from=1-2, to=1-3]
	\arrow["{m \odot n}", from=1-3, to=1-4]
\end{tikzcd}\]
This boils down to the commutativity of:
\[\begin{tikzcd}
	{U_I} & {U_I \odot U_I} & {U_I {\iota_B}_! \odot {\iota_B}_!U_I} & {U_I {\iota_B}_! \odot {\iota_B}_!U_I}
	\arrow["{\mathfrak l_I^{-1}}", from=1-1, to=1-2]
	\arrow["\cct \odot \cct", from=1-2, to=1-3]
	\arrow["{\id \odot \beta}", curve={height=-12pt}, from=1-3, to=1-4]
	\arrow["{\beta^{\mathrm{op}} \odot \id}"', curve={height=12pt}, from=1-3, to=1-4]
\end{tikzcd}.\]
To prove its commutativity, it suffices to show that the following two diagrams commute:
\begin{equation}
\begin{tikzcd}[column sep=scriptsize]
	{U_I} & {U_I \odot U_I} && {U_I {\iota_B}_! \odot {\iota_B}_!U_I} \\
	& {{\iota_B}_!U_I} & {U_B \odot {\iota_B}_!U_I } \\
	& {{\iota_B}_!U_I} & {U_B \odot {\iota_B}_!U_I } \\
	{{\iota_B}^*{\iota_B}_!U_I} & {{\iota_B}^*{\iota_B}_!U_I} && {U_I {\iota_B}_! \odot {\iota_B}_!U_I}
	\arrow["{\mathfrak l_{U_I}^{-1}}", from=1-1, to=1-2]
	\arrow["\cct"{description}, from=1-1, to=2-2]
	\arrow["\eta"', from=1-1, to=4-1]
	\arrow["{\cct \odot \cct}", from=1-2, to=1-4]
	\arrow["{\id_{U_I} \odot \cct }"{description}, from=1-2, to=2-3]
	\arrow["{\id_{U_I {\iota_B}_!} \odot \beta}"{description}, from=1-4, to=4-4]
	\arrow["{\mathfrak l_{{\iota_B}_!U_I}^{-1}}"', from=2-2, to=2-3]
	\arrow["\beta", from=2-2, to=3-2]
	\arrow["{\cct \odot \id_{{\iota_B}_!U_I }}"{description}, from=2-3, to=1-4]
	\arrow["{\id_{U_B} \odot \beta}", from=2-3, to=3-3]
	\arrow["{\mathfrak l_{{\iota_B}_!U_I}^{-1}}"', from=3-2, to=3-3]
	\arrow["{\cct \odot \id_{{\iota_B}_!U_I }}"{description}, from=3-3, to=4-4]
	\arrow["\ct"{description}, from=4-1, to=2-2]
	\arrow["{{\iota_B}^* \beta}", from=4-1, to=4-2]
	\arrow["\ct"{description}, from=4-2, to=3-2]
	\arrow["\cong", from=4-2, to=4-4]
\end{tikzcd}\label{diagram:pair1}
\end{equation}
\begin{equation}\begin{tikzcd}[column sep=scriptsize]
	{U_I} & {U_I \odot U_I} && {U_I {\iota_B}_! \odot {\iota_B}_!U_I} \\
	&& {\text{Symmetric of~}\ref{diagram:pair1}} \\
	& {U_I{\iota_B}_!{\iota_B}^*} & {U_I{\iota_B}_!{\iota_B}^*} \\
	{{\iota_B}^*{\iota_B}_!U_I} && {{\iota_B}^*{\iota_B}_!U_I} & {U_I {\iota_B}_! \odot {\iota_B}_!U_I}
	\arrow["{\mathfrak l_{U_I}^{-1}}", from=1-1, to=1-2]
	\arrow["{\eta'}"{description}, from=1-1, to=3-2]
	\arrow["\eta"', from=1-1, to=4-1]
	\arrow["\cct \odot \cct", from=1-2, to=1-4]
	\arrow["{\beta^{\mathrm{op}} \odot \id_{{\iota_B}_!U_I}}"{description}, from=1-4, to=4-4]
	\arrow["{\beta^{\mathrm{op}}{\iota_B}^*}", from=3-2, to=3-3]
	\arrow["{\kappa^{-1}}"{description}, from=3-2, to=4-1]
	\arrow["{\kappa^{-1}}", from=3-3, to=4-3]
	\arrow["\cong", from=3-3, to=4-4]
	\arrow["{{\iota_B}^* \beta}", from=4-1, to=4-3]
	\arrow["\cong", from=4-3, to=4-4]
\end{tikzcd}\label{diagram:pair2}
\end{equation}

It is elementary to establish the commutativity of the former, and symmetrically of the top right cell of the latter. Likewise, the commutativity of the bottom right cell is elementary. Let us show that the remaining two cells commute. For the leftmost one, consider the following commutative diagram: 
\[\begin{tikzcd}
	&& {U_I} \\
	\\
	{U_I{\iota_B}_!\iota_B^*} && {\iota_B^*U_B\iota_B^*} && {\iota_B^*{\iota_B}_!U}
	\arrow["{\eta'}"', from=1-3, to=3-1]
	\arrow["{\fact{U(\iota_B)}}"{description}, from=1-3, to=3-3]
	\arrow["\eta", from=1-3, to=3-5]
	\arrow["{\sigma_B'}", from=3-1, to=3-3]
	\arrow["{\sigma_B^{-1}}", from=3-3, to=3-5]
\end{tikzcd}\]
For the bottom middle cell, apply the fact that cartesian cells are monic to the following commutative diagram. For the sake of readability, we write $(\mathunderscore)*$ for restriction $(\mathunderscore){\iota_B}^*$ on the right along ${\iota_B}$. Similarly, we write $(\mathunderscore)!$ for extension along ${\iota_B}$, and we use the symmetric notations for change of coefficients on the left:
\[\begin{tikzcd}
	&&& {(*!U)!*} & {(*!U)!*} & {U!*!*} & {U!*} \\
	{U_I} & {U!*} & {U!} & {(*!U)!} & {(*!U)!} & {U!*!} & {U!} \\
	&&& {*!U} & {*!U} & {U!*}
	\arrow["{(*b)!*}", from=1-4, to=1-5]
	\arrow[from=1-4, to=2-4]
	\arrow["{\kappa!*}", from=1-5, to=1-6]
	\arrow[from=1-5, to=2-5]
	\arrow["{\epsilon'*}", from=1-6, to=1-7]
	\arrow[from=1-6, to=2-6]
	\arrow[from=1-7, to=2-7]
	\arrow["{\eta'}", from=2-1, to=2-2]
	\arrow["{\eta!*}", from=2-2, to=1-4]
	\arrow[from=2-2, to=2-3]
	\arrow["{\kappa_B^{-1}}"', from=2-2, to=3-4]
	\arrow["{\eta!}", from=2-3, to=2-4]
	\arrow["{(*b)!}", from=2-4, to=2-5]
	\arrow["{\kappa!}", from=2-5, to=2-6]
	\arrow["\epsilon'", from=2-6, to=2-7]
	\arrow[from=3-4, to=2-4]
	\arrow["{*b}", from=3-4, to=3-5]
	\arrow[from=3-5, to=2-5]
	\arrow["\kappa_B", from=3-5, to=3-6]
	\arrow[from=3-6, to=2-6]
	\arrow[from=3-6, to=2-7]
\end{tikzcd}\]
where the bottom left cell commutes when precomposed by $\eta'$, as seen in the following diagram:
\[\begin{tikzcd}
	{U_I} & {U!*} && {*U_B*} && {*!U} \\
	&& {U!*!} \\
	&&& {(*U_B*)!} \\
	& {U!} &&&& {(*!U)!}
	\arrow["{\eta'}", from=1-1, to=1-2]
	\arrow["{\sigma_B'}", from=1-2, to=1-4]
	\arrow["\cct", from=1-2, to=2-3]
	\arrow["\ct"', from=1-2, to=4-2]
	\arrow["{\sigma_B^{-1}}", from=1-4, to=1-6]
	\arrow["\cct", from=1-4, to=3-4]
	\arrow["\cct", from=1-6, to=4-6]
	\arrow["{\sigma_B'!}"{description}, from=2-3, to=3-4]
	\arrow["{\sigma_B^{-1}!}"{description}, from=3-4, to=4-6]
	\arrow["{\eta!}"{description}, from=4-2, to=2-3]
	\arrow["{\fact{U\iota_B}!}"{description}, from=4-2, to=3-4]
	\arrow["{\eta!}"', from=4-2, to=4-6]
\end{tikzcd}\]
where the leftmost cell precomposed by $\eta'$ commutes, as seen in the following diagram:
\[\begin{tikzcd}
	& {U!*} \\[8pt]
	& U \\[-8pt]
	{U!} && {U!*!}
	\arrow["\ct"', from=1-2, to=3-1]
	\arrow["\cct", from=1-2, to=3-3]
	\arrow["\eta'"{description}, from=2-2, to=1-2]
	\arrow["\cct", from=2-2, to=3-1]
	\arrow["{\eta'!}", from=3-1, to=3-3]
\end{tikzcd}\]

Let us now show that ${F_\odot}_{M,N}$ respects the actions of $F_0(A)$ and $F_0(C)$, by showing that for all elements $m\in F_1(M), n \in F_1(N)$, $\alpha \in F_0(A)$ and $\gamma \in F_0(C)$, we have ${F_\odot}_{M,N}((\alpha \act m, n \act \gamma)) = \alpha \act {F_\odot}_{M,N}((m, n)) \act \gamma$. Unfolding the definitions, we have to show that the following diagram commutes: 
\[\begin{tikzcd}[column sep=3.5em, row sep=3em]
	{{\iota_A}_!U_I {\iota_C}_!} &&& {{\iota_A}_!U_I {\iota_C}_!} \\
	& |[xshift=-1.7em,overlay]| {{\iota_A}_!U_I \odot U_I{\iota_C}_!} &&& |[xshift=-1.7em,yshift=-1.4em,overlay]| {M \odot N} \\
	&& |[xshift=1.7em,overlay]| {{\iota_A}_!U_I \odot U_I{\iota_C}_!} \\
	{{\iota_A}_!U_I {\iota_B}_! \odot {\iota_B}_! U_I {\iota_C}_!} &&& {{\iota_A}_!U_I {\iota_B}_! \odot {\iota_B}_! U_I {\iota_C}_!}
	\arrow["{\psi_{A,C}(\alpha, \gamma)}", from=1-1, to=1-4]
	\arrow["{{i_1}_{A,C}}"{description}, from=1-1, to=2-2]
	\arrow["{{i_2}_{A,B,C}}"{description}, from=1-1, to=4-1]
	\arrow["{{F_1}_{M,N}(m,n)}", from=1-4, to=2-5]
	\arrow["{{i_2}_{A,B,C}}"{description}, from=1-4, to=4-4]
	\arrow["{\alpha \odot \swap_C(\gamma^{\mathrm{op}})}"{description}, from=2-2, to=3-3]
	\arrow["{\cct \odot \cct}"{description}, from=2-2, to=4-1]
	\arrow["{{i_1}_{A,C}^{-1}}"{description}, from=3-3, to=1-4]
	\arrow["{\cct \odot \cct}"{description}, from=3-3, to=4-4]
	\arrow["{\psi_{A,B}(\alpha, \id) \odot \psi_{B,C}(\id, \gamma)}"', shift right=2, draw=none, from=4-1, to=4-4]
	\arrow[from=4-1, to=4-4]
	\arrow["{m \odot n}"{description}, from=4-4, to=2-5]
\end{tikzcd}\]
where the bottom middle diagram commutes by commutativity of the following ones (and their symmetric variants):
\[\begin{tikzcd}
	{{\iota_A}_!U_I } &&& {{\iota_A}_!U_I \odot U_I} \\
	& {U_I} & {U_I \odot U_I} \\
	{{\iota_A}_!U_I{\iota_B}_!} &&& {{\iota_A}_!U_I \odot U_I{\iota_B}_!}
	\arrow["{\mathfrak r^{-1}_{{\iota_A}_! U_I}}", from=1-1, to=1-4]
	\arrow["\cct"', from=1-1, to=3-1]
	\arrow["{\id \odot \cct}", from=1-4, to=3-4]
	\arrow["\cct"{description}, from=2-2, to=1-1]
	\arrow["{\mathfrak l_{U_I}^{-1}}", from=2-2, to=2-3]
	\arrow["\cct"{description}, from=2-2, to=3-1]
	\arrow["{\cct \odot \id}"{description}, from=2-3, to=1-4]
	\arrow["{\cct \odot \cct}"{description}, from=2-3, to=3-4]
	\arrow["{{i_1}_{A,B}}", from=3-1, to=3-4]
\end{tikzcd}\]
\[\begin{tikzcd}[baseline=(\tikzcdmatrixname-\the\pgfmatrixcurrentrow-1.base)]
	{{\iota_A}_!U_I } &&& {{\iota_A}_!U_I } \\
	& {{\iota_A}_!U_I \odot U_I} & {{\iota_A}_!U_I \odot U_I} \\
	& {{\iota_A}_!U_I \odot U_I{\iota_B}_!} & {{\iota_A}_!U_I \odot U_I{\iota_B}_!} \\
	{{\iota_A}_!U_I{\iota_B}_!} &&& {{\iota_A}_!U_I{\iota_B}_!}
	\arrow["{\alpha }", from=1-1, to=1-4]
	\arrow["{\mathfrak r^{-1}_{{\iota_A}_!U_I}}"{description}, from=1-1, to=2-2]
	\arrow["\cct"', from=1-1, to=4-1]
	\arrow["\cct", from=1-4, to=4-4]
	\arrow["{\alpha \odot \id_{U_I}}", from=2-2, to=2-3]
	\arrow["{\id_{{\iota_A}_!U_I} \odot \cct}"', from=2-2, to=3-2]
	\arrow["{\mathfrak r_{{\iota_A}_!U_I}}"{description}, from=2-3, to=1-4]
	\arrow["{\id_{{\iota_A}_!U_I} \odot \cct}", from=2-3, to=3-3]
	\arrow["{\alpha \odot \id_{U_I{\iota_B}_!}}", from=3-2, to=3-3]
	\arrow["{{i_1}_{A,B}^{-1}}"{description}, from=3-3, to=4-4]
	\arrow["{{i_1}_{A,B}}"{description}, from=4-1, to=3-2]
	\arrow["{\psi(\alpha, \id_{{\iota_B}_!U_I})}", from=4-1, to=4-4]
\end{tikzcd}\label{diagram:psiid}\qedhere
\]
 \end{proof}
\begin{definition} 
For all $1$-cells $A$, let ${F_U}_A: F_0(A) \to F_1(U_A)$ be the map defined on elements $\alpha : {\iota_A}_!U_I \to {\iota_A}_!U_I \in F_0(A) = R(A,I)$ by $\fact{U\iota_A} \circ R(\id_A, \iota_A)(\alpha)$, where $\fact{U\iota_A}$ is the factorization of $U(\iota_A)$ through the cocartesian $2$-cell $U_I \xrightarrow{\cct} {\iota_A}_!U_I{\iota_A}_!$. Since composition in an abelian category is additive and $R(\id_A, \iota_A)$ is a ring homomorphism, ${F_U}_A$ is a group homomorphism.
\end{definition}
\begin{proposition}
The abelian group homomorphisms ${F_U}_A$ assemble into a natural homomorphism.
 \end{proposition}
 \begin{proof}
Naturality means that for all objects $A$ and $B$, vertical arrows $f: A \to B$, and elements $\alpha \in F_0(A)$, we have the identity: $$F_1(U_f)(\fact{U_{\iota_A}} \circ R(\id_A, \iota_A)(\alpha) = \fact{U_{\iota_A}} \circ R(\id_A, \iota_A)(F_0(f)(\alpha)).$$
By the uniqueness property in the definition of $F_1(U_f)(\fact{U_{\iota_A}} \circ R(\id_A, \iota_A)(\alpha)$, this follows from the commutativity of the following diagram: 
\[\begin{tikzcd}
	{{\iota_A}_!U_I{\iota_A}_!} && {{\iota_A}_!U_I{\iota_A}_!} && {U_A} \\
	{{\iota_B}_!U_I{\iota_B}_!} && {{\iota_B}_!U_I{\iota_B}_!} && {U_B}
	\arrow["{R(\id_A,\iota_A)(\alpha)}", from=1-1, to=1-3]
	\arrow["{\widehat {f, f}}"', from=1-1, to=2-1]
	\arrow["{\fact{U\iota_A}}", from=1-3, to=1-5]
	\arrow["{\widehat {f, f}}", from=1-3, to=2-3]
	\arrow["{U(f)}", from=1-5, to=2-5]
	\arrow["{R(\id_B,\iota_B)(R(f,\id_I)(\alpha))}"', shift right=2, draw=none, from=2-1, to=2-3]
	\arrow[from=2-1, to=2-3]
	\arrow["{\fact{U\iota_B}}"', from=2-3, to=2-5]
\end{tikzcd}\]
where the leftmost square commutes since cocartesian $2$-cells are epic, and by the following diagram where the triangle and all the faces of the square, except the bottom one, commute:
\[\begin{tikzcd}[column sep= .7em, row sep=2em]
	& {{\iota_A}_!U_I} && {{\iota_A}_!U_I} \\
	{U_I} && {{\iota_B}_!U_I} && {{\iota_B}_!U_I} \\
	& {{\iota_A}_!U_I{\iota_A}_!} && {{\iota_A}_!U_I{\iota_A}_!} \\
	&& {{\iota_B}_!U_I{\iota_B}_!} && {{\iota_B}_!U_I{\iota_B}_!}
	\arrow["\alpha", from=1-2, to=1-4]
	\arrow["{\widehat {f, \id_I}}"{description}, from=1-2, to=2-3]
	\arrow["{\widehat {\id_A,\iota_A}}"{description}, from=1-2, to=3-2]
	\arrow["{\widehat {f, \id_I}}", from=1-4, to=2-5]
	\arrow["{\widehat {\id_A,\iota_A}}"{description, pos=0.2}, from=1-4, to=3-4]
	\arrow["{ \cct}"{description}, from=2-1, to=1-2]
	\arrow["\cct"{description}, from=2-1, to=3-2]
	\arrow["{R(f,\id_I)(\alpha)}"{description}, from=2-3, to=2-5]
	\arrow["{\widehat {\id_B,\iota_B}}", from=2-5, to=4-5]
	\arrow["{R(\id_A,\iota_A)(\alpha)}"{description}, from=3-2, to=3-4]
	\arrow["{\widehat {f, f}}"', from=3-2, to=4-3]
	\arrow["{\widehat {f, f}}"{description}, from=3-4, to=4-5]
    \arrow["{\widehat {\id_B,\iota_B}}"{description, pos=0.2}, from=2-3, to=4-3]
	\arrow["{R(\id_B,\iota_B)(R(f,\id_I)(\alpha))}"', shift right=2, draw=none, from=4-3, to=4-5]
	\arrow[from=4-3, to=4-5]
\end{tikzcd}\]
and where the rightmost square commutes since cocartesian $2$-cells are epic, and by the following commutative diagram:
\[\begin{tikzcd}[baseline=(\tikzcdmatrixname-\the\pgfmatrixcurrentrow-1.base)]
	{{\iota_A}_!U_I{\iota_A}_!} && {U_A} \\
	& {U_I} \\
	{{\iota_B}_!U_I{\iota_B}_!} && {U_B}
	\arrow["{\fact{U\iota_A}}", from=1-1, to=1-3]
	\arrow["{\widehat {f, f}}", from=1-1, to=3-1]
	\arrow["{U(f)}", from=1-3, to=3-3]
	\arrow["\cct"{description}, from=2-2, to=1-1]
	\arrow["{U\iota_A}"{description}, from=2-2, to=1-3]
	\arrow["\cct"{description}, from=2-2, to=3-1]
	\arrow["{U\iota_B}", from=2-2, to=3-3]
	\arrow["{\fact{U\iota_B}}"', from=3-1, to=3-3]
\end{tikzcd}\qedhere\]
\end{proof}

\begin{proposition} For all objects $A$, ${F_U}_A$ is actually a $F_0(A)$-$F_0(A)$-bimodule of $\End(U_I)$-algebras homomorphism.
\end{proposition}
\begin{proof}
    We have to show that ${F_U}_A$ respects the two actions of $F_0(A)$ on the left and on the right. For the left action, we show for all elements $\alpha, a \in F_0(A)$ that ${F_U}_A(a \act \alpha ) = a \act {F_U}_A(\alpha)$, i.e., ${F_U}_A(\alpha \circ a) = {F_U}_A(\alpha) \circ \psi_{A,A}(a, \id_{{\iota_A}_!U_I})$. Since $R(\id_A, \iota_A)$ is a ring homomorphism, it preserves composition, and it suffices to show that $R(\id_A, \iota_A)(\alpha)$ equals $\psi_{A,A}(\alpha, \id_{{\iota_A}_!U_I})$ for all $\alpha \in F_0(A)$. We start by giving an alternative formula for the $2$-cell $\widehat{\id_A, \iota_A}$, using its uniqueness property. Consider the following commutative diagram: 
\[\begin{tikzcd}[]
	&& {{\iota_A}_!U_I} \\
	&& {{\iota_A}_!U_I \odot U_I} \\
	{U_I} & {U_I \odot U_I} & {{\iota_A}_!U_I \odot U_I{\iota_A}_!} \\
	\\
	&& {{\iota_A}_!U_I{\iota_A}_!}
	\arrow["{\mathfrak l_{{\iota_A}_!U_I}^{-1}}", from=1-3, to=2-3]
	\arrow["{\id \odot \cct}", from=2-3, to=3-3]
	\arrow["\cct", from=3-1, to=1-3]
	\arrow["{\mathfrak l_{U_I}^{-1}}"{description}, from=3-1, to=3-2]
	\arrow["\cct"', from=3-1, to=5-3]
	\arrow["{\cct \odot \id_{U_I}}"{description}, from=3-2, to=2-3]
	\arrow["{\cct \odot \cct}"', from=3-2, to=3-3]
	\arrow["{{i_1}_{A,A}^{-1}}", from=3-3, to=5-3]
\end{tikzcd}\]

We then use this alternative expression to prove that: $$R(\id_A, \iota_A)(\alpha) = \psi_{A,A}(\alpha, \id_{{\iota_A}_!U_I})$$ using the uniqueness property of $R(\id_A, \iota_A)(\alpha)$ and the commutativity of the following diagram:
\[\begin{tikzcd}[sep=small]
	&& {{\iota_A}_!U_I} &&& {{\iota_A}_!U_I} \\
	&&& {{\iota_A}_!U_I\odot U_I} & {{\iota_A}_!U_I\odot U_I} \\
	{U_I} \\
	&&& {{\iota_A}_!U_I\odot U_I{\iota_A}_!} & {{\iota_A}_!U_I\odot U_I{\iota_A}_!} \\
	&& {{\iota_A}_!U_I{\iota_A}_!} &&& {{\iota_A}_!U_I{\iota_A}_!}
	\arrow["\alpha", from=1-3, to=1-6]
	\arrow["{\mathfrak l_{{\iota_A}_!U_I}^{-1}}"', from=1-3, to=2-4]
	\arrow["{\widehat{\id_A, \iota_A}}"', from=1-3, to=5-3]
	\arrow["{\mathfrak l_{{\iota_A}_!U_I}^{-1}}", from=1-6, to=2-5]
	\arrow["{\widehat{\id_A, \iota_A}}"{description}, from=1-6, to=5-6]
	\arrow["{\alpha \odot \id_{U_I}}", from=2-4, to=2-5]
	\arrow["{\id_{{\iota_A}_!U_I} \odot \cct}"', from=2-4, to=4-4]
	\arrow["{\id_{{\iota_A}_!U_I} \odot \cct}", from=2-5, to=4-5]
	\arrow["\cct", from=3-1, to=1-3]
	\arrow["\cct"', from=3-1, to=5-3]
	\arrow["{\alpha \odot \id_{U_I{\iota_A}_!}}"', shift right=2, draw=none, from=4-4, to=4-5]
	\arrow[from=4-4, to=4-5]
	\arrow["{{i_1}_{A,A}^{-1}}"', from=4-4, to=5-3]
	\arrow["{{i_1}_{A,A}^{-1}}", from=4-5, to=5-6]
	\arrow["{\psi_{A,A}(\alpha, \id_{{\iota_A}_!U_I})}"', from=5-3, to=5-6]
\end{tikzcd}\]

For the right action, the reasoning is similar. Using the symmetric of the above argument, one shows  $R( \iota_A, \id_A)(\alpha^{\mathrm{op}}) = \psi_{A,A}(\id_{{\iota_A}_!U_I}, \alpha^{\mathrm{op}})$ for all elements $\alpha \in F_0(A)$. Then, it suffices to show for all $\alpha \in F_0(A)$:
\begin{equation}
 \fact{U\iota_A} \circ R(\id_A, \iota_A)(\alpha) = \fact{U\iota_A} \circ R( \iota_A, \id_A)(\alpha^{\mathrm{op}})\label{eq:Rop}
\end{equation}
in order to get the following equality for all elements $\alpha, a \in F_0(A)$:
\begin{align*}
    {F_U}_A(a \circ \alpha) &= \fact{U\iota_A} \circ R(\id_A, \iota_A)(a \circ \alpha)\\
    &= \fact{U\iota_A} \circ R( \iota_A, \id_A)(\alpha^{\mathrm{op}} \circ a^{\mathrm{op}})\\
    &= \fact{U\iota_A} \circ R( \iota_A, \id_A)(\alpha^{\mathrm{op}}) \circ R( \iota_A, \id_A)(a^{\mathrm{op}})\\
    &= \fact{U\iota_A} \circ R( \id_A, \iota_A)(\alpha) \circ R( \iota_A, \id_A)(a^{\mathrm{op}})\\
    &= {F_U}_A(\alpha) \circ \psi_{A,A}(\id_{{\iota_A}_!U_I}, a^{\mathrm{op}})
\end{align*}

Let us now show Equation~\ref{eq:Rop} for all $\alpha \in F_0(A)$. First, notice that by the uniqueness property of the $2$-cell $\widehat{\id_A, \iota_A}$ and by the definition of our canonical cocartesian $2$-cells, $\widehat{\id_A, \iota_A}$ also equals ${\iota_A}_!U_I \xrightarrow{\cct} {\iota_A}_!U_I {\iota_A}_!$. Then, the result follows from the commutativity of the following diagram: 
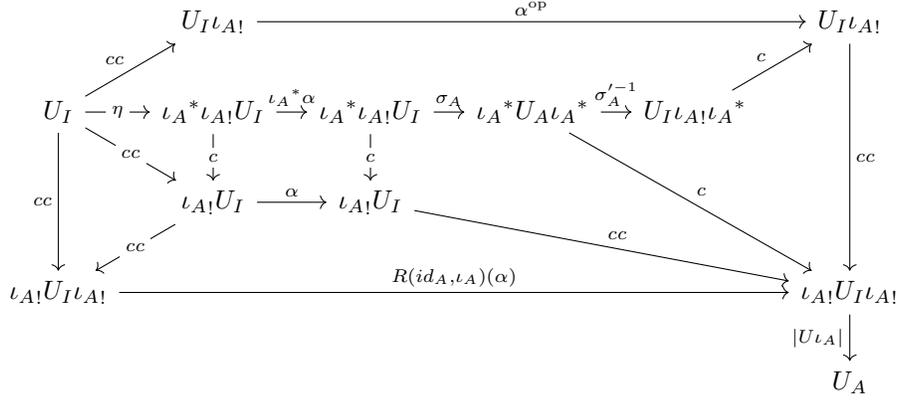
\[\begin{tikzcd}[column sep=small]
	& {U_I{\iota_A}_!} &&&& {U_I{\iota_A}_!} \\
	{U_I} & {{\iota_A}^*{\iota_A}_!U_I} & {{\iota_A}^*{\iota_A}_!U_I} & {{\iota_A}^*U_A{\iota_A}^*} & {U_I{\iota_A}_!{\iota_A}^*} \\
	& {{\iota_A}_!U_I} & {{\iota_A}_!U_I} \\
	{{\iota_A}_!U_I{\iota_A}_!} &&&&& {{\iota_A}_!U_I{\iota_A}_!} \\
	&&&&& {U_A}
	\arrow["{\alpha^{\mathrm{op}}}", from=1-2, to=1-6]
	\arrow["\cct", from=1-6, to=4-6]
	\arrow["\cct", from=2-1, to=1-2]
	\arrow["\eta"{description}, from=2-1, to=2-2]
	\arrow["\cct"{description}, from=2-1, to=3-2]
	\arrow["\cct"', from=2-1, to=4-1]
	\arrow["{{\iota_A}^*\alpha}", from=2-2, to=2-3]
	\arrow["\ct"{description}, from=2-2, to=3-2]
	\arrow["{\sigma_A}", from=2-3, to=2-4]
	\arrow["\ct"{description}, from=2-3, to=3-3]
	\arrow["{\sigma_A'^{-1}}", from=2-4, to=2-5]
	\arrow["{ \ct}", from=2-4, to=4-6]
	\arrow["\ct", from=2-5, to=1-6]
	\arrow["\alpha", from=3-2, to=3-3]
	\arrow["\cct"{description}, from=3-2, to=4-1]
	\arrow["\cct", from=3-3, to=4-6]
	\arrow["{R(id_A, \iota_A)(\alpha)}", from=4-1, to=4-6]
	\arrow["{\fact{U\iota_A}}"', from=4-6, to=5-6]
\end{tikzcd}\]
where the top cell commutes by definition of $\alpha^{\mathrm{op}}$, the bottom cell by definition of $R(\id_A, \iota_A)(\alpha)$, and the two middle rightmost cells by commutativity of the following diagram (and its symmetric): 
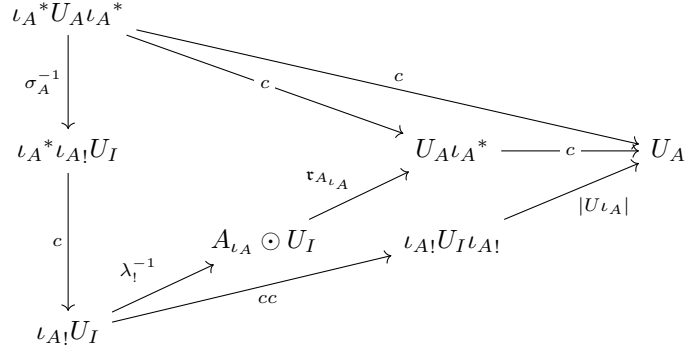
\[\begin{tikzcd}
	{{\iota_A}^*U_A{\iota_A}^*} \\
	\\
	{{\iota_A}^*{\iota_A}_!U_I} && {U_A{\iota_A}^*} && {U_A} \\
	& {A_{\iota_A} \odot U_I} & {{\iota_A}_!U_I{\iota_A}_!} \\
	{{\iota_A}_!U_I}
	\arrow["{\sigma_A^{-1}}"', from=1-1, to=3-1]
	\arrow["\ct"{description}, from=1-1, to=3-3]
	\arrow["\ct", from=1-1, to=3-5]
	\arrow["\ct"', from=3-1, to=5-1]
	\arrow["\ct"{description}, from=3-3, to=3-5]
	\arrow["{\mathfrak r_{A_{\iota_A}}}", from=4-2, to=3-3]
	\arrow["{\fact{U\iota_A}}"', from=4-3, to=3-5]
	\arrow["{\lambda_!^{-1}}", from=5-1, to=4-2]
	\arrow["\cct"', from=5-1, to=4-3]
\end{tikzcd}\]
where the leftmost triangle commutes by definition of $\sigma_A$ (see Diagram~\ref{diag:identity}), and the bottom right cell by the commutativity of the following diagram and the fact that cocartesian $2$-cells are epic:
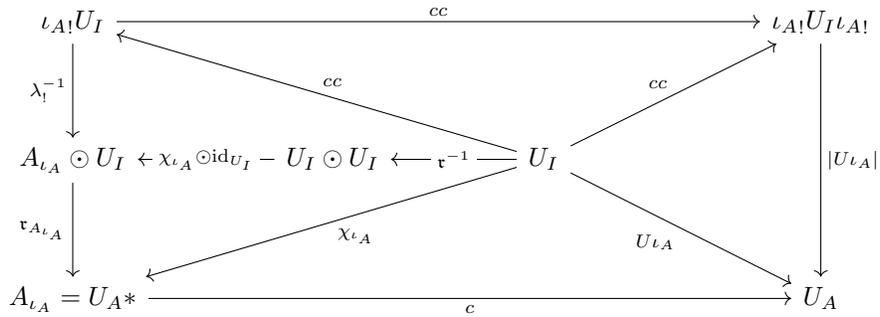
\[\begin{tikzcd}[baseline=(\tikzcdmatrixname-\the\pgfmatrixcurrentrow-1.base)]
	{{\iota_A}_!U_I} &&&&&&& {{\iota_A}_!U_I{\iota_A}_!} \\
	\\
	{A_{\iota_A} \odot U_I} && {U_I \odot U_I} && {U_I} \\
	\\
	{A_{\iota_A} =U_A*} &&&&&&& {U_A}
	\arrow["\cct", from=1-1, to=1-8]
	\arrow["{\lambda_!^{-1}}"', from=1-1, to=3-1]
	\arrow["{\fact{U\iota_A}}", from=1-8, to=5-8]
	\arrow["{\mathfrak r_{A_{\iota_A}}}"', from=3-1, to=5-1]
	\arrow["{\chi_{\iota_A} \odot \id_{U_I}}"{description}, from=3-3, to=3-1]
	\arrow["\cct"', from=3-5, to=1-1]
	\arrow["\cct", from=3-5, to=1-8]
	\arrow["{\mathfrak r^{-1}}"{description}, from=3-5, to=3-3]
	\arrow["{\chi_{\iota_A}}", from=3-5, to=5-1]
	\arrow["{U\iota_A}"', from=3-5, to=5-8]
	\arrow["\ct"', from=5-1, to=5-8]
\end{tikzcd}\qedhere\]
\end{proof}

Finally, it only remains to show that the natural transformations $F_\odot$ and $F_U$ satisfy the usual axioms for $2$-functors, i.e. that for all $1$-cells $M: A \tobar B, N: B \tobar C$ and $P: C \tobar D$, the following diagram commutes:
\[\begin{tikzcd}[]
	\substack{{F_1(M) \otimes_{F_0(B)}} \\ {(F_1(N) \otimes_{F_0(C)} F_1(P))}} & \substack{{(F_1(M) \otimes_{F_0(B)} F_1(N))} \\ {\otimes_{F_0(C)} F_1(P)}} \\
	{F_1(M) \otimes_{F_0(B)} F_1(N \odot P)} & {F_1(M \odot N) \otimes_{F_0(C)} F_1(P)} \\
	{F_1(M \odot (N \odot P)) } & {F_1((M \odot N) \odot P)) }
	\arrow["\cong", from=1-1, to=1-2]
	\arrow["{F_1(M) \otimes_{F_0(B)} {F_\odot}_{N,P}}"{description}, from=1-1, to=2-1]
	\arrow["{{F_\odot}_{M,N} \otimes_{F_0(C)} F_1(P) }"{description}, from=1-2, to=2-2]
	\arrow["{{F_\odot}_{M,N\odot P}}"', from=2-1, to=3-1]
	\arrow["{{F_\odot}_{M\odot N, P}}", from=2-2, to=3-2]
	\arrow["{F_1(\mathfrak a_{M,N,P})}"', from=3-1, to=3-2]
\end{tikzcd}\]
and that the following $2$-cells are both equal to $\id_{F_1(M)}$: 
\[\begin{tikzcd}
	{F_0(A)} && {F_0(B)} \\
	{F_0(A)} & {F_0(B)} & {F_0(B)} \\
	{F_0(A)} & {F_0(B)} & {F_0(B)} \\
	{F_0(A)} && {F_0(B)} \\
	{F_0(A)} && {F_0(B)}
	\arrow[""{name=0, anchor=center, inner sep=0}, "{F_1(M)}"{inner sep=.8ex}, "\shortmid"{marking}, from=1-1, to=1-3]
	\arrow[equals, from=1-1, to=2-1]
	\arrow[equals, from=1-3, to=2-3]
	\arrow[""{name=1, anchor=center, inner sep=0}, "{F_1(M)}"{inner sep=.8ex}, "\shortmid"{marking}, from=2-1, to=2-2]
	\arrow[equals, from=2-1, to=3-1]
	\arrow[""{name=2, anchor=center, inner sep=0}, "{U_{F_0(B)}}"{inner sep=.8ex}, "\shortmid"{marking}, from=2-2, to=2-3]
	\arrow[equals, from=2-2, to=3-2]
	\arrow[equals, from=2-3, to=3-3]
	\arrow[""{name=3, anchor=center, inner sep=0}, "{F_1(M)}"'{inner sep=.8ex}, "\shortmid"{marking}, from=3-1, to=3-2]
	\arrow[equals, from=3-1, to=4-1]
	\arrow[""{name=4, anchor=center, inner sep=0}, "{F_1(U_B)}"'{inner sep=.8ex}, "\shortmid"{marking}, from=3-2, to=3-3]
	\arrow[equals, from=3-3, to=4-3]
	\arrow[""{name=5, anchor=center, inner sep=0}, "{F_1(M \odot U_B)}"'{inner sep=.8ex}, "\shortmid"{marking}, from=4-1, to=4-3]
	\arrow[equals, from=4-1, to=5-1]
	\arrow[equals, from=4-3, to=5-3]
	\arrow[""{name=6, anchor=center, inner sep=0}, "{F_1(M)}"'{inner sep=.8ex}, "\shortmid"{marking}, from=5-1, to=5-3]
	\arrow["{\mathfrak r^{-1}_{F_1(M)}}", between={0.3}{0.8}, Rightarrow, from=0, to=2-2]
	\arrow["\id", between={0.3}{0.7}, Rightarrow, from=1, to=3]
	\arrow["{{F_U}_B}", between={0.3}{0.7}, Rightarrow, from=2, to=4]
	\arrow["{{F_\odot}_{M, U_B}}", between={0.1}{0.7}, Rightarrow, from=3-2, to=5]
	\arrow["{F_1(\mathfrak r_M)}", between={0.4}{0.8}, Rightarrow, from=5, to=6]
\end{tikzcd}~
\begin{tikzcd}
	{F_0(B)} && {F_0(C)} \\
	{F_0(B)} & {F_0(B)} & {F_0(C)} \\
	{F_0(B)} & {F_0(B)} & {F_0(C)} \\
	{F_0(B)} && {F_0(C)} \\
	{F_0(B)} && {F_0(C)}
	\arrow[""{name=0, anchor=center, inner sep=0}, "{F_1(N)}"{inner sep=.8ex}, "\shortmid"{marking}, from=1-1, to=1-3]
	\arrow[equals, from=1-1, to=2-1]
	\arrow[equals, from=1-3, to=2-3]
	\arrow[""{name=1, anchor=center, inner sep=0}, "{U_{F_0(B)}}"{inner sep=.8ex}, "\shortmid"{marking}, from=2-1, to=2-2]
	\arrow[equals, from=2-1, to=3-1]
	\arrow[""{name=2, anchor=center, inner sep=0}, "{F_1(N)}"{inner sep=.8ex}, "\shortmid"{marking}, from=2-2, to=2-3]
	\arrow[equals, from=2-2, to=3-2]
	\arrow[equals, from=2-3, to=3-3]
	\arrow[""{name=3, anchor=center, inner sep=0}, "{F_1(U_B)}"'{inner sep=.8ex}, "\shortmid"{marking}, from=3-1, to=3-2]
	\arrow[equals, from=3-1, to=4-1]
	\arrow[""{name=4, anchor=center, inner sep=0}, "{F_1(N)}"'{inner sep=.8ex}, "\shortmid"{marking}, from=3-2, to=3-3]
	\arrow[equals, from=3-3, to=4-3]
	\arrow[""{name=5, anchor=center, inner sep=0}, "{F_1(U_B \odot N)}"'{inner sep=.8ex}, "\shortmid"{marking}, from=4-1, to=4-3]
	\arrow[equals, from=4-1, to=5-1]
	\arrow[equals, from=4-3, to=5-3]
	\arrow[""{name=6, anchor=center, inner sep=0}, "{F_1(N)}"'{inner sep=.8ex}, "\shortmid"{marking}, from=5-1, to=5-3]
	\arrow["{\mathfrak l^{-1}_{F_1(N)}}", between={0.3}{0.8}, Rightarrow, from=0, to=2-2]
	\arrow["{{F_U}_B}", between={0.3}{0.7}, Rightarrow, from=1, to=3]
	\arrow["\id", between={0.3}{0.7}, Rightarrow, from=2, to=4]
	\arrow["{{F_\odot}_{U_B, N}}", between={0.1}{0.7}, Rightarrow, from=3-2, to=5]
	\arrow["{F_1(\mathfrak l_N)}", between={0.4}{0.8}, Rightarrow, from=5, to=6]
\end{tikzcd}\]

The commutativity of the first diagram amounts to the commutativity of the following diagram for all elements $m \in F_1(M)$, $n \in F_1(N)$ and $p \in F_1(P)$:
\[\begin{tikzcd}[column sep=-1.1em]
	& {{\iota_A}_!U_I{\iota_A}_!} \\
	{{\iota_A}_!U_I{\iota_B}_! \odot {\iota_B}_!U_I{\iota_D}_!} & {U_I} & {{\iota_A}_!U_I{\iota_C}_! \odot {\iota_C}_!U_I{\iota_D}_!} \\
	& {U_I \odot U_I} \\
	{U_I \odot (U_I\odot U_I)} && {(U_I \odot U_I)\odot U_I} \\
	{{\iota_A}_!U_I{\iota_B}_! \odot ({\iota_B}_!U_I{\iota_C}_! \odot {\iota_C}_!U_I{\iota_D}_!}) && {({\iota_A}_!U_I{\iota_B}_! \odot {\iota_B}_!U_I{\iota_C}_!) \odot {\iota_C}_!U_I{\iota_D}_!} \\
	{M\odot(N \odot P)} && {M\odot(N \odot P)}
	\arrow["{{i_2}_{A,B,D}}"', from=1-2, to=2-1]
	\arrow["{{i_2}_{A,C,D}}", from=1-2, to=2-3]
	\arrow["{\id_{{\iota_A}_!U_I{\iota_B}_!} \odot {i_2}_{B,C,D}}"{description, pos=0.4}, shift right=9, curve={height=20pt}, from=2-1, to=5-1]
	\arrow["\cct"{description}, from=2-2, to=1-2]
	\arrow["{ \mathfrak l_{U_I}^{-1}}"{description}, from=2-2, to=3-2]
	\arrow["{{i_2}_{A,B,C}\odot \id_{ {\iota_C}_!U_I{\iota_D}_!}}"{description, pos=0.4}, shift left=9, curve={height=-20pt}, from=2-3, to=5-3]
	\arrow["{\cct\odot\cct}"', from=3-2, to=2-1]
	\arrow["{\cct\odot\cct}", from=3-2, to=2-3]
	\arrow["{  \id_{U_I} \odot \mathfrak l_{U_I}^{-1}}"', from=3-2, to=4-1]
	\arrow["{  \mathfrak l_{U_I}^{-1} \odot \id_{U_I}}", from=3-2, to=4-3]
	\arrow["{\mathfrak a_{U_I,U_I,U_I}}"', from=4-1, to=4-3]
	\arrow["{\cct\odot(\cct\odot \cct)}"{description}, from=4-1, to=5-1]
	\arrow["{(\cct\odot\cct)\odot \cct}"{description}, from=4-3, to=5-3]
	\arrow["{m \odot (n \odot p)}"{description}, from=5-1, to=6-1]
	\arrow["{(m \odot n) \odot p}"{description}, from=5-3, to=6-3]
	\arrow["{\mathfrak a_{M,N,P}}"', from=6-1, to=6-3]
\end{tikzcd}\]

Similarly, for the leftmost equality (the proof of right one is symmetric), consider for all elements $m \in F_1(M)$ 
the following commutative diagram:
\[\begin{tikzcd}
	{U_I} &&& {{\iota_A}_!U_I{\iota_B}_!} \\
	& {U_I\odot U_I} & {{\iota_A}_!U_I{\iota_B}_! \odot U_I} \\
	{{\iota_A}_!U_I{\iota_B}_!} & {{\iota_A}_!U_I{\iota_B}_! \odot {\iota_B}_!U_I{\iota_B}_!} & {M \odot U_B} & M
	\arrow["\cct", from=1-1, to=1-4]
	\arrow["{\mathfrak r_{U_I}^{-1}}"{description}, from=1-1, to=2-2]
	\arrow["\cct"', from=1-1, to=3-1]
	\arrow["m", from=1-4, to=3-4]
	\arrow["{\cct \odot \id_{U_I}}", from=2-2, to=2-3]
	\arrow["{\cct \odot \cct}"', from=2-2, to=3-2]
	\arrow["{\mathfrak r_{{\iota_A}_!U_I{\iota_B}_!}}"{description}, from=2-3, to=1-4]
	\arrow["{m \odot U(\iota_B)}", from=2-3, to=3-3]
	\arrow["{{i_2}_{A,B,B}}"', from=3-1, to=3-2]
	\arrow["{m \odot \fact{U\iota_B}}"', from=3-2, to=3-3]
	\arrow["{\mathfrak r_M}"', from=3-3, to=3-4]
\end{tikzcd}\]

We have just shown the following theorem.

\begin{theorem}
\label{thm:embedding}
Let $\AAA$ be a module-like abelian framed bicategory. Then, there is a framed lax functor $F : \AAA \to \BimodUI$ such that $F_1$ is locally an equivalence of categories. 
\end{theorem}

By~\cite[Proposition 6.8]{shulman2007framed}, $F$ preserves cartesian $2$-cells up to isomorphism. In particular, restriction functors in $\AAA$ and $\BimodUI$ are compatible, as in Proposition~\ref{prop:pointwise_gabriel}.
This implies that horizontally, $F$ is fully faithful:

\begin{corollary}
    The framed lax functor $F$ is horizontally full and faithful.
\end{corollary}
\begin{proof}
   For any pair of vertical arrows $f,g$ and $1$-cells $M,N$ in $\AAA$, let us show that $F$ induces a bijection from the set $\AAA_{f,g}(M,N)$ of $2$-cells $M \cc{f}{g} N$ in $\AAA$ to the set $\BimodUI_{F_0(f),F_0(g)}(F_1M,F_1N)$, denoted by $T$. Let $c$ and $c'$ denote the cartesian $2$-cells $M \xrightarrow{\ct} f^*Ng^*$ and $F_1(M) \xrightarrow{\ct} F_0(f)^*F_1(N)F_0(g)^*$,
   respectively. Since $F$ preserves cartesian cells, $F_1(c) \cong c'$ and $F_1(c)$ is cartesian. Then, for all 2-cells $\alpha$ in $\AAA_{f,g}(M,N)$, we have:
   \begin{align*}
   F_1(\alpha) &= F_1(\fact{\alpha} \circ c)\\
               &= F_1(\fact{\alpha}) \circ F_1(c).
   \end{align*}
    By the uniqueness of factorization through cartesian maps, and since $F_1$ is locally fully faithful (as an equivalence of categories), the function $\alpha \mapsto F_1(\fact{\alpha}) \circ F_1(c)$ is a bijection.
\end{proof}

For $F$ to also preserve cocartesian $2$-cells, $F$ would have to be strong, i.e., one would have to show that the natural transformations $F_U$ and $F_\odot$ are isomorphic.

\begin{example}
    In $\Bimod$, $F_\odot$ and $F_U$ amount to the identity natural transformations $M \odot_B N \Rightarrow M \odot_B N$ and $U_A \Rightarrow U_A$, respectively. In particular they are natural isomorphisms.  
\end{example}

\begin{remark}
The development of Theorem~\ref{thm:embedding} in Sections~\ref{section:gt_pointwise} and~\ref{section:framed_embedding} relies on Gabriel's theorem and Theorem~\ref{th:preservation_triple}, which shows that certain adjoint triples preserve compact projective generators. However, it may be possible to use \FM{}'s theorem instead. Recall the proof sketch of Theorem~\ref{th:freyd-mitchell}. By Lemma~\ref{lemma:ind_pseudofunc}, the $\Pro$ construction is pseudofunctorial and preserves faithful functors, and the adjoint triple of Theorem~\ref{th:preservation_triple} preserves strong generators when the middle functor is faithful. Therefore, as in Proposition~\ref{prop:pointwise_gabriel}, \FM{}'s theorem can be applied pointwise:

\begin{proposition}\label{prop:pointwise_fm}
    Let $\AAA$ be a closed abelian framed bicategory, and let $f: A \to C$ and $g: B \to D$ be vertical arrows such that the local categories $\A(A,B)$ and $\A(C,D)$ are small and such that the restriction along $(f,g)$ is faithful. Then there exists a ring homomorphism $r: R \to R'$ and exact fully faithful functors $\A(A,B) \to {}_R\mmod$ and $\A(C,D) \to {}_{R'}\mmod$. This construction is pseudofunctorial.
\end{proposition}

It still remains to determine whether compatibility conditions for restriction functors can be obtained, as in Proposition~\ref{prop:pointwise_gabriel}, and more generally, whether \FM{}'s theorem can be used to construct a framed functor as in Theorem~\ref{thm:embedding}.
\end{remark}

\begin{remark}\label{remark:pare}
    This work bears some resemblance to the embedding theorem of~\cite{pare2011yoneda}. In that paper, Paré develops a Yoneda theory, including a Yoneda embedding, for a weaker version of double categories: virtual double categories in the sense of~\cite{cruttwell2009unified}, where composition might only be defined up to isomorphism. 
    In contrast, we devote considerable space to ensuring that our embedding is strictly functorial, because the target of our embedding is a concrete category of modules.
    Nevertheless, the Yoneda embedding is strong and horizontally, is dense in addition to being fully faithful.
\end{remark}


\section*{CRediT authorship contribution statement}

\look{Augustin Albert:} Conceptualization, Writing – original draft. \look{J\'er\'emy Dubut:} Conceptualization, Writing – review and editing, Supervision. \look {Eric Goubault:} Conceptualization, Writing – review and editing, Supervision.

\section*{Declaration of competing interest}

The authors declare that they have no known competing financial interests or personal relationships that
could have appeared to influence the work reported in this paper.

\section*{Acknowledgments}

The authors thank Eliot M\'edioni and Steve Oudot, for numerous interactions during the preparation of this work.

\bibliography{refs}

@misc{goubault2026homological,
  author       = {Goubault, Eric and Médioni, Eliot},
  title        = {Homological Algebra over Non-Unital Rings and Algebras, with Applications to $(\infty, 1)$-Categories},
  note         = {In preparation, preprint available upon request},
  year         = {2026}
}

@article{pare2011yoneda,
  title={Yoneda theory for double categories},
  author={Par{\'e}, Robert},
  journal={Theory and Applications of Categories},
  volume={25},
  number={17},
  pages={436--489},
  year={2011}
}

@book{etingof2015tensor,
  title={Tensor categories},
  author={Etingof, Pavel and Gelaki, Shlomo and Nikshych, Dmitri and Ostrik, Victor},
  volume={205},
  year={2015},
  publisher={American Mathematical Soc.}
}

@webpage{qchu2015generators,
title={Generators},
author={Yuan, Qiaochu},
year={2015},
url={https://qchu.wordpress.com/2015/05/17/generators/},
urldate={2025-09-21}
}

@book{borceux1994handbook,
  title={Handbook of categorical algebra: Basic category theory},
  author={Borceux, Francis},
  volume={1},
  year={1994},
  publisher={Cambridge University Press}
}

@book{rivano1972categories,
  title={Cat{\'e}gories tannakiennes},
  author={Rivano, N Saavedra},
  journal={Lecture Notes in Mathematics},
  volume={265},
  year={1972},
  adress={Berlin-New York},
  publisher={Springer-Verlag}
}

@article{cualuguareanu2010modules,
  title={Modules with Abelian endomorphism rings},
  author={C{\u{a}}lug{\u{a}}reanu, Grigore and Schultz, Phill},
  journal={Bulletin of the Australian Mathematical Society},
  volume={82},
  number={1},
  pages={99--112},
  year={2010},
  publisher={Cambridge University Press}
}

@book{borceux2004mal,
  title={Mal'cev, protomodular, homological and semi-abelian categories},
  author={Borceux, Francis and Bourn, Dominique},
  volume={566},
  year={2004},
  publisher={Springer Science \& Business Media}
}

@misc{van2006homology,
      title={Homology and homotopy in semi-abelian categories}, 
      author={Tim Van der Linden},
      year={2006},
      eprint={math/0607100},
      archivePrefix={arXiv},
      primaryClass={math.CT},
      url={https://arxiv.org/abs/math/0607100}, 
}

@article{ehresmann1963categories,
author = {Ehresmann, Charles},
journal = {Annales scientifiques de l'École Normale Supérieure},
language = {fre},
number = {4},
pages = {349-426},
publisher = {Elsevier},
title = {Catégories structurées},
url = {http://eudml.org/doc/81794},
volume = {80},
year = {1963},
}

@article{kelly1964maclane,
  title         = {On MacLane's conditions for coherence of natural associativities, commutativities, etc.},
  author        = {Kelly, Gregory Maxwell},
  year          = {1964},
  journal       = {Journal of Algebra},
  publisher     = {Elsevier},
  volume        = {1},
  number        = {4},
  pages         = {397--402},
}

@article{mitchell1964full,
  title         = {The full imbedding theorem},
  author        = {Mitchell, Barry},
  year          = {1964},
  journal       = {American Journal of Mathematics},
  publisher     = {Jstor},
  volume        = {86},
  number        = {3},
  pages         = {619--637},
}

@article{burroni1971t,
  title         = {$ T $-cat{\'e}gories (cat{\'e}gories dans un triple)},
  author        = {Burroni, Albert},
  year          = {1971},
  journal       = {Cahiers de topologie et g{\'e}om{\'e}trie diff{\'e}rentielle},
  volume        = {12},
  number        = {3},
  pages         = {215--321},
}

@article{gabriel1972unzerlegbare,
  title         = {Unzerlegbare darstellungen I},
  author        = {Gabriel, Peter},
  year          = {1972},
  journal       = {Manuscripta mathematica},
  publisher     = {Springer},
  volume        = {6},
  number        = {1},
  pages         = {71--103},
}

@article{wood1982abstract,
  title         = {Abstract pro arrows I},
  author        = {Wood, Richard J},
  year          = {1982},
  journal       = {Cahiers de topologie et g{\'e}om{\'e}trie diff{\'e}rentielle},
  volume        = {23},
  number        = {3},
  pages         = {279--290},
}

@article{wood1985proarrows,
  title         = {Proarrows ii},
  author        = {Wood, Richard J},
  year          = {1985},
  journal       = {Cahiers de topologie et g{\'e}om{\'e}trie diff{\'e}rentielle cat{\'e}goriques},
  volume        = {26},
  number        = {2},
  pages         = {135--168},
}

@book{weibel1994introduction,
language = {eng},
publisher = {Cambridge University Press},
series = {Cambridge studies in advanced mathematics ; no.38},
title = {An introduction to homological algebra },
address = {Cambridge},
booktitle = {An introduction to homological algebra},
isbn = {0521559871},
author = {Weibel, Charles A.},
year = {1995},
}

@book{selick1997introduction,
  title         = {Introduction to homotopy theory},
  author        = {Selick, Paul},
  year          = {1997},
  publisher     = {American Mathematical Soc.},
  volume        = {9},
}

@article{grandis1999limits,
  title         = {Limits in double categories},
  author        = {Grandis, Marco and Par{\'e}, Robert},
  year          = {1999},
  journal       = {Cahiers de topologie et g{\'e}om{\'e}trie diff{\'e}rentielle cat{\'e}goriques},
  volume        = {40},
  number        = {3},
  pages         = {162--220},
}

@article{fahrenberg2004directed,
  title         = {Directed homology},
  author        = {Fahrenberg, Ulrich},
  year          = {2004},
  journal       = {Electronic Notes in Theoretical Computer Science},
  publisher     = {Elsevier},
  volume        = {100},
  pages         = {111--125},
}

@mastersthesis{fluch2004kunneth,
  title         = {The Kunneth Formula in Abelian Categories},
  author        = {Fluch, Martin G},
  year          = {2004},
  school        = {University of Helsinki},
}

@article{grandis2004,
  title         = {Inequilogical spaces, directed homology and noncommutative geometry.},
  author        = {Grandis, Marco},
  year          = {2004},
  journal       = {Homology, Homotopy and Applications},
  publisher     = {International Press, Somerville},
  volume        = {6},
  number        = {1},
  pages         = {413--437},
  url           = {http://eudml.org/doc/51971},
  language      = {eng},
}

@book{kashiwara2006categories,
  title         = {Categories and sheaves},
  author        = {Kashiwara, Masaki and Schapira, Pierre},
  year          = {2006},
  publisher     = {Springer},
  place         = {Berlin},
}

@article{shulman2007framed,
  title         = {Framed bicategories and monoidal fibrations},
  author        = {Shulman, Michael A},
  year          = {2008},
  journal       = {Theory and Applications of Categories},
volume = {20},
number = {18},
pages = {650--738},
}

@article{ponto2014linearity,
  title={The linearity of traces in monoidal categories and bicategories},
  author={Ponto, Kate and Shulman, Michael},
   year          = {2016},
  journal       = {Theory and Applications of Categories},
volume = {31},
number = {23},
pages = {594--689},
}

@misc{dupont2008abelian,
  title         = {Abelian categories in dimension 2},
  author        = {Mathieu Dupont},
  year          = {2008},
  url           = {https://arxiv.org/abs/0809.1760},
  eprint        = {0809.1760},
  archiveprefix = {arXiv},
  primaryclass  = {math.CT},
}

@article{cruttwell2009unified,
author = {Cruttwell, G. and Shulman, Michael},
year = {2009},
month = {07},
pages = {},
title = {A unified framework for generalized multicategories},
volume = {24},
journal = {Theory and Applications of Categories [electronic only]},
doi = {10.70930/tac/mxocppsl}
}

@article{nakaoka2011comparison,
  title         = {Comparison of the definitions of abelian 2-categories},
  author        = {Nakaoka, Hiroyuki},
  year          = {2011},
  journal       = {Tsukuba Journal of Mathematics},
  publisher     = {Department of Mathematics, University of Tsukuba},
  volume        = {34},
  number        = {2},
  pages         = {173--182},
}

@book{dold2012lectures,
  title         = {Lectures on algebraic topology},
  author        = {Dold, Albrecht},
  year          = {2012},
  publisher     = {Springer Science \& Business Media},
  volume        = {200},
}

@inproceedings{dubut2015natural,
  title         = {Natural homology},
  author        = {Dubut, J{\'e}r{\'e}my and Goubault, Eric and Goubault-Larrecq, Jean},
  year          = {2015},
  booktitle     = {Automata, Languages, and Programming: 42nd International Colloquium, ICALP 2015, Kyoto, Japan, July 6-10, 2015, Proceedings, Part II 42},
  pages         = {171--183},
  organization  = {Springer},
}

@article{dubut2017directed,
  title={Directed homology theories and Eilenberg-Steenrod axioms},
  author={Dubut, J{\'e}r{\'e}my and Goubault, Eric and Goubault-Larrecq, Jean},
  journal={Applied Categorical Structures},
  volume={25},
  number={5},
  pages={775--807},
  year={2017},
  publisher={Springer}
}

@article{ziemianski2020spaces,
  title         = {Spaces of directed paths on pre-cubical sets II},
  author        = {Ziemia{\'n}ski, Krzysztof},
  year          = {2020},
  journal       = {Journal of Applied and Computational Topology},
  publisher     = {Springer},
  volume        = {4},
  number        = {1},
  pages         = {45--78},
}

@article{eilenberg1953acyclic,
  title={Acyclic models},
  author={Eilenberg, Samuel and MacLane, Saunders},
  journal={American journal of mathematics},
  volume={75},
  number={1},
  pages={189--199},
  year={1953},
  publisher={JSTOR}
}

@article{goubault2024semi,
  title         = {A semi-abelian approach to directed homology},
  author        = {Goubault, Eric},
  year          = {2024},
  journal       = {Journal of Applied and Computational Topology},
  publisher     = {Springer},
  volume        = {8},
  number        = {2},
  pages         = {271--299},
}

@book{kashiwara2013sheaves,
  title={Sheaves on Manifolds: With a Short History.{\guillemotleft}Les d{\'e}buts de la th{\'e}orie des faisceaux{\guillemotright}. By Christian Houzel},
  author={Kashiwara, Masaki and Schapira, Pierre},
  volume={292},
  year={2013},
  publisher={Springer Science \& Business Media}
}

@misc{goubault2025homotopy,
  title         = {Directed homotopy modules},
  author        = {Eric Goubault},
  year          = {2025},
  url           = {https://arxiv.org/abs/2504.13838},
  eprint        = {2504.13838},
  archiveprefix = {arXiv},
  primaryclass  = {math.AT},
}

@article{goubault2025directed,
  title         = {Directed homology and persistence modules},
  author        = {Goubault, Eric},
  year          = {2025},
  journal       = {Journal of Applied and Computational Topology},
  publisher     = {Springer},
  volume        = {9},
  number        = {1},
  pages         = {3},
}

@phdthesis{goubault1995geometrie,
    author = {Goubault, Eric},
    title = {G\'{e}om\'{e}trie du parall\'{e}lisme},
    school = {Ecole Polytechnique},
    year = {1995}
}

@article{patchkoria2006,
    author = {Patchkoria, Alex},
    title = {On exactness of long exact sequences of homology semimodules},
    journal = {Journal of Homotopy and Related structures},
    volume = {1},
    number = {1},
    pages = {229--243},
    year = {2006}
}

@phdthesis{dubut2017,
    author = {Dubut, J\'{e}r\'{e}my},
    title = {Directed homotopy and homology theories for geometric models of true concurrency},
    school = {Universit\'{e} Paris-Saclay},
    year = {2017}
}

@inproceedings{dubut2016directed,
    author = {Dubut, J\'{e}r\'{e}my and Goubault, Eric and Goubault-Larrecq Jean},
    title = {{The Directed Homotopy Hypothesis}},
    booktitle = {Proceedings of the 25th Annual EACSL Conference on Computer Science Logic (CSL’16)},
    series = {Leibniz International Proceedings in Informatics},
    pages = {1--16},
    volume = {62},
    year = {2016}
}

\appendix

\renewcommand{\thesection}{\Alph{section}}

\section{The acyclic model theorem}\label{section:amt}

We recall the proof of the well-known acyclic model theorem~\cite{eilenberg1953acyclic} for functors valued in chain complexes of $R$-modules. We follow the presentation of~\cite{dold2012lectures}.

\begin{definition}
A \look{category with models} $(\C, \mathcal{M})$ is a category $\C$ with a distinguished set of objects $\mathcal M$.
\end{definition}

\begin{definition}[\cite{dold2012lectures}]
A functor $F: \C \to \Ch({}_R\mmod)$ is \look{free} on models $\mathcal{M}$ if there exists a set $A$ and elements $M_\alpha \in \mathcal M$, $\alpha \in A$ such that for all $i \geq 0$ there are subsets $B_\alpha^i \subseteq F_i(M_\alpha)$, $\alpha$ in $A$, with the property that for every object $X$ in $\C$, the following set is a basis of $F_i(X)$:
$$\bigcup_{\alpha \in A} \bigcup_{f: M_\alpha \to X} F_i(f)(B_\alpha^i).$$
\end{definition}
\begin{definition}[\cite{dold2012lectures}]
A functor $F: \C \to \Ch({}_R\mmod)$ is \look{acyclic} on models $\mathcal{M}$ if for every object $M \in \mathcal M$ and every $i \geq 1$, $H_i(F(M)) \cong 0$; i.e., the homology groups are trivial except possibly in degree $0$.
\end{definition}
The acyclic model theorem is proved in~\cite{dold2012lectures} for abelian groups, and its proof generalizes to modules over commutative rings.
\begin{theorem}[{\cite{dold2012lectures}}]
Let $F, G: \C \to \Ch({}_R\mmod)$ be two functors. Suppose $F$ is free on models $\mathcal{M}$ and $G$ is acyclic on $\mathcal{M}$. Then, for every natural transformation $\overline{\tau}_0: H_0(F) \rightarrow H_0(G)$, there exists a natural transformation $\tau: F \rightarrow G$ that induces it on homology.
Furthermore, given two natural transformations $F \rightarrow G$ inducing the same map $H_0(F) \rightarrow H_0(G)$, there exists a natural chain homotopy between them.
\end{theorem}
\begin{proof}[Proof idea]
The construction proceeds inductively. For the first part, given $\overline{\tau}_0: H_0(F) \rightarrow H_0(G)$, we define $\tau_0: F_0 \to G_0$. Since $F$ is free, it suffices to define $\tau_0$ on basis elements $m \in B_\alpha^0$. For each such $m$, choose $\tau_0(m)$ to be any element of $G_0(M_\alpha)$ whose homology class equals $\overline{\tau}_0(\overline{m})$, where $\overline{m}$ is the homology class of $m$. Extend $\tau_0$ linearly and naturally.
Assume inductively that $\tau_k: F_k \to G_k$ has been defined for $k < n$ ($n > 0$) and that $\partial \tau_{k} = \tau_{k-1} \partial$ for all relevant $k$. To define $\tau_n$ on a basis element $m \in B_\alpha^n$, observe that
\[
\partial \tau_{n-1}(\partial m) = \tau_{n-2}(\partial^2 m) = 0.
\]
Since $G$ is acyclic, there exists an element $e \in G_n(M_\alpha)$ with $\partial e = \tau_{n-1}(\partial m)$. Set $\tau_n(m) = e$. This ensures $\partial \tau_n(m) = \tau_{n-1}(\partial m)$. Extending linearly and naturally yields $\tau_n: F_n \to G_n$ satisfying the chain map condition.
The second part follows a similar inductive construction of a natural homotopy.
\end{proof}
As an immediate consequence, if $\tau, \tau': F \rightarrow G$ are two natural transformations constructed by the acyclic models theorem, then they induce the same map $\overline \tau$ in homology. Moreover, if $F$ and $G$ are both free and acyclic on some models $\mathcal{M}$, then a natural isomorphism $H_0(F) \cong H_0(G)$ induces a natural chain homotopy equivalence $F \simeq G$.
\section{Algebraic theorems}\label{annex:algebraic}
We recall the main algebraic theorems used in this document.
\begin{theorem}[Algebraic \MV{} theorem~{\cite[Lemma 4.1.8]{selick1997introduction}}]
\[
\begin{tikzcd}[column sep=scriptsize]
\cdots \arrow[r] & A_n \arrow[r, "i"] \arrow[d, "\alpha"] & B_n \arrow[r, "j"] \arrow[d, "\beta"] & C_n \arrow[r, "\partial"] \arrow[d, "\gamma"] & A_{n-1} \arrow[r, "i"] \arrow[d, "\alpha"] & B_{n-1} \arrow[r, "j"] \arrow[d, "\beta"] & C_{n-1} \arrow[r] \arrow[d, "\gamma"] & \cdots \\
\cdots \arrow[r] & A'_n \arrow[r, "i'"] & B'_n \arrow[r, "j'"] & C'_n \arrow[r, "\partial"] & A'_{n-1} \arrow[r, "i'"] & B'_{n-1} \arrow[r, "j'"] & C'_{n-1} \arrow[r] & \cdots
\end{tikzcd}
\]
Let the above diagram be a commutative diagram with exact rows and suppose that $\gamma : C_n \to C'_n$ is an isomorphism for all $n$. Then there is an induced long exact sequence
\[
\cdots \to A_n \to B_n \oplus A'_n \to B'_n \xrightarrow{\Delta} A_{n-1} \to B_{n-1} \oplus A'_{n-1} \to B'_{n-1} \to \cdots .
\]
\end{theorem}
\begin{theorem}[{Algebraic Künneth theorem for right exact functors~\cite[Theorem 3.19]{fluch2004kunneth}}]\label{th:algebraic_kunneth}
Let $\mathcal A_1$, $\mathcal A_2$ and $\mathcal A$ be abelian categories and $t: \mathcal A_1 \times \mathcal A_2 \to \mathcal A$ an additive right exact functor, covariant in the first and contravariant in the second variable. Assume that $\mathcal A_1$ has enough projectives, $\mathcal A_2$ has enough injectives and $A$ has coproducts. Let $X_1$ be a chain complex in $A_1$ and $X_2$ a chain complex in $A_2$. Let $t_1$ denote the first left derived functor of $t$. 
Consider the following condition, denoted C1:
The homomorphisms
    \begin{align*}
    \alpha_1: t(B(X_1), H(X_2)) \to Ht(B(X_1), X_2)\\
    \alpha_2:  t(Z(X_1), H(X_2)) \to Ht(Z(X_1), X_2)
    \end{align*}
are isomorphisms, and we have
$t_1(B(X_1), X_2) = 0 = t_1(Z(X_1), H(X_2))$.
Then, if condition C1 holds, there exists a homomorphism $\beta$ of degree $-1$ such that the following sequence is exact.
$$0 \to t(H(X_1), H(X_2)) \xrightarrow{\alpha} Ht(X_1, X_2) \xrightarrow{\beta} t_1(H(X_1), H(X_2)) \to 0$$
\end{theorem}

\begin{proposition}[{\cite[Proposition 4.1]{fluch2004kunneth}}]\label{prop:sufficient_kunneth}
    Let $t$ be right exact and assume that $\mathcal A_1$ has projectives and $\mathcal A_2$ has injectives. Let $X_1$ be a chain complex in $\mathcal A_1$ and $X_2$ a chain complex in $\mathcal A_2$ such that
    \begin{align*}
        t_1 (B(X_1 ), B(X_2 )) = 0 = t_1 (B(X_1 ), H(X_2 ))\\
t_1 (Z(X_1 ), B(X_2 )) = 0 = t_1 (Z(X_1 ), H(X_2 ))
    \end{align*}
Then condition C1 is satisfied.
\end{proposition}

\end{document}